\documentclass[12pt]{amsart}
\usepackage{graphicx, amssymb, amsmath, amsfonts, amsthm, mathtools, tcolorbox,hyperref,cleveref,mathrsfs,aliascnt,tikz}
\usepackage[backend=bibtex]{biblatex}
\definecolor{lg}{RGB}{146, 224, 180}
\definecolor{dg}{RGB}{45, 194, 110}
\definecolor{db}{RGB}{36, 56, 212}
\definecolor{lb}{RGB}{149, 189, 237}
\definecolor{do}{RGB}{250, 116, 5}
\definecolor{lo}{RGB}{245, 176, 140}
\definecolor{dr}{RGB}{247, 47, 47}
\definecolor{lr}{RGB}{245, 161, 161}
\definecolor{dp}{RGB}{114, 15, 219}
\definecolor{lp}{RGB}{209, 170, 250}
\definecolor{dy}{RGB}{186, 160, 30}
\definecolor{ly}{RGB}{232, 207, 84}

\newtheoremstyle{defex}
  {}
  {}
  {\text}
  {}
  {\bfseries}
  {}
  {}
  {}

\usepackage{amsthm}
\makeatletter
\def\th@plain{%
  \thm@notefont{}
  \itshape 
}
\def\th@note{%
  \thm@notefont{}
  \itshape 
}
\makeatother

\newtheorem{theorem}{Theorem}[section]

\newaliascnt{lemma}{theorem}
\newtheorem{lemma}[lemma]{Lemma}
\aliascntresetthe{lemma}

\newaliascnt{corollary}{theorem}
\newtheorem{corollary}[corollary]{Corollary}
\aliascntresetthe{corollary}

\newaliascnt{proposition}{theorem}
\newtheorem{proposition}[proposition]{Proposition}
\aliascntresetthe{proposition}

\newaliascnt{conjecture}{theorem}
\newtheorem{conjecture}[conjecture]{Conjecture}
\aliascntresetthe{conjecture}

\newaliascnt{question}{theorem}
\newtheorem{question}[question]{Question}
\aliascntresetthe{question}

\theoremstyle{definition}

\newaliascnt{definition}{theorem}
\newtheorem{definition}[definition]{Definition}
\aliascntresetthe{definition}

\newaliascnt{example}{theorem}
\newtheorem{example}[example]{Example}
\aliascntresetthe{example}

\newaliascnt{remark}{theorem}
\newtheorem{remark}[remark]{Remark}
\aliascntresetthe{remark}

\newaliascnt{claim}{theorem}
\newtheorem{claim}[claim]{Claim}
\aliascntresetthe{claim}

\crefname{theorem}{theorem}{theorems}
\Crefname{theorem}{Theorem}{Theorems}

\crefname{lemma}{lemma}{lemmas}
\Crefname{lemma}{Lemma}{Lemmas}

\crefname{corollary}{corollary}{corollaries}
\Crefname{corollary}{Corollary}{Corollaries}

\crefname{proposition}{proposition}{propositions}
\Crefname{proposition}{Proposition}{Propositions}

\crefname{definition}{definition}{definitions}
\Crefname{definition}{Definition}{Definitions}

\crefname{example}{example}{examples}
\Crefname{example}{Example}{Examples}

\crefname{conjecture}{conjecture}{conjectures}
\Crefname{conjecture}{Conjecture}{Conjectures}

\crefname{question}{question}{questions}
\Crefname{question}{Question}{Questions}

\crefname{remark}{remark}{remarks}
\Crefname{remark}{Remark}{Remarks}

\crefname{claim}{claim}{claims}
\Crefname{claim}{Claim}{Claims}

\newcommand{\al}{\alpha}
\newcommand{\be}{\beta}

\newcommand{\dd}{\delta}

\newcommand{\si}{\sigma}

\newcommand{\Z}{\mathbb{Z}}
\newcommand{\N}{\mathbb{N}}
\newcommand{\Q}{\mathbb{Q}}

\newcommand{\mc}{\mathcal}

\newcommand{\fs}[2]{\mathrm{FS}({#1},{#2})}
\newcommand{\FS}{\mathrm{FS}(X,Y)}
\newcommand{\diam}{\mathrm{diam}}
\newcommand{\starp}{\prescript{+}{}{\mathrm{Star}}}
\newcommand{\st}{\mathrm{Star}}

\newcommand{\back}{\backslash}
\newcommand{\blue}[1]{\textcolor{blue}{\emph{#1}}}
\newcommand{\sq}{\mathbin{\square}}
\newcommand{\sm}{\setminus}
\usepackage{todonotes,enumitem}

\usepackage[margin=1in]{geometry}

\usepackage{fancyhdr}
\title{Connectivity and Diameter of Friends and Strangers Graphs Using Minimum Degree}
\author{Kai Mawhinney}

\begin{document}

\begin{abstract}
    Given graphs $X,Y$ on $n$ vertices, Defant and Kravitz defined the friends and strangers graph, denoted $\FS$. The vertices of $\FS$ are bijections $\si:V(X)\to V(Y)$, and two bijections $\si$ and $\si'$ are adjacent in $\FS$ if they differ by a transposition $(ab)$ such that $a$ and $b$ are adjacent in $X$ and $\si(a)$ and $\si(b)$ are adjacent in $Y$. 
Our main result resolves a conjecture of Bangachev that if $X,Y$ are connected graphs and $2\min\{\dd(X),\dd(Y)\}+3\max\{\dd(X),\dd(Y)\}\ge 3n$, then $\FS$ is connected. 
The main step in the proof is showing that there exists an absolute constant $C$ such that for any adjacent vertices $a,b\in V(X)$ and any vertex $\si\in V(\FS)$, $a$ and $b$ can swap positions in at most $C$ moves. 
Furthermore, we prove that if $X,Y$ are connected graphs satisfying the above minimum degree condition, then the diameter of $\FS$ is $O(n^2/\min\{\dd(X),\dd(Y)\})$, and that this is sharp as a bound on the maximum diameter of $\FS$ in this regime. This work improves the diameter results of Akella and Li, and resolves many of their questions. Moreover, we show that even for disconnected graphs $X$ and $Y$, the same minimum degree condition implies that the diameter of any connected component of $\FS$ is $O(n^2)$. 
\end{abstract}

\maketitle 

\vspace{-0.3in}

\section{Introduction}

Friends and strangers graphs were defined as follows in 2021 by Colin Defant and Noah Kravitz \cite{Defant}.
\begin{definition}[{\cite{Defant}}]\label{def:fs}
    Given simple graphs $X,Y$ on $n$ vertices, the \blue{friends and strangers graph} $\FS$ is the graph whose vertices are bijections $\si:V(X)\to V(Y)$. Two bijections $\si,\si'$ are adjacent in $\FS$ if there exist $a,b\in V(X)$ that satisfy the following conditions:
    \begin{enumerate}
        \item $\si = \si'\circ (ab)$;
        \item $\{a,b\}\in E(X)$;
        \item $\{\si(a),\si(b)\}\in E(Y)$.
    \end{enumerate}
\end{definition}
\subsection{An Intuitive Analogy}\label{sec:analogy}

The reader may find the following intuitive analogy useful. Suppose we have $n$ distinct people and $n$ distinct houses. Any two people can either be friends or not, and any two houses may either have a road between them or not. Thus:
\begin{align*}
    E(X) &\text{ encodes }\textbf{friendships} \text{ between }\textbf{people}.\\
    E(Y) &\text{ encodes }\textbf{roads} \text{ between }\textbf{houses}.
\end{align*}
Since each vertex of $\FS$ is a bijection $\si:V(X)\to V(Y)$, it corresponds to a configuration where each house is occupied by exactly one person. Two configurations are adjacent in $\FS$ if and only if one can be obtained from the other by swapping two people who are friends and who currently reside in houses connected by a road. When this interchanges $a,b\in V(X)$, we say we are \blue{swapping} $a$ and $b$, or \blue{performing the transposition} $(ab)$. We also use these terms if there exists a path in $\FS$ from $\si$ to $\si\circ (ab)$. Usually this refers to a single edge. When this swap is instead obtained via a path of length at most $k$, we will write the swap as $(ab)^{k}$, and say $(ab)$ can be performed in (at most) $k$ \blue{moves}. See that \Cref{def:fs} conditions (2) and (3) are exactly the requirements that the people are friends and that their houses have a road between them, respectively. We give an example to make this analogy clear.
\begin{example}
    Let $X,Y$ be graphs on $n=4$ vertices. Fix vertex $\si\in V(\FS)$ and recall $\si:V(X)\to V(Y)$ is a bijection. Let
    \begin{align*}
        E(X) &= \{\{a,b\},\{a,d\}\}\\
        E(Y) &= \{\{\si(a),\si(b)\},\{\si(b),\si(c)\},\{\si(c),\si(d)\}\}.
    \end{align*}
    \Cref{fig:people and houses} illustrates the configuration corresponding to $\si$. Because people $a$ and $b$ are friends and houses $\si(a),\si(b)$ have a road between them, we may swap $a$ and $b$. Formally, this means that the permutations $\si\circ (ab)$ and $\si$ are adjacent in $\FS$. Note that this swap requires a condition in both $X$ and $Y$. For example, friends $a$ and $d$ cannot be swapped in the current configuration since houses $\si(a)$ and $\si(d)$ do not have a road between them. Similarly, since people $b$ and $c$ are not friends they cannot be swapped, even though there is a road between houses $\si(b)$ and $\si(c)$. When using this analogy, the houses and roads are fixed but the people and friendships move. 
    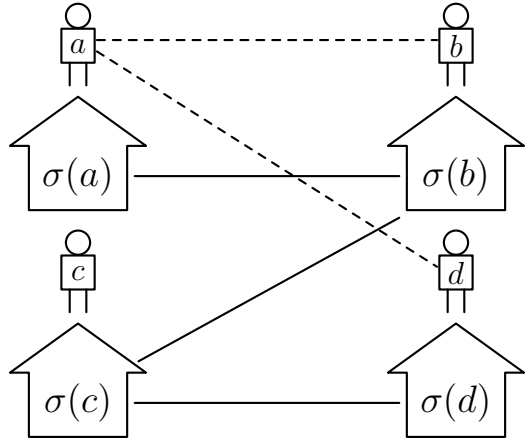
\begin{figure}[h!]
    \centering
    \begin{tikzpicture}[line cap=round,line join=round,scale=1]

\newcommand{\house}[3]{
  \begin{scope}[shift={(#1,#2)}]
    \draw[thick] (-0.65,0) -- (-0.65,0.85) -- (-0.9,0.85)
      -- (0,1.5) -- (0.9,0.85) -- (0.65,0.85)
      -- (0.65,0) -- cycle;
    \node at (0,0.45) {\large $#3$};
  \end{scope}
}

\newcommand{\person}[3]{
  \begin{scope}[shift={(#1,#2)}]
    \draw[thick] (0,0.72) circle (0.16);
    \draw[thick] (-0.22,0.1) rectangle (0.22,0.55);
    \node at (0,0.32) {$#3$};
    \draw[thick] (-0.11,0.1) -- (-0.11,-0.2);
    \draw[thick] (0.11,0.1) -- (0.11,-0.2);
  \end{scope}
}

\draw[thick] (0.75,3.45) -- (4.25,3.45);
\draw[thick] (0.75,0.45) -- (4.25,0.45);
\draw[thick] (0.8,1) -- (4.25,2.9);

\house{0}{3}{\si(a)}
\house{5}{3}{\si(b)}
\house{0}{0}{\si(c)}
\house{5}{0}{\si(d)}

\person{0}{4.85}{a}
\person{5}{4.85}{b}
\person{0}{1.85}{c}
\person{5}{1.85}{d}

\draw[dashed,thick] (0.25,5.25) -- (4.75,5.25);
\draw[dashed,thick] (0.25,5.1) -- (4.75,2.25);

\end{tikzpicture}
    \caption{A vertex of $\FS$ viewed with the `people and houses' analogy.}
    \label{fig:people and houses}
\end{figure}
\end{example}

\subsection{Motivation}

Friends and strangers graphs generalize token swapping problems and sliding puzzles; the former are of interest in computer science and are the special case $X = K_n$ \cite{Wilson,Token} and the latter is the special case where $X=\st_n$ \cite{ChaoYang}. One famous sliding puzzle is the `15-puzzle.' In this game, the player is given a random configuration of 15 wooden squares labelled 1 through 15, arranged in a $4\times 4$ grid with one square left blank. The goal is to slide the tiles so they read $1,2,\dots,15$ from left to right, top to bottom. However, a tile can only be slid into the blank space in the grid. Thus, the missing 16th square must be used in every swap. The configurations of this sliding puzzle are represented by $\fs{\mathrm{Star}_{16}}{\mathrm{Grid}_{4\times 4}}$. Here $X=\mathrm{Star}_{16}$ indicates that the 16th square must be used in every swap, and $Y = \mathrm{Grid}_{4\times 4}$ reflects the constraints of the board. In general, sliding puzzles use $X = \mathrm{Star}$.

Friends and strangers graphs are closely related to Cayley graphs. Let $T$ be the set of transpositions in $S_n$. A friends and strangers graph is a subgraph of $\mathrm{Cay}_{T}(S_n)$. First note that $\fs{K_n}{K_n} = \mathrm{Cay}_{T}(S_n)$. Let $T(X)$ denote the set of transpositions corresponding to edges of $X$, so $(ab)\in T(X)$ if and only if $\{a,b\}\in E(X)$. Then $\fs{X}{K_n} \cong \mathrm{Cay}_{T(X),\text{right}}(S_n)$. Similarly, $\fs{K_n}{Y} \cong \mathrm{Cay}_{T(Y),\text{left}}(S_n)$. Moreover, $\FS$ is precisely their intersection:
\begin{align*}
    \FS &\cong \mathrm{Cay}_{T(X),\text{right}}(S_n)\cap \mathrm{Cay}_{T(Y),\text{left}}(S_n).
\end{align*}
In this way, any friends and strangers graph can be viewed as an intersection of Cayley graphs: one given by right multiplication using transpositions from $X$, the other given by left multiplication using transpositions from $Y$. Unfortunately, $\FS$ is not necessarily a Cayley graph. For this reason, friends and strangers graphs can be a difficult object to study.
\subsection{Existing Work}

Much of the previous work in this field examines properties of $\FS$ such as connectedness, diameter, and number of components. The results largely fall into three areas:
\begin{enumerate}
    \item Fixing $X$ to be a specific graph and varying $Y$.
    \item Using conditions on the minimum degrees of $X$ and $Y$.
    \item Probabilistic results when $X$ and $Y$ are Erd\H{o}s-R\'enyi random graphs.
\end{enumerate}
Significant work has been done in area (1), including results when $X$ is a cycle, path, star graph, or bipartite graph \cite{Defant,BipartiteJeong,Stanley,Bipartite,Wilson,Xinghui,Honglin}. In area (2), the most notable results are due to Bangachev \cite{Bangachev}, including sufficient conditions for connectivity of $\FS$ in terms of the \blue{minimum degrees} of $X$ and $Y$, denoted $\dd(X)$ and $\dd(Y)$. Following Bangachev's paper, Akella and Li produced some diameter bounds based on minimum degree \cite{FSDiameter}. Area (3) has been studied by Alon, Defant, and Kravitz \cite{NogaProb}, and the diameter has also been studied by Akella and Li \cite{FSDiameter}.

Outside of these areas, Jeong showed that a disconnected friends and strangers graph may have a component with diameter $e^{\Omega(n)}$ \cite{Jeong}. It is currently unknown if connected friends and strangers graphs can also have exponential diameter. Lastly, some authors have studied the $k$-connectedness of various friends and strangers graphs \cite{KConn,Xinghui}.

This paper will focus on (2). We improve upon a result of Bangachev by finding a weaker sufficient condition on $\dd(X)$ and $\dd(Y)$ for connectivity. We state Bangachev's result here, and include a diagram for reference; see \Cref{fig:connectivity diagram}. Note that all graphs $X$ and $Y$ plotted in \Cref{fig:connectivity diagram} are connected. If either $X$ or $Y$ is disconnected, then $\FS$ is disconnected as well, so there is no need to consider this case.
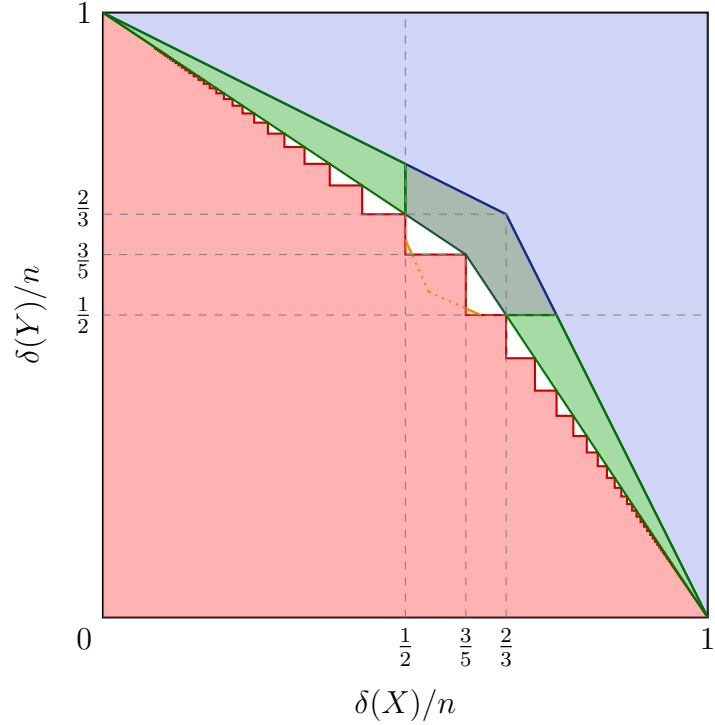
\begin{figure}
\centering
\begin{tikzpicture}[scale=8]

\draw[thick] (0,0) rectangle (1,1);


\def\RedPath{(1,0)}

\foreach \k in {34,33,...,5}{
  \pgfmathtruncatemacro{\kp}{\k+1}

  \xdef\RedPath{\RedPath
    -- ({(\k-2)/\k},{3/\kp})
    -- ({(\k-2)/\k},{3/\k})}
}

\foreach \k in {5,...,34}{
  \pgfmathtruncatemacro{\kp}{\k+1}

  \xdef\RedPath{\RedPath
    -- ({3/\kp},{(\k-2)/\k})
    -- ({3/\kp},{(\kp-2)/\kp})}
}

\fill[red!55,opacity=.55]
  (0,0) -- \RedPath -- (0,1) -- cycle;

\draw[ thick,red!80!black]
  \RedPath -- (0,1);

\fill[dg!30!black,opacity=.35]
  (1/2,2/3) -- (1/2,3/4) -- (2/3,2/3) -- (3/4,1/2) -- (2/3,1/2)
  -- (3/5,3/5)
  -- cycle;

\draw[ thick,dg!45!black]
  (1/2,2/3) -- (1/2,3/4) -- (2/3,2/3) -- (3/4,1/2) -- (2/3,1/2)
  -- (3/5,3/5)
  -- cycle;

\fill[db!60,opacity=.35]
  (0,1) -- (1,1) -- (1,0)
  -- (2/3,2/3)
  -- cycle;

\draw[ thick,db!70!black]
  (0,1) -- (2/3,2/3) -- (1,0);

\fill[green!60!black,opacity=.35]
  (0,1) -- (1/2,3/4) -- (1/2,2/3)
  -- cycle;

\fill[green!60!black,opacity=.35]
  (1,0) -- (3/4,1/2) -- (2/3,1/2)
  -- cycle;

\draw[ thick,green!45!black]
  (0,1) -- (1/2,3/4) -- (1/2,2/3)
  -- cycle;

\draw[ thick,green!45!black]
  (1,0) -- (3/4,1/2) -- (2/3,1/2)
  -- cycle;

\draw[thick,orange]
(1/2,5/8) -- (23/45,3/5);
\draw[thick,orange]
(5/8,1/2) -- (3/5,23/45);
\draw[dotted,thick,orange]
(7/13,7/13) -- (23/45,3/5);
\draw[dotted,thick,orange]
(7/13,7/13) -- (3/5,23/45);
\fill[orange, opacity = .35]
(1/2,5/8) -- (23/45,3/5) -- (1/2,3/5) -- cycle;
\fill[orange, opacity = .35]
(5/8,1/2) -- (3/5,23/45) -- (3/5,1/2) -- cycle;

\node[below] at (0.5,-0.1) {$\delta(X)/n$};
\node[rotate=90] at (-0.12,0.5) {$\delta(Y)/n$};

\node[below left] at (0,0) {$0$};
\node[below] at (1,0) {$1$};
\node[left] at (0,1) {$1$};

\node[below] at (3/5,0) {$\frac35$};
\node[left,font=\small] at (0,3/5-.01) {$\frac35$};

\node[below] at (2/3,0) {$\frac23$};
\node[left,font=\small] at (0,2/3+.01) {$\frac23$};

\node[below] at (1/2,0) {$\frac12$};
\node[left,font=\small] at (0,1/2) {$\frac12$};

\draw[dashed,gray] (3/5,0) -- (3/5,3/5) -- (0,3/5);
\draw[dashed,gray] (2/3,0) -- (2/3,2/3) -- (0,2/3);
\draw[dashed,gray] (1/2,0) -- (1/2,1);
\draw[dashed,gray] (0,1/2) -- (1,1/2);

\end{tikzpicture}
    \caption{Connectivity of $\FS$ for connected $X,Y$ in terms of the minimum degrees $\dd(X)$ and $\dd(Y)$. The graph $\FS$ is connected in the \textcolor{dg}{green}, \textcolor{dg!60!black}{dark green}, and \textcolor{blue}{blue} regions, however, $\FS$ may be disconnected in the \textcolor{red}{red} and \textcolor{orange}{orange} regions. It is unknown if $\FS$ in the white region are always connected.}
    \label{fig:connectivity diagram}
\end{figure}

The regions in \Cref{fig:connectivity diagram} are colored as follows:
\begin{enumerate}
     \item \textcolor{blue}{Blue}: $\min\{\dd(X),\dd(Y)\}+2\max\{\dd(X),\dd(Y)\}\ge 2n$;
    \item \textcolor{dg}{Green}: $2\min\{\dd(X),\dd(Y)\}+3\max\{\dd(X),\dd(Y)\}\ge 3n$;
    \item \textcolor{dg!50!black}{Dark Green}: $\min\{\dd(X),\dd(Y)\}>\frac{n}{2}$ and $2\min\{\dd(X),\dd(Y)\}+3\max\{\dd(X),\dd(Y)\}\ge~3n$;
    \item \textcolor{red}{Red}: Below the points $(3/k,(k-2)/k)$ or $((k-2)/k),3/k)$ for $k\ge 5$ an integer.
    \item \textcolor{orange}{Orange}: Below the lines from $(1/2,5/8)$ to $(7/13,7/13)$ and $(5/8,1/2)$ to $(7/13,7/13)$.
\end{enumerate}

If a minimum degree pair $(\dd(X)/n,\dd(Y)/n) = (\al,\beta)$ implies $\FS$ is connected, then it is clear that all pairs $(\al',\be')$ for $\al'\ge \al$ and $\be'\ge \be$ also imply $\FS$ is connected, as this is a strictly stronger condition. Similarly, if $(\dd(X)/n,\dd(Y)/n) =(\al,\be)$ is insufficient to guarantee that $\FS$ is connected, then any pair $(\al',\be')$ with $\al'\le \al$ and $\be'\le \be$ is also insufficient, as this is a strictly weaker condition. Thus, we use the \textcolor{dg}{green} region to refer to all points above the line $2\min + 3\max = 3n$, which is the union of the \textcolor{dg}{green}, \textcolor{dg!50!black}{dark green}, and \textcolor{blue}{blue} regions in \Cref{fig:connectivity diagram}. We continue with a summary of previous results.

Bangachev proved that all $\FS$ in the \textcolor{blue}{blue} region are connected \cite{Bangachev}.

\begin{theorem}[{\cite[Theorem 1.5]{Bangachev}}]
    Suppose $X,Y$ are connected graphs on $n$ vertices satisfying 
    \[\min\{\dd(X),\dd(Y)\}+2\max\{\dd(X),\dd(Y)\}\ge 2n.\]
    Then $\FS$ is connected.
\end{theorem}

Bangachev also proved that all $\FS$ in the \textcolor{dg!50!black}{dark green} region are connected \cite{Bangachev}.

\begin{theorem}[{\cite[Theorem 1.4]{Bangachev}}]\label{thm:kiril}
    Suppose $X,Y$ are graphs on $n\ge 6$ vertices satisfying
    \begin{enumerate}
        \item $\min\{\dd(X),\dd(Y)\}>\frac{n}{2}$.
        \item $2\min\{\dd(X),\dd(Y)\}+3\max\{\dd(X),\dd(Y)\}\ge 3n$.
    \end{enumerate}
    Then $\FS$ is connected.
\end{theorem}
The case $n<6$ holds by casework on small $n$. Condition (1) and (2) outline the \textcolor{dg!50!black}{dark green} region. The proof of \Cref{thm:kiril} uses that condition (1) is sufficient to ensure that $X$ and $Y$ are connected, and furthermore than $\diam(X),\diam(Y)\le 2$. The proof in this paper will instead have to use the assumption that $X$ and $Y$ are connected in our argument, as it does not follow from condition (2) alone. Bangachev also conjectured that $\FS$ are connected in the entirety of the \textcolor{dg}{green} region, which is equivalent to removing condition (1) from \Cref{thm:kiril}.
\begin{conjecture}[{\cite[Conjecture 8.1]{Bangachev}}]\label{conj:kiril}
    Suppose $X,Y$ are connected graphs on $n$ vertices satisfying
\[2\min\{\dd(X),\dd(Y)\}+3\max\{\dd(X),\dd(Y)\}\ge 3n.\]
    Then $\FS$ is connected.
\end{conjecture}

Bangachev also found a nice family of disconnected friends and strangers graphs where $X$ and $Y$ have large minimum degrees \cite{Bangachev}; these examples form the \textcolor{red}{red} `staircase' region in \Cref{fig:connectivity diagram}.

\begin{proposition}[{\cite[Proposition 1.6] {Bangachev}}]\label{prop:bangachev counterexamples}
    Suppose that $n\ge k\ge 5$ are integers. Then there exist connected graphs $X,Y$ on $n$ vertices such that $\dd(X)\ge \frac3kn-4$, $\dd(Y)\ge \frac{k-2}{k}n-3$, and $\FS$ is disconnected.
\end{proposition}
The \textcolor{red}{red} region is open, meaning that none of the examples from this proposition lie on the boundary itself. Moreover, the examples from this proposition yield only the `corners' of the staircase; by removing edges from $X$ and $Y$, we can decrease the minimum degrees $\dd(X)$ and $\dd(Y)$ as we wish while keeping $\FS$ disconnected.

It is a natural question to ask for a bound on the diameter in the regions where $\FS$ is known to be connected, such as the \textcolor{blue}{blue} and \textcolor{dg!50!black}{dark green} regions. Using the proofs of Bangachev, results on the diameter of $\FS$ were obtained by Akella and Li \cite{FSDiameter}. Their main result is that the diameter of $\FS$ in the \textcolor{blue}{blue} region of \Cref{fig:connectivity diagram} is $O(n^5)$.
\begin{theorem}[{\cite[Theorem 1.3]{FSDiameter}}]\label{thm:ALIn5}
    Let $X$ and $Y$ be connected graphs on $n$ vertices such that 
    \[\min\{\dd(X),\dd(Y)\}+2\max\{\dd(X),\dd(Y)\}\ge 2n.\]
    Then $\FS$ is connected and has diameter $O(n^5)$.
\end{theorem}
Akella and Li ask the following question, which furthers \Cref{conj:kiril}.
\begin{question}[{\cite[Question 6.5]{FSDiameter}}]\label{q:ALI}
    Is it true that when 
    \[2\min\{\dd(X),\dd(Y)\}+3\max\{\dd(X),\dd(Y)\}\ge 3n,\]
    the diameter of $\FS$ is polynomially bounded in $n$?
\end{question}
If either $X$ or $Y$ is not connected, then $\FS$ is not connected. However, we can still ask for a bound on the diameter of any connected component of $\FS$. We will answer this question in the affirmative, both if $X,Y$ are connected, and if we look at any component of $\FS$ for disconnected $X$ or $Y$. 

We resolve \Cref{conj:kiril} and \Cref{q:ALI} in the affirmative and find a sharp diameter bound in the \textcolor{dg}{green} region.

\subsection{Main Results}

The main result of this paper proving that all $\FS$ in the \textcolor{dg}{green} region are connected, and proving a sharp diameter bound in the \textcolor{dg}{green} region. This answers \Cref{conj:kiril} in the affirmative.
\begin{theorem}\label{thm:main}
    Suppose $X,Y$ are connected and
    \[2\min\{\dd(X),\dd(Y)\}+3\max\{\dd(X),\dd(Y)\}\ge 3n.\]
    Then $\FS$ is connected. Furthermore,
    \[\diam(\FS) = O\!\left(\frac{n^2}{\min\{\dd(X),\dd(Y)\}}\right).\]
\end{theorem}
We simultaneously establish the connectivity of $\FS$ in the \textcolor{dg}{green} region and bound the diameter here. \Cref{thm:main} improves on the bounds given by \Cref{thm:ALIn5} and establishes new diameter bounds in the \textcolor{dg}{green} region of \Cref{fig:connectivity diagram} where none existed previously. In addition, we answer \Cref{q:ALI} when $X$ and $Y$ are connected in the affirmative. When $\min\{\dd(X),\dd(Y)\}$ is at least a fixed portion of $n$, we have the following corollary.
\begin{corollary}\label{cor:diam}
    Suppose $X,Y$ are connected graphs on $n$ vertices such that  $\min\{\dd(X),\dd(Y)\} = \Omega(n)$ and
    \[2\min\{\dd(X),\dd(Y)\}+3\max\{\dd(X),\dd(Y)\}\ge 3n.\]
    Then $\diam (\FS) = \Theta(n)$.
\end{corollary}
The diameter bound in this regime is sharp in the following sense: 
\begin{proposition}\label{prop:diam lower bound}
    For all positive integers $n$ and $k$ with $n\ge 2k+1$, there exist graphs $X$ and $Y$ on $n$ vertices such that $\dd(X) = 2k$ and $2\min\{\dd(X),\dd(Y)\}+3\max\{\dd(X),\dd(Y)\}\ge 3n$, and
    \[\diam(\FS)\ge \frac{n^2-n}{2\min\{\dd(X),\dd(Y)\}}.\]
    If $n$ is even, then 
    \[\diam(\FS)\ge \frac{n^2}{2\min\{\dd(X),\dd(Y)\}}.\]
\end{proposition}
This proposition demonstrates that the diameter bound
\[O\!\left(\frac{n^2}{\min\{\dd(X),\dd(Y)\}}\right)\]
is asymptotically sharp in the \textcolor{dg}{green} and \textcolor{blue}{blue} regions of \Cref{fig:connectivity diagram}. In the proof of \Cref{thm:main}, we will extract an even stronger statement, which is the following.
\begin{theorem}\label{thm:2x+3y in finite time}
    There exists an absolute constant $C$ that satisfies the following. Let $X,Y$ be connected graphs with $\dd(X)\le \dd(Y)$ and
    \[2\dd(X)+3\dd(Y) \ge 3n.\]
    Then for all $a,b\in V(X)$ with $\{a,b\}\in E(X)$ and for all $\si\in V(\FS)$, the distance between $\si$ and $\si\circ (ab)$ in $\FS$ is bounded by
    \[d_{FS}(\si,\si\circ(ab))\le C.\]
\end{theorem}
We answer \Cref{q:ALI} in the case where $X,Y$ are not necessarily connected, but still satisfy the minimum degree condition of the \textcolor{dg}{green} region in \Cref{fig:connectivity diagram}. By an inductive argument that reduces larger cases to the connected components of $X$ and $Y$, \Cref{thm:main} allows us to obtain the following bound even when $X$ and $Y$ may not be connected.
\begin{corollary}\label{cor:diam of 2x+3y component is n2}
    Let $X,Y$ be graphs on $n$ vertices such that
    \[2\min\{\dd(X),\dd(Y)\}+3\max\{\dd(X),\dd(Y)\}\ge 3n.\]
    Then the diameter of any connected component of $\FS$ is $O(n^2)$.
\end{corollary}
We extend the region of \Cref{fig:connectivity diagram} that is known to be insufficient for connectivity of $\FS$ by demonstrating that $\FS$ may be disconnected in the \textcolor{orange}{orange} region. We introduce weighted clique pairs, which are responsible for the disconnected examples in the \textcolor{orange}{orange} region. These techniques likely yield improvements to the \textcolor{red}{red} region beyond this result in the \textcolor{orange}{orange} region, but we give this result as proof of concept.
\begin{proposition}\label{prop:weighted clique pair}
    If $(\dd^+_w[G],\dd_w^+[G^C])$ is a weighted clique pair, then any pair $(\al,\be)\in[0,1]^2$ with $\al<\dd^+_w[G]$ and $\be<\dd_w^+[G^C]$ does not imply $\FS$ is connected. In particular, there exist disconnected $\FS$ with minimum degrees in the \textcolor{orange}{orange} region.
\end{proposition}

\section{Preliminaries}

In this section we will make our definitions precise and state useful existing facts about friends and strangers graphs.

\subsection{Notation and Terminology}\label{sec:notation and terminology}

We let $[n] = \{1,\dots,n\}$, and we use $(\cdot)^C$ to denote the complement of a set. We will use $\st_n$ to refer to the star graph on $n$ vertices, with edges $\{1,2\},\{1,3\},\cdots,\{1,n\}$. We will use $\starp_n$ to refer to $\st_n$ with one additional edge. The graph $\starp_n$ is depicted in \Cref{fig:starp}. In both cases, we may refer to the vertex with degree $n-1$ as the \blue{star vertex} (vertex 1 in the diagram). We may write $\st,\starp$ if $n$ is clear from context. 

\begin{figure}[h!]
    \centering
    \begin{tikzpicture}[
    vertex/.style={circle, draw, fill=white, minimum size=18pt, inner sep=1pt},
    edge/.style={thick}
]

\node[vertex] (1) at (0,2) {$1$};

\node[vertex] (2) at (-2.5,0) {$2$};
\node[vertex] (3) at (-1.5,0) {$3$};
\node[vertex] (4) at (-0.5,0) {$4$};
\node[vertex] (5) at (0.5,0) {$5$};
\node[vertex] (6) at (2.5,0) {$n$};
\node at (1.5,0) {$\cdots$};

\foreach \v in {1,2,3,4,5,6}
  \draw[edge] (1) -- (\v);

\draw[edge] (2) -- (3);

\end{tikzpicture}
    \caption{$\starp_n$.}
    \label{fig:starp}
\end{figure}

We will use $\mathrm{Cycle}_n$ to denote the cycle graph on $n$ vertices, with edges $\{1,2\},\{2,3\},\dots,\{n-1,n\},\{n,1\}$, and $K_n$ to denote the complete graph on $n$ vertices. For graphs $G_1$ and $G_2$, the notations $G_1\subset G_2$ and $G_1\subsetneq G_2$ will denote that $G_1$ is a (not necessarily induced) subgraph or proper subgraph of $G_2$, respectively. We will use the notations $d_X,d_Y,d_{FS}$ to denote the distance between two vertices in $X,Y,$ or $\FS$ respectively, which is the length of the shortest path between the two vertices. We extend this definition to subsets of vertices, where $d_X(S,T)$ denotes the minimum length of a path in $X$ from a vertex in $S$ to a vertex in $T$. The \blue{diameter} of a connected graph $G$ is the longest distance between any two vertices:
\[\diam(G) = \max_{x,y\in V(G)}d(x,y).\]

For a graph $G$ and $k\ge 1$, we denote the $k$-th power of $G$ as $G^k$. The graph $G^k$ has the same vertex set as $G$, and has edges defined by:
\[E(G^k) = \{\{x,y\}:1\le d_G(x,y)\le k\}.\]

The \blue{disjoint union} of graphs $G_1$ and $G_2$, denoted $G_1\oplus G_2$ has vertex set $V(G_1)\sqcup V(G_2)$ and edge set $E(G_1)\sqcup E(G_2)$. The \blue{graph Cartesian product} of graphs $G_1,\dots,G_k$, denoted $G_1\sq \cdots \sq G_k$ has vertex set $V(G_1)\times\cdots \times V(G_k)$, and $(g_1,\dots,g_k)$ and $(g_1',\dots,g_k')$ are adjacent if there exists an index $i\in [k]$ such that $g_j = g_j'$ for all $j\neq i$, and $\{g_i,g_i'\}\in E(G_i)$.

This paper will use the convention that lowercase letters $a,b,c,d,\dots$ will denote vertices in $V(X)$, and that primed letters $a' = \si(a),b' = \si(b),c' = \si(c),d' = \si(d),\dots$ will denote their image in $V(Y)$ under $\si$. In the language of the analogy in \Cref{sec:analogy}, $a,b,c,d,\dots $ are people, and $a',b',c',d'\dots$ are the houses they start in.

For $v\in V(X)$, we write $N_X(v)$ to refer to the set of vertices in $X$ that share an edge with $v$, not including $v$ itself (the open neighborhood of $v$). $N_X[v] = N_X(v)\cup \{v\}$ will refer to the closed neighborhood of $v$. Lastly, if $S\subset V(X)$, then $N_X(S)$ and $N_X[S]$ denote the open and closed neighborhoods of $S$ respectively; that is, $N_X(S)$ consists of all vertices that share an edge with any vertex in $S$, and $N_X[S] = N_X(S)\cup S$. When there is no ambiguity, we may use the notation $N(S)$.

The \blue{deletion} of vertex $v\in V(G)$ from $G$ is the graph obtained from removing $v$ from the vertex set and removing all edges containing $v$ from the edge set. We call $G'\subset G$ an \blue{induced subgraph} of $G$ if $G'$ can be obtained from $G$ after a sequence of vertex deletions. We say $G$ is \blue{connected} if there exists a path between any two vertices of $G$, and we say $G$ is \blue{biconnected} if $G$ remains connected after the deletion of any vertex. When $U\subset V(G)$, we let $G|_U$ denote the restriction of $G$ to $U$, and when $W\subset V(G)$, we let $G|_W$ denote the subgraph of $G$ induced by $W$. We say $\{v,w\}$ is an \blue{internal edge} of $S\subset V(G)$ if $v,w\in S$.

\begin{remark}[On Notation]\label{rk:to the X notation}
   For graphs $X$ and $Y$, we may refer to subsets $\mc{S}\subset V(X)\times V(Y)$. For convenience, we denote the projection of $\mc{S}$ onto $X$ or $Y$ by
   \begin{align*}
       \mc{S}^X &= \{v\in V(X):\text{there exists } w\in V(Y) \text{ such that }(v,w)\in \mc{S}\}\\
       \mc{S}^Y &= \{w\in V(Y):\text{there exists } v\in V(X) \text{ such that }(v,w)\in \mc{S}\}.
   \end{align*}
\end{remark}

\begin{remark}[On Diagrams]\label{rk:on diagrams}
    The diagrams throughout the paper will have vertices, which correspond to pairs $(v,\si(v))\in V(X)\times V(Y)$. Dashed edges in diagrams represent edges in $X$ (or friendships), and solid edges represent edges in $Y$ (or roads). Even though the capital letters denoting the houses will not be labelled in most diagrams, note that person $a$ generally begins in house $a'$, and so forth. 
    
    As the reader interprets a sequence of swaps in the context of a diagram, they should imagine the lowercase letters (people) swapping positions, and the dashed lines (friendships) moving with them. By contrast, the nodes in the figures (houses) and the solid lines between them (roads) will stay in place as swaps occur. At any intermediate position, we judge whether or not a swap can occur between two people based on the houses they reside in at that intermediate configuration. Note that many cases will list sequences of swaps. 
        \begin{figure}[h!]
        \centering
        \begin{tikzpicture}[
    vertex/.style={circle, draw, inner sep=1.5pt},
    scale=1.3
]

\node[vertex] (a) at (-1,1) {$a$};
\node[vertex] (b) at ( 1,1) {$b$};
\node[vertex] (c) at (-1,-1) {$c$};
\node[vertex] (d) at (1,-1) {$d$};

\draw[thick] (a) -- (c);
\draw[thick] (a) -- (d);
\draw[thick] (b) -- (d);
\draw[thick] (c) -- (d);

\draw[dashed, thick] (a) -- (b);

\draw[dashed, thick] (a) to[out=-110,in=110] (c);
\draw[dashed, thick] (b) -- (c);
\draw[dashed, thick] (a) to[out=-25,in=115] (d);
\draw[dashed, thick] (b) to[out=-70,in=70] (d);

\node at (2,0) {$\xrightarrow{(ac)}$};

\node[vertex] (a1) at (3,1) {$c$};
\node[vertex] (b1) at ( 5,1) {$b$};
\node[vertex] (c1) at (3,-1) {$a$};
\node[vertex] (d1) at (5,-1) {$d$};

\draw[thick] (a1) -- (c1);
\draw[thick] (a1) -- (d1);
\draw[thick] (b1) -- (d1);
\draw[thick] (c1) -- (d1);

\draw[dashed, thick] (a1) -- (b1);

\draw[dashed, thick] (a1) to[out=-110,in=110] (c1);
\draw[dashed, thick] (b1) -- (c1);
\draw[dashed, thick] (b1) to[out=-70,in=70] (d1);
\draw[dashed, thick] (c1) to[out=-20,in=-160] (d1);

\node at (6,0) {$\xrightarrow{(ad)}$};

\node[vertex] (a2) at (7,1) {$c$};
\node[vertex] (b2) at ( 9,1) {$b$};
\node[vertex] (c2) at (7,-1) {$d$};
\node[vertex] (d2) at (9,-1) {$a$};

\draw[thick] (a2) -- (c2);
\draw[thick] (a2) -- (d2);
\draw[thick] (b2) -- (d2);
\draw[thick] (c2) -- (d2);

\draw[dashed, thick] (a2) -- (b2);

\draw[dashed, thick] (b2) -- (c2);
\draw[dashed, thick] (a2) to[out=-25,in=115] (d2);
\draw[dashed, thick] (b2) to[out=-70,in=70] (d2);
\draw[dashed, thick] (c2) to[out=-20,in=-160] (d2);

\node at (0,-3) {$\xrightarrow{(ab)}$};

\node[vertex] (a3) at (1,-2) {$c$};
\node[vertex] (b3) at (3,-2) {$a$};
\node[vertex] (c3) at (1,-4) {$d$};
\node[vertex] (d3) at (3,-4) {$b$};

\draw[thick] (a3) -- (c3);
\draw[thick] (a3) -- (d3);
\draw[thick] (b3) -- (d3);
\draw[thick] (c3) -- (d3);

\draw[dashed, thick] (a3) -- (b3);

\draw[dashed, thick] (b3) -- (c3);
\draw[dashed, thick] (a3) to[out=-25,in=115] (d3);
\draw[dashed, thick] (b3) to[out=-70,in=70] (d3);
\draw[dashed, thick] (c3) to[out=-20,in=-160] (d3);

\node at (4,-3) {$\xrightarrow{(bd)}$};

\node[vertex] (a4) at (5,-2) {$c$};
\node[vertex] (b4) at (7,-2) {$a$};
\node[vertex] (c4) at (5,-4) {$b$};
\node[vertex] (d4) at (7,-4) {$d$};

\draw[thick] (a4) -- (c4);
\draw[thick] (a4) -- (d4);
\draw[thick] (b4) -- (d4);
\draw[thick] (c4) -- (d4);

\draw[dashed, thick] (a4) -- (b4);

\draw[dashed, thick] (a4) to[out=-110,in=110] (c4);
\draw[dashed, thick] (b4) -- (c4);
\draw[dashed, thick] (b4) to[out=-70,in=70] (d4);
\draw[dashed, thick] (c4) to[out=-20,in=-160] (d4);

\node at (8,-3) {$\xrightarrow{(bc)}$};

\node[vertex] (a5) at (9,-2) {$b$};
\node[vertex] (b5) at (11,-2) {$a$};
\node[vertex] (c5) at (9,-4) {$c$};
\node[vertex] (d5) at (11,-4) {$d$};

\draw[thick] (a5) -- (c5);
\draw[thick] (a5) -- (d5);
\draw[thick] (b5) -- (d5);
\draw[thick] (c5) -- (d5);

\draw[dashed, thick] (a5) -- (b5);

\draw[dashed, thick] (a5) to[out=-110,in=110] (c5);
\draw[dashed, thick] (b5) -- (c5);
\draw[dashed, thick] (a5) to[out=-25,in=115] (d5);
\draw[dashed, thick] (b5) to[out=-70,in=70] (d5);

\end{tikzpicture}
\caption{An example of how figures will be drawn. We do not label the houses $a',b',c'$, but will always draw with the convention that person $a$ begins in house $a'$, and so on. Note that the friendships move with the people, but the houses and roads never move.}
        \label{fig:example diagram}
    \end{figure}
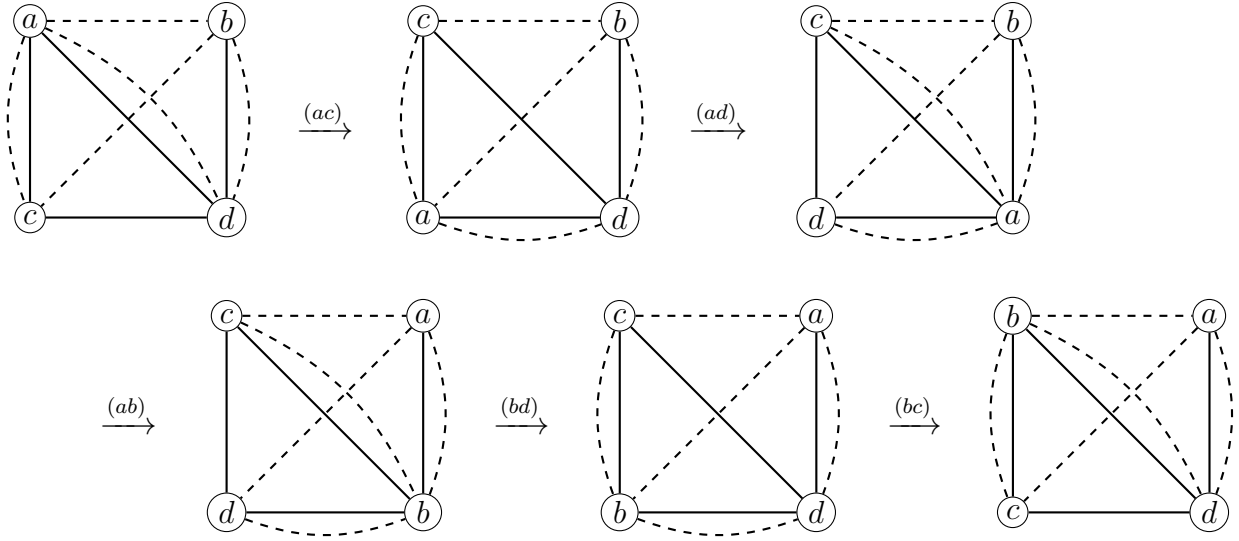
    \begin{align*}
        \textbf{Unlike standard permutation}&\textbf{ notation, transpositions}\\
        \textbf{ are meant to be carri}&\textbf{ed out left to right}
    \end{align*}
    so the reader may imagine people actually swapping places in this natural order. See \Cref{fig:example diagram} for an example diagram with all steps displayed. In future diagrams, we will omit the intermediate diagrams, but this is how the reader should verify the correctness of a given sequence of swaps.
To accompany the diagram, we will write the equality
\[(ac)(ad)(ab)(bd)(bc) = (ab).\]
Thus, friends $a$ and $b$ can be swapped in 5 moves from this configuration. Thus, if we could build this configuration, we would conclude $d_{FS}(\si,\si\circ(ab))\le 5$. We say ``$c$ and $d$ are adjacent,'' or ``$e'$ and $f'$ are adjacent'' without specifying the underlying graph. It is assumed that all nonprimed letters are vertices in $X$, all primed letters are vertices in $Y$, and any adjacencies refer to edges of corresponding graph, either $E(X)$ or $E(Y)$.
\end{remark}

\subsection{Results on Friends and Strangers Graphs}

Here we state useful existing results on friends and strangers graphs.

\begin{lemma}[\cite{Defant}]\label{lem:XY=YX}
    $\FS\cong \fs{Y}{X}$ is a graph isomorphism.
\end{lemma}

Observe that $\si\in V(\FS)\mapsto \si^{-1}\in V(\fs{Y}{X})$ is an isomorphism. Defant and Kravitz also state:

\begin{proposition}[{\cite[Proposition 2.1]{Defant}}]\label{prop:FS subset}
    If $X\subset \tilde X$ and $Y\subset\tilde Y$, then $\FS\subset \fs{\tilde X}{\tilde Y}$.
\end{proposition}

The same authors provide an interpretation of $\FS$ when $X$ is disconnected.

\begin{proposition}[{\cite[Proposition 2.4]{Defant}}]\label{prop:components}
    Let $X,Y$ be graphs on $n$ vertices, and let $X_1,\dots,X_r$ denote the connected components of $X$. For $i\in [r]$, let $n_i = |V(X_i)|$. Let $\mc{OP}_{n_1,\dots,n_r}(Y)$ denote the collection of ordered set partitions $(V_1,\dots,V_r)$ of $V(Y)$ such that $|V_i| = n_i$ for all $i\in [r]$. Then
    \[\FS\cong \bigoplus_{(V_1,\dots,V_r)\in \mc{OP}_{n_1,\dots,n_r}(Y)} (\fs{X_1}{Y|_{V_1}}\square \cdots \square \fs{X_r}{Y|_{V_r}}) \]
\end{proposition}

In Bangachev's proof of \Cref{thm:kiril}, the author used heavily the results of \cite{Wilson} and Defant and Kravitz \cite{Defant}, which tackle the case $X = \st$. To parse the following definition, \Cref{fig:theta0} depicts the graph $\theta_0$.
\begin{figure}[h!]
    \centering
    \begin{tikzpicture}[
    vertex/.style={circle, draw, fill=white, inner sep=1.5pt, minimum size=10pt},
    edge/.style={thick}
]

\node[vertex] (1) at (-0.5,1) {};
\node[vertex] (2) at (0.5,1) {};
\node[vertex] (5) at (1,0) {};
\node[vertex] (7) at (0.5,-1) {};

\node[vertex] (3) at (-1,0) {};
\node[vertex] (4) at (0,0) {};
\node[vertex] (6) at (-0.5,-1) {};

\draw[edge] (1) -- (2);
\draw[edge] (2) -- (5);
\draw[edge] (5) -- (7);
\draw[edge] (7) -- (6);
\draw[edge] (6) -- (3);
\draw[edge] (3) -- (1);
\draw[edge] (3) -- (4);
\draw[edge] (4) -- (5);

\end{tikzpicture}
    \caption{The graph $\theta_0$.}
    \label{fig:theta0}
\end{figure}
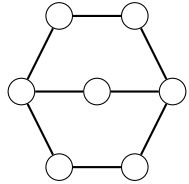

\begin{definition}[Wilsonian, Almost-Wilsonian]
    We say that a graph $G$ is \blue{Wilsonian} if it satisfies all of the following conditions:
    \begin{enumerate}
        \item $G$ is biconnected;
        \item $G$ is not bipartite;
        \item $G$ is not a cycle graph on $n\ge 4$ vertices;
        \item $G$ is not the graph $\theta_0$.
    \end{enumerate}
    If $G$ satisfies (1), (3), and (4), we say $G$ is \blue{almost-Wilsonian} ($G$ may be bipartite).
\end{definition}

Note that if $\dd(Y)>\frac{n}{2}$, then all these conditions are satisfied. 

\begin{lemma}[{\cite[Lemma 2.4]{NogaProb}}]\label{lem:delta Y is n/2}
    If $\dd(Y)>\frac{n}{2}$, then $Y$ is Wilsonian.
\end{lemma}

A key result of Wilson is the following.

\begin{proposition}[{\cite[Theorem 1]{Wilson}}]\label{prop:star wilson}
    If $Y$ is a Wilsonian graph on $n$ vertices, then $\fs{\st_n}{Y}$ is connected.
\end{proposition}

Recall the definition of $\starp$ from \Cref{sec:notation and terminology}. If we assume $X = \starp$ instead of $\st$, then Defant and Kravitz show that we may allow $Y$ to be bipartite and still have $\FS$ connected.

\begin{proposition}[{\cite[Remark 2.8]{Defant}}]\label{prop:star defant}
    If $Y$ is almost-Wilsonian, then $\fs{\starp_n}{Y}$ is connected.
\end{proposition}

The following remark states that the only obstruction to a connected friends and strangers graph (under the above conditions on $X,Y$) is parity.

\begin{remark}[{\cite[Remark 2.8]{Defant}}]\label{rk:2 component star}
    If $Y$ is almost-Wilsonian, then $\fs{\mathrm{Star}_n}{Y}$ has at most 2 components, and they are of equal size.
\end{remark}

It is a natural question to ask for the diameter of $\fs{\st}{Y}$ when it is connected. Kornhauser, Miller, and Spirakis proved a diameter estimate for any connected component of $\fs{\st}{Y}$ \cite{Kornhauser}.
\begin{lemma}[{\cite[Theorem 4]{Kornhauser}}]\label{lem:star On3}
    Let $Y$ be a graph on $n$ vertices. Then all connected components of $\fs{\st_n}{Y}$ have diameter $O(n^3)$.
\end{lemma}

We use \Cref{rk:2 component star} to extend this estimate.

\begin{lemma}\label{lem:starp On3}
    Suppose $Y$ is almost-Wilsonian or $\mathrm{Cycle}_4$. Then:
    \[\diam(\fs{\starp}{Y}) \text{ is } O(n^3).\]
\end{lemma}

\begin{proof}
    A direct check shows that $\fs{\starp_4}{\mathrm{Cycle}_4}$ is connected. If $Y$ is almost-Wilsonian, then by \Cref{rk:2 component star} and \Cref{lem:star On3}, $\fs{\st}{Y}$ has at most 2 components, each of diameter $O(n^3)$. If there is only one component, the conclusion follows, so suppose there are exactly 2 components. We find that $\fs{\starp}{Y}\supset \fs{\st}{Y}$ and $\fs{\starp}{Y}$ has one component by \Cref{prop:star defant}, so the conclusion follows.
\end{proof}

We introduce a reduction tool.

\begin{lemma}[{\cite[Lemma 2.9]{NogaProb}}]\label{lem:friends in same component}
    Suppose that for all $a,b\in V(X)$ with $\{a,b\}\in E(X)$ and all $\si \in V(\FS)$, the vertices $\si,\si\circ (ab)$ are in the same component of $\FS$. Then if $X$ is connected, $\FS$ is connected.
\end{lemma}

We will repeatedly write permutations as compositions of transpositions, and need the following elementary result. We assume the identity is the empty product.

\begin{lemma}\label{lem:n-1 trans}
    Any $\si\in S_n$ can be written as the product of at most $n-1$ transpositions.
\end{lemma}

In the next section, we will adapt these existing tools to the setting of \Cref{thm:main}.

\section{Tools for Connectivity and Diameter}

This section is devoted to results that will be used in the proof of \Cref{thm:main}. 

\subsection{Results from Structural Graph Theory}

We give results from structural graph theory on the diameter of a graph given the minimum degree. 

\begin{lemma}[{\cite[Theorem 1]{Erdos}}]\label{lem:diam from delta(X)}
    Suppose $X$ is a connected graph on $n$ vertices with minimum degree $\dd(X)$. Then $\diam(X)\le \frac{3n}{\dd(X)+1}-1$.
\end{lemma}

When $\dd(X)\ge \frac{n}{k}$ for some fixed integer $k$, the following will also be useful:

\begin{lemma}\label{lem:diam from delta(X)>n/k}
    Suppose $X$ is a connected graph on $n$ vertices with minimum degree $\dd(X)\ge \frac{n}{k}$ for some integer $k\ge 2$. Then $\diam(X)\le 3k-4$.
\end{lemma}

\begin{proof}
    Suppose for contradiction that there exist $x,y\in V(X)$ with $d_X(x,y)\ge 3k-3$. Pick a path of minimal length $x=x_0,x_1,\dots,x_{d_X(x,y)} = y$ from $x$ to $y$. However, each of the $k$ disjoint neighborhoods $N_X[x_0],N_X[x_3],\dots,N_X[x_{3k-3}]$ has size at least $ \dd(X)+1>\frac{n}{k}$. The sum of their sizes is larger than $n$, so two of them intersect.  However, this would create a shorter path from $x$ to $y$, contradicting our assumption.
\end{proof}

We say two paths are \blue{internally vertex disjoint} if their starting point and endpoint are the same, yet they do not share any intermediate vertices. The following three results show we can always find two short, internally vertex disjoint paths in a biconnected graph with high minimum degree. 

\begin{theorem}[Menger's Theorem \cite{Menger1927}]\label{thm:menger}
    Let $G$ be a biconnected graph. For any $x,y\in V(G)$, there exist two internally vertex disjoint paths between $x$ and $y$.
\end{theorem}

To leverage a large minimum degree, we use the following variant of Menger's theorem.

\begin{theorem}[{\cite[Theorem 2]{MengerPathSystems}}]\label{thm:mengerpathoriginal}
    Let $G$ be a biconnected graph on $n$ vertices. For every positive integer $d$ there exists a constant $N(d)$ satisfying the following. If $n\ge N(d)$ and
    \[\dd(G)\ge \frac{4n}{d+2}+d+12,\]
    then for every pair $x,y\in V(G)$ there exist two internally vertex disjoint paths between $x$ and $y$, each of length $\le d$. 
\end{theorem}

We will adapt this theorem slightly by letting $\dd(G)\ge n/k$ and showing the existence of a constant $D_k$ that works for all $n$.

\begin{lemma}\label{lem:finite menger}
    Let $G$ be a biconnected graph on $n$ vertices. Fix $k\in\N$, and suppose $\dd(G)>\frac{n}{k}$. Then there exists a constant $D_k\in \N$ such that for every pair $x,y\in V(G)$, there exist two internally vertex disjoint paths between $x$ and $y$, each of length $\le D_k$. 
\end{lemma}

\begin{proof}
    Let $N$ be the constant from \Cref{thm:mengerpathoriginal}. Pick $d$ such that $\frac{1}{k}> \frac{4}{d+2}$. Choose $M\in\N$ such that $n\ge M$ implies $\frac{n}{k}\ge \frac{4n}{d+2}+d+12$. Thus, if $n\ge \max\{N,M\}$, we see that $\dd(G)>\frac{n}{k}$ implies we can find two internally vertex disjoint paths of length $\le d$. Otherwise, we can use \Cref{thm:menger} to find two internally vertex disjoint paths, and clearly each has length at most $ n\le \max\{N,M\}$. Pick $D_k = \max\{N,M,d\}$, and the lemma follows.
\end{proof}

\subsection{Tools for Friends and Strangers Graphs}

The next two lemmas are modifications of \Cref{prop:star wilson} and \Cref{prop:star defant}, and are tailored to more specific cases. They tackle the cases where $X = \st$ and $X=\starp$ respectively. We require $\dd(Y)>n/2$ to have the necessary connectivity when $X = \st$, but if $X = \starp$, then the weaker condition $\dd(Y)\ge (n+2)/3$ will suffice.

\begin{lemma}\label{lem:star bound}
    Suppose $Y$ is a graph on $n$ vertices with $\dd(Y)>\frac{n}{2}$. Set $X = \st$, and let $a\in V(X)$ be a star vertex of $X$. For any $\si\in V(\fs{\st_n}{Y})$ and any $b\in V(X)$ we have
    \[d_{FS}(\si,\si\circ (ab))\le 15.\]
    It follows that for any vertices $c,d\in V(X)$, we have $d_{FS}(\si,\si\circ (cd))\le 45$. Thus, $\fs{\st_n}{Y}$ is connected and
    $\diam( \fs{\st}{Y})\le 45(n-1)$.
\end{lemma}

\begin{proof}
    Since $\dd(Y)>\frac{n}{2}$, by \Cref{lem:delta Y is n/2}, $Y$ satisfies the hypothesis of \Cref{prop:star wilson}, so $\fs{\st_n}{Y}$ is connected. Recall the convention $a' = \si(a),b' = \si(b)$, and let $a$ and $b$ be as given above. If $\{a',b'\}\in E(Y)$, then we can perform the swap $(ab)$ and $d_{FS}(\si,\si\circ (ab))=1$. Suppose instead $\{a',b'\}\notin E(Y)$. Because $\dd(Y)>\frac{n}{2}$, we have $\deg_Y(a')+\deg_Y(b')\ge n+1$. However, there are only $n-2$ vertices in $Y\back \{a',b'\}$, so at least 3 other vertices are adjacent to both $a'$ and $b'$. Label two of these vertices $c'$ and $d'$. 

    Suppose $\{c',d'\}\in E(Y)$. Then $Y$ has the edges shown in \Cref{fig:k4le}.
    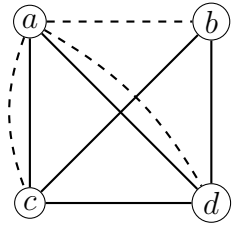
\begin{figure}[h!]
        \centering
      
        \begin{tikzpicture}[
    every node/.style={circle, draw, inner sep=1.5pt},
    scale=1.2
]

\node (a) at (-1,1) {$a$};
\node (b) at ( 1,1) {$b$};
\node (c) at (-1,-1) {$c$};
\node (d) at ( 1,-1) {$d$};

\draw[thick] (a) -- (c);
\draw[thick] (b) -- (c);
\draw[thick] (a) -- (d);
\draw[thick] (b) -- (d);
\draw[thick] (c) -- (d);

\draw[dashed, thick] (a) -- (b);
\draw[dashed, thick] (a) to[out=-110,in=110] (c);
\draw[dashed, thick] (a) to[out=-30,in=120] (d);

\end{tikzpicture}
  \caption{A subgraph of $Y$ in the case where $\{c',d'\}\in E(Y)$.}
  \label{fig:k4le}
    \end{figure}
The sequence of 7 swaps $(ac)(ab)(ad)(ac)(ab)(ac)(ad)=(ab)$ is a path from $\si$ to $\si\circ (ab)$, so $d_{FS}(\si,\si\circ (ab))\le 7$.

Finally, suppose instead that $\{c',d'\}\notin E(Y)$. Because $\dd(Y)>\frac{n}{2}$, there exists $e'\in V(Y)$ such that $\{b',e'\},\{c',e'\}\in E(Y)$. Because $\{a',b'\},\{c',d'\}\notin E(Y)$, it must be that $e'\neq a',d'$. Then $Y$ has edges as shown in \Cref{fig:C4 and a bridge}.
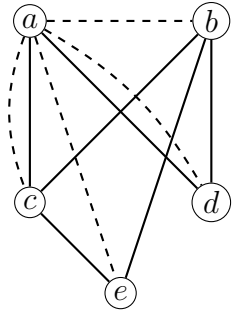
\begin{figure}[h!]
        \centering
      
        \begin{tikzpicture}[
    every node/.style={circle, draw, inner sep=1.5pt},
    scale=1.2
]

\node (a) at (-1,1) {$a$};
\node (b) at ( 1,1) {$b$};
\node (c) at (-1,-1) {$c$};
\node (d) at ( 1,-1) {$d$};
\node (e) at ( 0,-2) {$e$};

\draw[thick] (a) -- (c);
\draw[thick] (b) -- (c);
\draw[thick] (a) -- (d);
\draw[thick] (b) -- (d);
\draw[thick] (c) -- (e);
\draw[thick] (b) -- (e);

\draw[dashed, thick] (a) -- (b);
\draw[dashed, thick] (a) -- (e);
\draw[dashed, thick] (a) to[out=-110,in=110] (c);
\draw[dashed, thick] (a) to[out=-30,in=120] (d);

\end{tikzpicture}
  \caption{The graph $Y$ in the case where $\{c',d'\}\notin E(Y)$.}
  \label{fig:C4 and a bridge}
    \end{figure}
The sequence of 15 swaps:
\[(ac)(ae)(ab)(ae)(ac)(ad)(ae)(ac)(ab)(ac)(ae)(ad)(ab)(ac)(ae) = (ab).\]
is a path from $\si$ to $\si\circ (ab)$, so $d_{FS}(\si,\si\circ (ab))\le 15$.

For any $f,g\in V(X)$, we find that $(fg) = (af)(ag)(af)$. The upper bound  $d_{FS}(\si,\si\circ (ab))\le 15$ holds for all $\si$, so
\begin{align*}
    &d_{FS}(\si,\si\circ (fg))\\
    &\le d(\si,\si\circ(af))+d(\si\circ(af),\si\circ(af)(ag))+d(\si\circ(af)(ag),\si\circ(af)(ag)(af))\le 45.
\end{align*}
Thus for any transposition $(fg)$ and any vertex $\si$, the vertices $\si$ and $\si\circ(fg)$ are connected by a path of length at most $ 45$ in $\fs{\st}{Y}$. By \Cref{lem:n-1 trans}, any permutation can be written as a product of at most $ n-1$ transpositions, so $d_{FS}(\si,\tau)\le 45(n-1)$ for all $\si,\tau\in V(\fs{\st}{Y})$. 
\end{proof}

We state a similar lemma for $X = \starp$.

\begin{lemma}\label{lem:starp bound}
     Suppose $Y$ is a biconnected graph on $n\ge 5$ vertices with $\dd(Y)\ge \frac{n+2}{3}$. Set $X = \starp_n$, and let $a\in V(X)$ be a star vertex of $X$. There exists an absolute constant $\al$ such that for any $\si\in V(\fs{\starp_n}{Y})$ and any $b\in V(X)$, we have
    \[d_{FS}(\si,\si\circ (ab))\le \al.\]
    It follows that for any vertices $c,d\in V(X)$:
    \[d_{FS}(\si,\si\circ (cd))\le 3\al.\]
    Thus, $\fs{\starp_n}{Y}$ is connected and
    \[\diam( \fs{\starp_n}{Y})\le 3\al(n-1).\]
\end{lemma}

\begin{proof}
    Let $a$ be the star vertex and $a' = \si(a)$. Suppose we intend to swap $a,b$, and that $\{c,d\}\in E(X)$ is the extra edge guaranteed by $X = \starp$. If any of $a,b,c,d$ are equal, we will add other vertices in $V(X)$ until we have $4$ distinct vertices, $\{a,b,c,d\}$.
    
    We will use our results on almost-Wilsonian $Y$ from \Cref{prop:star defant} and \Cref{lem:starp  On3} to show that we can swap $a$ and $b$ in constant time. We construct an almost-Wilsonian subgraph $G\subset Y$ of size $O(1)$ that includes $a',b',c',d'$. By \Cref{lem:finite menger}, there exists a cycle of size at most $2D_3$ containing $\{a',b'\}$. Similarly, there exists such a cycle containing $\{c',d'\}$. Let us denote these cycles as $Z_1$ and $Z_2$, respectively, and note that $V(Z_1\cup Z_2)$ contains $\{a',b',c',d'\}$. We perform casework on the size of $|V(Z_1)\cap V(Z_2)|$.
    
    If $|V(Z_1)\cap V(Z_2)|\ge 2$, then the graph induced by $Z_1\cup Z_2$ is biconnected. In this case, let $H$ be the subgraph of $Y$ induced by $Z_1\cup Z_2$, and note $|V(G)|\le 4D_3$. We describe a `correction process' when $H$ is a cycle graph on $n\ge 5$ vertices or $H\cong \theta_0$. Pick vertex $e'\in V(H)$ such that $\deg_H(e') = 2$. Since $\dd(Y)\ge 3$, there exists some $f'\in V(Y)\sm V(H)$ such that $\{e',f'\}\in E(Y)$. Pick vertex $g'\in V(H)$ with $g'\neq e'$. By \Cref{lem:finite menger}, let us choose a path $P$ of length at most $ D_3$ from $f'$ to $g'$ which is vertex disjoint from $e'$, as depicted in \Cref{fig:theta0 ear}.
\begin{figure}[h!]
    \centering
    \begin{tikzpicture}[
    vertex/.style={circle, draw, fill=white, inner sep=1.5pt},
    edge/.style={thick}
]

\node[vertex] (1) at (-0.5,1) {};
\node[vertex] (2) at (0.5,1) {};
\node[vertex] (5) at (1,0) {};
\node[vertex] (7) at (0.5,-1) {};

\node[vertex] (3) at (-1,0) {$e'$};
\node[vertex] (4) at (0,0) {};
\node[vertex] (6) at (-0.5,-1) {$g'$};

\node[vertex] (8) at (-2,0) {$f'$};
\node[vertex] (9) at (-2.3,-0.7) {};
\node[vertex] (10) at (-2,-1.25) {};
\node[vertex] (11) at (-1.3,-1.5) {};

\node at (-2.5,-1.2) {$P$};

\draw[edge] (1) -- (2);
\draw[edge] (2) -- (5);
\draw[edge] (5) -- (7);
\draw[edge] (7) -- (6);
\draw[edge] (6) -- (3);
\draw[edge] (3) -- (1);
\draw[edge] (3) -- (4);
\draw[edge] (4) -- (5);

\draw[edge] (3) -- (8);
\draw[edge] (8) -- (9);
\draw[edge] (9) -- (10);
\draw[edge] (10) -- (11);
\draw[edge] (11) -- (6);

\end{tikzpicture}
    \caption{Adding a path $P$ to the graph $H$. In this picture, $H = \theta_0$, though we also apply this `correction process' when $H$ is a cycle graph.}
    \label{fig:theta0 ear}
\end{figure}
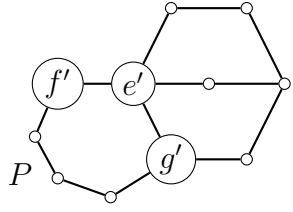
Let $H'$ be the subgraph of $Y$ induced by $V(H)\cup V(P)$. If $H \cong \theta_0$, then $H'$ is almost-Wilsonian. If $H$ is a cycle graph, then $H'$ cannot be a cycle graph, but it is possible that $H'\cong \theta_0$. In this case, applying this `correction process' twice yields an almost-Wilsonian graph $H''$. Let $G$ be the almost-Wilsonian graph after correction, where the `correction process' happens either $0$, $1$, or $2$ times. Every time the `correction process' is necessary, the number of vertices increases by at most $D_3$, the length of $P$. Therefore, $|V(G)|\le 6D_3$ in this case.

If $|V(Z_1)\cap V(Z_2)|=1$, let $e'$ be their common vertex. Pick $f'\in Z_1\back \{e'\},g'\in Z_2\back\{e'\}$. By \Cref{lem:finite menger}, we can pick a path $P'$ from $f'$ to $g'$ that is vertex disjoint from $e'$, and of length at most $ D_3$. In this case, let $G$ be the induced subgraph of $V(Z_1\cup Z_2\cup P)$. $G$ is not a cycle graph or $\theta_0$ because its edge set contains two disjoint cycles. We will show $G$ is connected after deletion of any vertex $v'$. If $v'\in V(P)$, then because $e'\in V(G)\sm \{v'\}$, the graph $G\sm \{v'\}$ is connected. Similarly, if $v'\notin V(P)$, then the path $P$ connects the cycles $V(Z_1)\sm \{v'\}$ and $V(Z_2)\sm \{v'\}$, each of which are connected after deletion. Thus, $G$ is also biconnected, and since we have shown that all other conditions are satisfied, $G$ is almost-Wilsonian. 

If $|V(Z_1)\cap V(Z_2)|=0$, then choose distinct $j',k'\in Z_1$ and $\ell',m'\in Z_2$. We apply \Cref{lem:finite menger} to $j',\ell'$ and to $k',m'$, and we call the cycles obtained $Z_3,Z_4$, respectively, each of which has at most $2D_3$ vertices. This construction is pictured in \Cref{fig:C1C2C3C4}. In this case, let $G$ be the subgraph induced by $V(Z_1\cup Z_2\cup Z_3\cup Z_4)$. Note that $|V(G)|\le 8D_3$.

    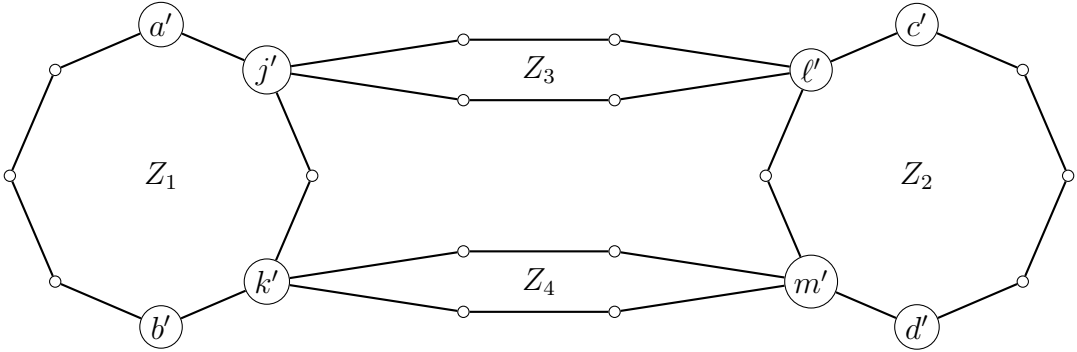
\begin{figure}[h!]
        \centering
        \begin{tikzpicture}[
    vertex/.style={circle, draw, inner sep=1.5pt},
    scale=1
]
\node[vertex] (a) at (-5,2) {$a'$};
\node[vertex] (b) at ( -5,-2) {$b'$};
\node[vertex] (j) at (-3.6,1.4) {$j'$};
\node[vertex] (k) at ( -3.6,-1.4) {$k'$};
\node[vertex] (c) at (5,2) {$c'$};
\node[vertex] (d) at ( 5,-2) {$d'$};
\node[vertex] (l) at (3.6,1.4) {$\ell'$};
\node[vertex] (m) at ( 3.6,-1.4) {$m'$};
\node[vertex] (e) at (-6.4,1.4) {};
\node[vertex] (f) at ( -6.4,-1.4) {};
\node[vertex] (g) at (-7,0) {};
\node[vertex] (h) at ( -3,0) {};
\node[vertex] (i) at (6.4,1.4) {};
\node[vertex] (n) at ( 6.4,-1.4) {};
\node[vertex] (o) at (7,0) {};
\node[vertex] (p) at ( 3,0) {};

\node[vertex] (q) at (-1,1.8) {};
\node[vertex] (r) at ( -1,-1.8) {};
\node[vertex] (s) at (1,1.8) {};
\node[vertex] (t) at ( 1,-1.8) {};
\node[vertex] (u) at (-1,1) {};
\node[vertex] (v) at ( -1,-1) {};
\node[vertex] (w) at (1,1) {};
\node[vertex] (x) at ( 1,-1) {};

\node at (-5,0) {$Z_1$};
\node at (5,0) {$Z_2$};
\node at (0,1.4) {$Z_3$};
\node at (0,-1.4) {$Z_4$};

\draw[thick] (a) -- (j);
\draw[thick] (j) -- (h);
\draw[thick] (h) -- (k);
\draw[thick] (k) -- (b);
\draw[thick] (b) -- (f);
\draw[thick] (f) -- (g);
\draw[thick] (g) -- (e);
\draw[thick] (e) -- (a);

\draw[thick] (c) -- (i);
\draw[thick] (i) -- (o);
\draw[thick] (o) -- (n);
\draw[thick] (n) -- (d);
\draw[thick] (d) -- (m);
\draw[thick] (m) -- (p);
\draw[thick] (p) -- (l);
\draw[thick] (l) -- (c);

\draw[thick] (j) -- (q);
\draw[thick] (j) -- (u);
\draw[thick] (l) -- (s);
\draw[thick] (l) -- (w);
\draw[thick] (q) -- (s);
\draw[thick] (u) -- (w);

\draw[thick] (k) -- (r);
\draw[thick] (k) -- (v);
\draw[thick] (m) -- (t);
\draw[thick] (m) -- (x);
\draw[thick] (r) -- (t);
\draw[thick] (v) -- (x);

\end{tikzpicture}
  \caption{Construction of cycles $Z_1,Z_2,Z_3,Z_4$.}
  \label{fig:C1C2C3C4}
    \end{figure}
Let us show that $G$ is biconnected by demonstrating that $G\back \{v'\}$ is connected for all $v'\in V(G)$. Suppose $v'\notin V(Z_1\cup Z_2)$. Then every vertex $w'\in V(G)\back \{v'\}$ is in $V(Z_1\cup Z_2)$ or is joined by a path to $V(Z_1\cup Z_2)$. Because the two paths in $Z_3$ from $j'$ to $\ell'$ are internally vertex disjoint, at least one path remains in $V(G)\back \{v'\}$, so $V(Z_1),V(Z_2)$ are connected. Thus, $G\back \{v'\}$ is connected in this case.

Suppose instead that $v'\in V(Z_1\cup Z_2)$, so take $v'\in V(Z_1)$ without loss of generality. It is clear that $Z_1\sm\{v'\}$ is still connected to $Z_2$ in $V(G)\sm \{v'\}$ because $j',k'$ are distinct, and thus all vertices are connected to $Z_2$, which is unchanged by the deletion of $v'$. Thus, $G$ is biconnected.
    
In each case, we obtain an almost-Wilsonian or $\mathrm{Cycle}_4$ subgraph $G\subset Y$ such that $\{a',b',c',d'\}~\subset~V(G)$. Moreover, $|V(G)|\le 8D_3$ in all cases. By \Cref{prop:star defant} in the almost-Wilsonian case, $\fs{\starp}{Y|_G}$ is connected. By \Cref{lem:starp On3} there is an absolute constant $\kappa$ such that $\diam(\fs{\starp}{G})\le \kappa |G|^3 \le \kappa(8D_3)^3=: \alpha$. It follows that
\[d_{FS}(\si,\si\circ(ab)) \le \alpha.\]
Thus, any vertices of $\FS$ that differ by a transposition involving $a$ are at most a distance $\alpha$ apart in $\FS$. Thus, any vertices that differ by a transposition (not necessarily involving $a$) are at most a distance $3\al$ apart, since $(ab)(ac)(ab) = (bc)$. By \Cref{lem:n-1 trans}, $\diam(\fs{\starp}{Y})\le 3\al(n-1)$.
\end{proof}

We use $\al$ throughout the paper to refer to the absolute constant given by \Cref{lem:starp bound}. Recall that by \Cref{lem:friends in same component}, we know that $\FS$ is connected if $\si$ and $\si\circ(ab)$ always lie in the same component of $\FS$. We introduce a reduction tool that adds a diameter bound to this result.

\begin{lemma}\label{lem:friend reduction}
    Suppose $X$ is connected and for all $\si\in V(\FS)$ and $a,b\in V(X)$ such that $\{a,b\}\in E(X)$, then $d_{FS}(\si,\si\circ(ab))\le t$. Then $\FS$ is connected and
    \[\diam(\FS)\le \left(\frac{6n}{\dd(X)+1}-3\right)(n-1)t<\frac{6n^2t}{\dd(X)}.\]
\end{lemma}

\begin{proof}
    We first show that we can perform any transposition in at most $(2\diam(X)-1)t$ moves. Let $a,b\in V(X)$. Because $X$ is connected, there exists a path $a,a_1,\dots,a_{\ell-1},b$ in $X$ from $a$ to $b$ of length $\ell\le \diam(X)$. Observe that
    \[(ab) = (aa_1)(a_1a_2)\cdots (a_{\ell-1}b)(a_{\ell-1}a_{\ell-2})\cdots(a_2a_1)(a_1a)\]
    is a product of $2\ell-1\le 2\diam(X)-1$ transpositions, all of which correspond to edges in $X$. By assumption, each of these transpositions in $X$ corresponds to a path in $\FS$ of length $\le t$. It follows that
    \[d_{FS}(\si,\si\circ(ab))\le (2\diam(X)-1)t.\]
    Now fix arbitrary vertices $\si,\tau\in V(\FS)$. By \Cref{lem:n-1 trans}, $\si^{-1}\tau$ can be written as a product of $n-1$ transpositions, so
    \begin{equation}\label{eq:dist in starp proof si tau}
        d_{FS}(\si,\tau)\le (2\diam(X)-1)(n-1)t.
    \end{equation}
    Since $X$ is connected, $\dd(X)\ge 1$. Combining \eqref{eq:dist in starp proof si tau} and \Cref{lem:diam from delta(X)} gives the desired bound.
\end{proof}

The above lemma allows us to bound the diameter of $\FS$ in terms of the time to swap $a,b$ for $\{a,b\}\in E(X)$, that is, the time to `swap' two friends, and the minimum degree $\dd(X)$. We state a structural lemma proven by Bangachev.  

\begin{lemma}[{\cite[Lemma 4.1]{Bangachev}}]\label{lem:bangachev structure}
    Let $G$ be a graph on $n$ vertices with minimum degree $\dd(G)$. Let $U\subset V(G)$ with $|U|\ge 5$, and $2|U|+3\dd(G)\ge 3n+2$. Then $G|_U$ has at most two connected components. Furthermore, there are three options for the structure of $G|_U$.
    \begin{enumerate}
        \item[\emph{(C1)}] $G|_U$ has 1 biconnected component. Then $G|_U$ is almost-Wilsonian.
        
        \item[\emph{(C2)}] $G|_U$ has 2 connected components $F_1$ and $F_2$. Then
        \begin{align}
            \label{eq:2comp vertices}\dd(G)+1+|U|-n&\le |V(F_i)|\le n-\dd(G)-1\\
            \label{eq:2comp degree}\text{ and }\quad\dd(G|_{F_i})\ge \dd(G)&+|U|-n> |V(F_i)|/2
        \end{align}
        holds for $i = 1,2$. In particular, $F_1$ and $F_2$ are Wilsonian.

        \item[\emph{(C3)}] $G|_U$ has 1 connected component with a cut vertex $v$. Then $G|_{U\backslash \{v\}}$ has two connected components $F_1',F_2'$ and
        \begin{align}
            \label{eq:cut vertices}\dd(G)+|U|-n&\le |V(F_i')|\le n-\dd(G)-1\\
            \label{eq:cut degree}\text{ and }\quad\dd(G_{F_i'})\ge \dd(G)&+|U|-n-1 >|V(F_i)|/2
        \end{align}
        holds for $i=1,2$. In particular, $F_1'$ and $F_2'$ are Wilsonian. Additionally, if $|N(v)\cap V(F_i')|\ge 2$, then $G|_{F_i'\cup \{v\}}$ is Wilsonian.
    \end{enumerate}
\end{lemma}

The three cases of Bangachev's structural lemma will come up often, and we refer to them as $(C1),(C2),$ and $(C3)$, respectively, throughout the paper. It is not difficult to check that under the hypothesis $2\dd(X)+3\dd(Y)\ge 3n$, taking $U = \si(N_X[v])$ for some $v\in V(X)$ and $G = Y$ satisfies the hypothesis $2|U|+3\dd(G)\ge 3n+2$ of \Cref{lem:bangachev structure}. In the cases $(C1),(C2),$ and $(C3)$, we obtain bounds on the distance between two vertices of $\FS$ that differ by a transposition. Recall that we use the convention $v' = \si(v)$, where $v\in V(X)$ and $v'\in V(Y)$.


\begin{lemma}\label{lem:structure}
    Let $X$ and $Y$ be graphs on $n$ vertices with $2\dd(X)+3\dd(Y)\ge 3n$ and $\dd(X)\ge 4$. Fix a vertex $v\in V(X)$ and a bijection $\si:V(X)\to V(Y)$. We apply \Cref{lem:bangachev structure} to $Y|_{\si(N_X[v])}$, and examine the cases $(C1),(C2),(C3)$.

    Suppose $(C1)$ holds. Then there exists an absolute constant $\al$ such that if $X|_{N_X[v]}\supset \starp$, then for any $w',x'\in V(Y|_{\si(N_X[v])})$ we have
    \begin{align*}
        d_{FS}(\si,\si\circ(vw))&\le \al\quad\text{ and }\quad d_{FS}(\si,\si\circ(wx))\le 3\al.
    \end{align*}

    Suppose $(C2)$ holds. Then for any $w',x'\in V(Y|_{\si(N_X[v])})$ in the same component of $Y|_{\si(N_X[v])}$ as $v'$, we have
    \begin{align*}
        d_{FS}(\si,\si\circ(vw))&\le 15\quad\text{ and }\quad
        d_{FS}(\si,\si\circ(wx))\le 45.
    \end{align*}

    Suppose $(C3)$ holds. Let $z'$ be the cut vertex of $Y|_{\si(N_X[v])}$. For any $w',x'\in V(Y|_{\si(N_X[v])})$ in the same component of $Y|_{\si(N_X[v])\back \{z'\}}$ as $v'$, we have
    \begin{align*}
        d_{FS}(\si,\si\circ(vw))&\le 15\quad\text{ and }\quad
        d_{FS}(\si,\si\circ(wx))\le 45.
    \end{align*}
\end{lemma}

\begin{proof}
    Suppose $(C1)$ holds and $X|_{N_X[v]}\supset \starp$. Since $|N_x[v]|\ge \dd(X)+1$, we have $2|N_X[v]|+3\dd(Y)\ge 3n+2$. We may rearrange this equation and use inclusion-exclusion to find that 
    \[\dd(Y|_{\si(N_X[v])})\ge \dd(Y)+|N_X[v]|-n\ge |N_X[v]|+\frac{2-2|N_X[v]|}{3} =  \frac{|N_X[v]|+2}{3}.\]
    Applying \Cref{lem:starp bound} to $Y|_{\si(N_X[v])}$ shows the desired result.

    If $(C2)$ or $(C3)$ holds, then by \Cref{lem:bangachev structure} the minimum degree in each component of $Y|_{\si(N_X[v])}$ or $Y|_{\si(N_X[v])}\sm\{z'\}$ is more than half the size of the component. An application of \Cref{lem:star bound} to the component of $Y|_{\si(N_X[v])}$ containing $v',w',x'$ yields the desired result.
\end{proof}

We often leverage the following simple observation.

\begin{lemma}\label{lem:ddG|U}
    Under the hypotheses of \Cref{lem:bangachev structure}, $\dd(G|_U)>\frac13|U|$ and the diameter of every component is at most $5$.
\end{lemma}
\begin{proof}
    Recall that $\dd(G|_U)\ge \dd(G)+|U|-n$ by inclusion exclusion, so
    \begin{align*}
        3\dd(G|_U)&\ge 3\dd(G)+3|U|-3n\ge |U|+2.
    \end{align*}
    This implies $\dd(G|_U)>\frac13|U|$, and we may apply \Cref{lem:diam from delta(X)>n/k} to any connected component of $G|_U$.
\end{proof}

To conclude the proof of \Cref{thm:kiril}, Bangachev uses a specific mechanism to swap two friends \cite[Section 4.4]{Bangachev}. We provide a mild strengthening of the techniques in \cite[Section 4.4]{Bangachev} so that the proof is amenable to the setting $\dd(X)\le n/2$. The setup is largely the same; indeed, vertices $u,v,x,$ and $w$ in \cite[Section 4.4]{Bangachev} correspond to vertices $a,b,c$, and $d$ in our setup. The proof in \cite{Bangachev} uses $\dd(X)> n/2$ to ensure that $c$ and $d$ have a mutual friend, and in fact shows that $c$ and $d$ have at least $3$ mutual friends. Since we cannot assume $\dd(X)>n/2$, we do not have this luxury. Thus, we take the existence of mutual friend $e$ of $c$ and $d$ as an assumption, and much of the technical work of the following lemma is in showing that one mutual friend ($e$) is sufficient to swap $a$ and $b$, even when $\dd(X)\le n/2$. Recall that after fixing $\si\in V(\FS)$, we may use pairs $(a,\si(a)) = (a,a')$ to refer to people/house pairs under $\si$, as discussed in \Cref{rk:on diagrams}. Let $\al$ denote the constant from \Cref{lem:starp bound}.

\begin{lemma}\label{lem:bangachev with P=1}
    Suppose $X$ and $Y$ are connected graphs satisfying $2\dd(X)+3\dd(Y)\ge 3n$ and $\dd(X)\ge 4$. Let $\si\in V(\FS)$. If there exist distinct pairs $(a,a'),(b,b'),(c,c'),(d,d'),(e,e')$ satisfying
    \begin{align*}
        \{a,b\},\{a,c\},\{b,c\},\{c,e\},\{d,e\}&\in E(X),\\
        \text{ and}\quad\{a',d'\},\{b',d'\},\{c',d'\}&\in E(Y),
    \end{align*}
    then $d_{FS}(\si,\si\circ(ab))\le \max\{3\al,219\}$.
\end{lemma}

 \begin{figure}[h!]
        \centering
        
        \begin{tikzpicture}[
    every node/.style={circle, draw, inner sep=1.5pt},
    scale=1.2
]

\node (a) at (-1.2,2) {$a$};
\node (b) at ( 1.2,2) {$b$};
\node (c) at (-1.2,0) {$c$};
\node (d) at ( 1.2,0) {$d$};
\node (e) at ( 0,-1.4) {$e$};

\draw[thick] (a) -- (d);
\draw[thick] (b) -- (d);
\draw[thick] (c) -- (d);

\draw[dashed, thick] (a) -- (b);
\draw[dashed, thick] (a) -- (c);
\draw[dashed, thick] (c) -- (b);
\draw[dashed, thick] (d) -- (e);
\draw[dashed, thick] (c) -- (e);

\end{tikzpicture}
        \caption{Expanding upon Bangachev's setup \cite[Section 4.4]{Bangachev}.}
        \label{fig:bangachev extension}

    \end{figure}
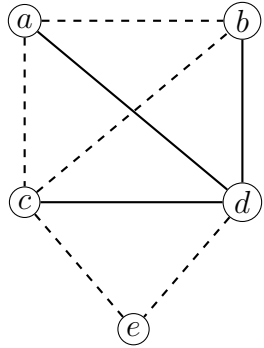

\begin{proof}
    \Cref{fig:bangachev extension} depicts our starting configuration. If $\{c,d\}\in E(X)$, then the sequence $(cd)(ac)(ab)(bc)(cd) = (ab)$ would guarantee $d_{FS}(\si,\si\circ(ab))\le 5$, so we may assume $\{c,d\}\notin E(X)$. Note also that if $a,b\in N_X(d)$, then $(ad)(ab)(bd) = (ab)$ would show $d_{FS}(\si,\si\circ(ab))\le 3$, so we may assume $\{b,d\}\notin E(X)$ without loss of generality. If there exists $f\neq a$ such that $f\in N_X(c)\cap N_X(d)$ and $f'\in N_Y(d')$, then $f\neq b$ as we have assumed $\{b,d\}\notin E(X)$, so we would have the setup given by the left side of \Cref{fig:PR}. The sequence $(df)(cf)(ac)(ab)(bc)(cf)(df) = (ab)$ shows $d_{FS}(\si,\si\circ(ab))\le 7$ in this case. Observe that if $d'$ and $e'$ were adjacent, $e=f$ would satisfy these hypotheses. Thus, we may assume $d'$ and $e'$ are nonadjacent. 
 \begin{figure}[h!]
        \centering
        \begin{tikzpicture}[
    every node/.style={circle, draw, inner sep=1.5pt},
    scale=1.2
]

\node (a) at (-1.2,2) {$a$};
\node (b) at ( 1.2,2) {$b$};
\node (c) at (-1.2,0) {$c$};
\node (d) at ( 1.2,0) {$d$};
\node (f) at ( 0,-1.4) {$f$};

\draw[thick] (a) -- (d);
\draw[thick] (b) -- (d);
\draw[thick] (c) -- (d);
\draw[thick] (f) -- (d);

\draw[dashed, thick] (a) -- (b);
\draw[dashed, thick] (a) -- (c);
\draw[dashed, thick] (c) -- (b);
\draw[dashed, thick] (d) to[out=-110,in=30] (f);
\draw[dashed, thick] (c) -- (f);

\node (a') at (-1.2+5,2) {$a$};
\node (b') at ( 1.2+5,2) {$b$};
\node (c') at (-1.2+5,0) {$c$};
\node (d') at ( 1.2+5,0) {$d$};
\node (e') at ( 0+5,-1.4) {$e$};
\node (f') at ( 2.5+5,-1) {$f_1$};
\node (g') at ( 2.5+5,-1.8) {$g$};

\draw[thick] (a') -- (d');
\draw[thick] (b') -- (d');
\draw[thick] (c') -- (d');
\draw[thick] (g') -- (d');
\draw[thick] (f') -- (d');
\draw[thick] (g') -- (e');
\draw[thick] (f') -- (e');

\draw[dashed, thick] (a') -- (b');
\draw[dashed, thick] (a') -- (c');
\draw[dashed, thick] (c') -- (b');
\draw[dashed, thick] (d') -- (e');
\draw[dashed, thick] (c') -- (e');
\draw[dashed, thick] (d') to[out=-20,in=120] (f');
\draw[dashed, thick] (d') to[out=-65,in=140] (g');

\end{tikzpicture}
        \caption{Left: $(f,f')$ adjacent to $c,d$, and $d'$. Right $(f_1,f_1')$ and $(g,g')$ are both adjacent to $d,d'$, and $e'$.}
        \label{fig:PR}
\end{figure}
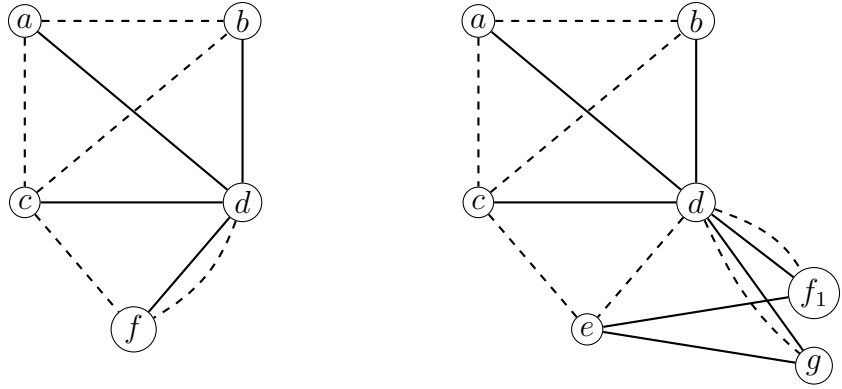
Note that $e\in N_X(c)\cap N_X(d)$ by definition. Identical to the step in \cite[Lemma 4.4]{Bangachev}, let us suppose that there exist pairs $(f_1,f_1')$ and $(g,g')$ that are both adjacent to $d,d'$, and $e'$, and $f_1,g\neq a$. The setup is pictured in the right side of \Cref{fig:PR}. The sequence:
\[(df_1)(de)(dg)(df_1)(de)(ec)(ac)(ab)(bc)(ec)(de)(df_1)(dg)(de)(df_1) = (ab)\]
demonstrates that $d_{FS}(\si,\si\circ(ab))\le 15$ in this case. Thus, there is at most one pair in addition to $(a,a')$ that could be adjacent to $d,d'$, and $e'$, so we may assume
\begin{equation}\label{eq:bound on dd'e}
    \left| \si(N_X(d))\cap N_Y(d')\cap N_Y(e')\right|\le 2.
\end{equation}
Recall that $b$ and $d$ are nonadjacent, so $(b,b')$ cannot have these adjacencies and we need not account for $b$ in \eqref{eq:bound on dd'e}. Since $d'\notin N_Y(d')\cup \si(N_X(d))$, we have $|N_Y(d')\cap \si(N_X(d))|\ge \dd(X)+\dd(Y)-n+1$ by inclusion-exclusion. Recall that $d'$ and $e'$ are nonadjacent, so we may extend the bound in \eqref{eq:bound on dd'e} to include $d'$:
\begin{equation}\label{eq:bound on dd'e updated}
    \left| \si(N_X[d])\cap N_Y[d']\cap N_Y(e')\right|\le 2.
\end{equation}
By our earlier lower bound on its open neighborhood counterpart,
\begin{equation}\label{eq:Ndd'}
    |\si(N_X[d])\cap N_Y[d']|\ge \dd(X)+\dd(Y)-n+2.
\end{equation}
Combining \eqref{eq:bound on dd'e updated} and \eqref{eq:Ndd'}, $e'$ is nonadjacent to at least $\dd(X)+\dd(Y)-n$ vertices in $N_Y[d']\cap \si(N_X[d])$. We conclude this proof by using \Cref{lem:structure} with $U = \si(N_X[c])$. Note that $a,b,c,e\in N_X[c]$ but $d\notin N_X[c]$. Since we have reduced to the case that avoids the configuration in the left side of \Cref{fig:PR}, we find
\begin{equation}\label{eq:cdd' in a'}
    U\cap \si(N_X[d])\cap N_Y[d'] = \si(N_X[c])\cap \si(N_X[d])\cap N_Y[d'] \subset \{a'\}.
\end{equation}
$b$ is absent from this intersection since $b$ and $d$ are nonadjacent. Moreover, since $a,b$ and $c$ are pairwise adjacent, $X|_{N_X[c]}\supset \starp$. We perform casework on $(C1)$,$(C2)$, and $(C3)$.

\textbf{If $(C1)$ holds}, $d_{FS}(\si,\si\circ(ab))\le 3\al$.

\textbf{If $(C2)$ holds}, let $F_1$ and $F_2$ denote the biconnected components of $Y|_{U}$. Thus, $e'\in F_1$ (without loss of generality) is nonadjacent to the other component, whose size is bounded below by \eqref{eq:2comp vertices}:
\[|V(F_2)| \ge |U|+\dd(Y)-n+1\ge \dd(X)+\dd(Y)-n+2\]
Recall that $e'$ is nonadjacent to at least $\dd(X)+\dd(Y)-n$ vertices in $\si(N_X[d])\cap N_Y[d']$ as a consequence of \eqref{eq:bound on dd'e updated} and \eqref{eq:Ndd'}, and see that \eqref{eq:cdd' in a'} implies $U$ and $\si(N_X[d])\cap N_Y[d']$ have small intersection. By the principle of inclusion-exclusion on $V(F_2)\subset Y|_U$ and $\si(N_X[d])\cap N_Y[d']$, it follows that $e'$ is nonadjacent to at least
\[(\dd(X)+\dd(Y)-n+2)+(\dd(X)+\dd(Y)-n)-1 \ge n+1-\dd(Y).\]
vertices in $Y$. This contradicts that any vertex has at most $n-1-\dd(Y)$ nonadjacencies. Thus, we may swap $a$ and $b$ using the methods from one of our reductions, and $d_{FS}(\si,\si\circ(ab))\le 15$.

\textbf{If $(C3)$ holds}, let $F_1',F_2'$ be the components after deleting the cut vertex of $Y|_U$. By \eqref{eq:cut vertices}, 
\[|V(F_1')|,|V(F_2')|\ge|U|+\dd(Y)-n\ge \dd(X)+\dd(Y)-n+1.\] 
If $e'$ was not the cut vertex, let $e'\in F_1'$ without loss of generality. Thus, $e'$ is nonadjacent to $V(F_2')$, so $e'$ is nonadjacent to at least $ \dd(X)+\dd(Y)-n$ vertices in $\si(N_X[d])\cap N_Y[d']$ and is nonadjacent to at least $\dd(X)+\dd(Y)-n+1$ vertices in $V(F_2')\subset U$. Using \eqref{eq:cdd' in a'}, the principle of inclusion-exclusion on $V(F_2')\subset Y|_U$ and $\si(N_X[d])\cap N_Y[d']$ implies that $e'$ is nonadjacent to at least
\begin{equation}\label{eq:PIE for e'}
    (\dd(X)+\dd(Y)-n+1)+(\dd(X)+\dd(Y)-n)-1 \ge n-\dd(Y).
\end{equation}
vertices in $Y$, a contradiction.

\textbf{Thus, we may assume $e'$ is the cut vertex.} We will reduce to the case where $e'$ has at least two adjacencies to each of $V(F_1')$ and $V(F_2')$. If this does not hold, $e'$ has at most one adjacency to $V(F_1')$, without loss of generality. By \eqref{eq:cut vertices}, $|V(F_1')|\ge \dd(X)+\dd(Y)-n\ge \frac13\dd(X)\ge \frac43$, so $|V(F_1')|\ge 2$, and there exists $h'\in V(F_1')$ nonadjacent to $e'$. Since $h'$ is nonadjacent to the cut vertex $e'$, $N_Y[h']\cap Y|_U\subset V(F_1')$ are in $F_1'$. By inclusion-exclusion, $\dd(Y|_U)\ge \dd(Y)+|U|-n$, so 
\begin{equation}\label{eq:VF1 and NYh'}
    |V(F_1')|\ge|N_Y[h']\cap V(Y|_U)|\ge \dd(Y)+|U|-n+1\ge \dd(X)+\dd(Y)-n+2.
\end{equation}
Since $e'$ is adjacent to at most one vertex in $V(F_1')$, \eqref{eq:VF1 and NYh'} implies that $e'$ is nonadjacent to at least $\dd(X)+\dd(Y)-n+1$ vertices in $V(F_1')\subset Y|_U$. Again, we use the inclusion-exclusion calculation in \eqref{eq:PIE for e'} to deduce that $e'$ has more than $n-1-\dd(Y)$ nonadjacencies, which is a contradiction.

\textbf{We may assume that $e'$ has at least two adjacencies in each of $V(F_1')$ and $V(F_2')$}. We show that if $a'$ and $b'$ are in the same component of $Y|_{U\back\{e'\}}$, then we can swap $a$ and $b$. Let $a',b'\in V(F_1')$ without loss of generality, and see that if $c'\in V(F_1')$ as well, then because $c$ is the star vertex of $X|_{N_X[c]}$, case $(C3)$ of \Cref{lem:structure} implies $d_{FS}(\si,\si\circ(ab))\le 45$. Thus, we may assume $c'\in V(F_2')$. 

\begin{figure}[h!]
        \centering
        \begin{tikzpicture}[
    vertex/.style={circle, draw, inner sep=1.5pt},
    scale=1
]
\node[vertex] (a) at (0,2) {$e$};
\node[vertex] (c) at (-1,1.5) {$a$};
\node[vertex] (d) at (-1,0.5) {$b$};
\node[vertex] (l) at (1,1.5) {};
\node[vertex] (l1) at (0.7,0.5) {};
\node[vertex] (l2) at (1.3,-0.4) {};
\node[vertex] (b') at (2,0) {$c$};
\draw (-1.5,0.5) ellipse [x radius=1.3cm, y radius=1.6cm];
\draw (1.5,0.5) ellipse [x radius=1.3cm, y radius=1.6cm];

\node at (-2.2,1) {$F_1'$};
\node at (2.2,1) {$F_2'$};

\draw[thick] (a) -- (c);
\draw[thick] (a) -- (d);
\draw[thick] (a) -- (l);
\draw[thick] (l1) -- (l);
\draw[thick] (l1) -- (l2);
\draw[thick] (b') -- (l2);

\node[vertex] (a') at (7,2) {$e$};
\node[vertex] (c') at (6,1.5) {$h$};
\node[vertex] (l') at (8,1.5) {};
\node[vertex] (l1') at (7.7,0.5) {};
\node[vertex] (l2') at (8.3,-0.4) {};
\node[vertex] (b') at (5,0) {$a$};
\node[vertex] (c'') at (5.7,-0.5) {$b$};
\node[vertex] (b'') at (9,0) {$c$};
\draw (5.5,0.5) ellipse [x radius=1.3cm, y radius=1.6cm];
\draw (8.5,0.5) ellipse [x radius=1.3cm, y radius=1.6cm];

\node at (4.8,1) {$F_1'$};
\node at (9.2,1) {$F_2'$};

\draw[thick] (a') -- (c');
\draw[thick] (a') -- (l');
\draw[thick] (l1') -- (l');
\draw[thick] (l1') -- (l2');
\draw[thick] (b'') -- (l2');

\end{tikzpicture}
        \caption{Swapping $(ab)$ in the case where $Y|_U$ has a cut vertex $e'$, and $a'$ and $b'$ lie in the same component of $Y|_{U\back\{e'\}}$. The star vertex is $c$, since $U = \si(N_X[c])$. Left: $e'$ is only adjacent to $a',b'$ in $V(F_1')$. Right: $e'$ is adjacent to $h'\in V(F_1')\back\{a',b'\}$.}
        \label{fig:a' and b' same component}
    \end{figure}
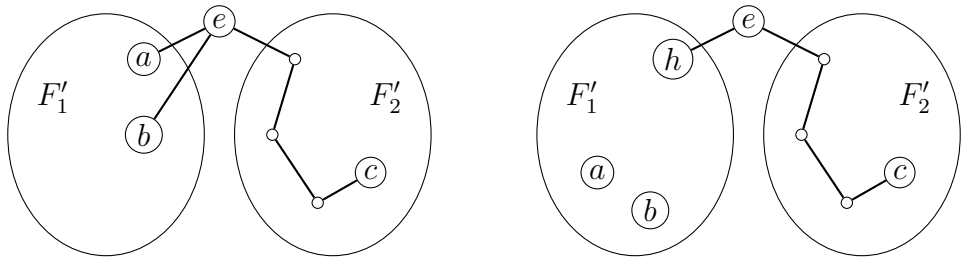

Suppose first that $a'$ and $b'$ are the only vertices in $V(F_1')$ adjacent to $e'$; the setup is pictured on the left side of \Cref{fig:a' and b' same component}. By \Cref{lem:ddG|U}, $\diam(Y|_U)\le 5$, so there exists a path of length at most $5$ in $Y|_U$ from $c'$ to $e'$. Moreover, this path can be chosen to avoid the vertices in $V(F_1')$ since $c'\in V(F_2')$ and $e'$ is the cut vertex. Since $c$ is the star vertex of $X|_{\si^{-1}(U)}$, we may move $c$ along this path to $e'$ using at most $5$ swaps. With $c$ in $e'$, we perform $(ac)(ab)(ac) = (ab)$. To conclude, we move $c$ back to $c'$ using the same path from $e'$ to $c'$, returning all vertices of the path to their initial positions in $\si$. In total, $d_{FS}(\si,\si\circ(ab))\le 5+3+5 = 13$ in this case.

Suppose instead that $e'$ is adjacent to some vertex $h'\in V(F_1')\back\{a',b'\}$; the setup is pictured on the right side of \Cref{fig:a' and b' same component}. By \Cref{lem:ddG|U}, $\diam(Y|_U)\le 5$, so there exists a path of length at most $5$ in $Y|_U$ from $c'$ to $h'$. Moreover, we may choose this path to have no internal vertices in $V(F_1')$ by letting the path begin $h',e'$ and continue in $V(F_2')$. Similar to the previous paragraph, we move $c$ to $h'$ along this path in at most $5$ moves. Once $c$ is in $h'$, we find that $a,b,c\in V(F_1')$ under the current configuration, which we refer to as $\tau$. Since we have only performed swaps inside $X|_{\si^{-1}(U)}$, we have $\si(N_X[c]) = \tau(N_X[c])$, so $U$ is unchanged. We apply $(C3)$ of \Cref{lem:structure} since $a,b,c\in V(F_1')$ and find that $d_{FS}(\tau,\tau\circ(ab))\le 45$. Lastly, we follow the same path to return $c$ to $c'$, returning all vertices in the path to their initial positions in $\si$. In total, $d_{FS}(\si,\si\circ(ab))\le 5+45+5 = 55$ in this case.

\textbf{We may assume $a'$ and $b'$ are in separate components of $Y|_{U\back\{e'\}}$.} Without loss of generality, let $a',c'\in V(F_1'), b'\in V(F_2')$ and let $e'$ be the cut vertex. Pick $j',k'\in V(F_1')$ and $\ell',m'\in V(F_2')$ such that each of $j',k',\ell',$ and $m'$ are adjacent to $e'$. It may be that $j'$ or $k'$ is $a'$ or $c'$, or that $\ell'$ or $m'$ is $b'$, but this is not a problem. 
  \begin{figure}[h!]
        \centering
        \begin{tikzpicture}[
    vertex/.style={circle, draw, inner sep=1.5pt},
    scale=1
]
\node[vertex] (a) at (0,2) {$e$};
\node[vertex] (c) at (-1,1.5) {$j$};
\node[vertex] (d) at (-1,0.5) {$k$};
\node[vertex] (l) at (1,1.5) {$l$};
\node[vertex] (m) at (1,0.5) {$m$};
\node[vertex] (b) at (-2,0) {$a$};
\node[vertex] (c') at (-1.3,-0.5) {$c$};
\node[vertex] (b') at (2,0) {$b$};
\draw (-1.5,0.5) ellipse [x radius=1.3cm, y radius=1.6cm];
\draw (1.5,0.5) ellipse [x radius=1.3cm, y radius=1.6cm];

\node at (-2.2,1) {$F_1'$};
\node at (2.2,1) {$F_2'$};

\draw[thick] (a) -- (c);
\draw[thick] (a) -- (d);
\draw[thick] (a) -- (l);
\draw[thick] (a) -- (m);

\end{tikzpicture}
        \caption{Swapping $(ab)$ in the case where $Y|_U$ has a cut vertex $e'$. The star vertex is $c$, since $U = \si(N_X[c])$.}
        \label{fig:E cut vertex}
    \end{figure}
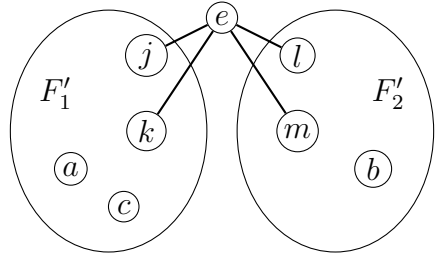

All swaps will be between two vertices in $X|_{\si^{-1}(U)}$, so under any intermediate configuration $\si_0$, we find $\si(N_X[c]) = \si_0(N_X[c])$, and thus $U$ is unchanged. Therefore, we may apply \Cref{lem:structure} case $(C3)$ at any intermediate configuration. By \Cref{lem:structure}, if $c$ is currently in the same component of $Y|_{U\back\{e'\}}$ as two vertices we wish to swap, then they can be swapped in $45$ moves. Moreover, to swap $c$ with another vertex in the same component of $Y|_{U\back\{e'\}}$ only requires $15$ moves.
Recall that the notation $(aj)^{45}$ indicates that they can be swapped in $45$ moves. If $j=a$ we omit both $(aj)^{45}$ and $(bj)^{45}$, if $k=c$ we omit both occurences of $(ck)^{15}$, and if $m=b$ we omit both $(bm)^{45}$ and $(am)^{45}$. With these caveats, we use the sequence
\[(aj)^{45}(ck)^{15}(ce)(cl)(bm)^{45}(cl)(bc)(ab)(ca)(cl)(am)^{45}(cl)(ce)(bj)^{45}(ck)^{15} =(ab).\]
Thus, $d_{FS}(\si,\si\circ(ab))\le 4\cdot 45+2\cdot 15+9 = 219$. In all cases, $d_{FS}(\si,\si\circ(ab))\le\max\{3\al,219\}$.
\end{proof}

\section{Outline of the Proof of \Cref{thm:main}}

Before beginning the proof of \Cref{thm:main}, we will provide an overview of the proof strategy. We will have a few standing assumptions, which we will state here. We assume $\dd(X)\le \dd(Y)$ by \Cref{lem:XY=YX}. We fix adjacent vertices $a,b\in V(X)$ and a vertex $\si\in V(\FS)$. We then bound the distance $d_{FS}(\si,\si\circ(ab))$. Provided that we can perform any such swap $(ab)$, the graph $\FS$ is connected by the connectivity of $X$. By demonstrating that $d_{FS}(\si,\si\circ(ab))$ is $O(1)$ and leveraging \Cref{lem:friend reduction}, we find our desired diameter bound. Here is the proof outline:

\begin{enumerate}
    \item We settle the case $\dd(X)\le 3$ using the fact that $Y$ is nearly the complete graph in this case. [\Cref{lem:d(X)le 3}]
    \item We examine the set of people/house pairs that have 3 adjacencies (in $X$ and $Y$ combined) to people $a,b$ and houses $a',b'$. For example, pair $(c,c')$ has 3 relations if $c$ is a friend of both $a$ and $b$ and $c'$ has a road to $a'$. If a pair has 4 adjacencies, we can easily swap $a$ and $b$. Thus, we reduce to the case where no pair has all 4 relations, so there exists at least a small number ($\sim 2\dd(X)/3$) of pairs with $3$ relations. [\Cref{lem:d(X)le 3}]
    \item We partition the pairs with $3$ adjacencies by the type of  adjacencies they have. This yields the sets $\mc{Q}_a,\mc{Q}_b,\mc{R}_a,\mc{R}_b$, where $(c,c')$ is in $\mc{Q}_a$ or $\mc{Q}_b$ if $c$ is friends with $a$ and $b$, and if $c'$ has a road to $a'$ or $b'$ respectively. Similarly, $\mc{R}_a$ and $\mc{R}_b$ consist of pairs $(c,c')$ where $c'$ has a road to both $a'$ and $b'$, and $c$ is friends with either $a$ or $b$ respectively. In the case where $|\mc{R}_a|\ge 2$ or $|\mc{R}_b|\ge 2$, we can swap $a$ and $b$ by focusing only on the friends of $a$ and using our results on star graphs from \Cref{lem:star bound}. Thus, we reduce to the case where $|\mc{R}_a|,|\mc{R}_b|\le 1$. [\Cref{lem:Ra}]
    \item We reduce to the case where $\mc{Q}_a,\mc{Q}_b$ are nonadjacent in $Y$ by showing that $a$ and $b$ could easily be swapped if such an edge existed. This swap would use the edge as a `bridge' between $\mc{Q}_a$ and $\mc{Q}_b$. [\Cref{lem:ban to S}]
    \item We introduce the set $\mc{S}$ of houses with at least $2$ roads to $\mc{Q}_a^*$ and at least $1$ road to $\mc{Q}_b^*$ (Note that $\mc{Q}_a^* = \mc{Q}_a\cup \{(a,a')\},\mc{Q}_b^* = \mc{Q}_b\cup \{(b,b')\}$). If pair $(c,c')$ with $c'\in \mc{S}$ had a friendship and a road to some pair $(d,d')\in \mc{Q}_a$, then we could swap the two. After doing so, we find that $d$ is friends with $a$ and $b$, and $c'$ has roads to $a'$ and $b'$, so we could swap $a$ and $b$, and then return $c$ and $d$ to their original positions. Thus, we reduce to the case where this configuration does not occur. [\Cref{lem:ban to S}]
    \item We show that if either $\mc{Q}_a$ or $\mc{Q}_b$ was empty, the other would be very large and allow us to swap $a$ and $b$. Thus, we reduce to the case where $\mc{Q}_a$ and $\mc{Q}_b$ are nonempty. [\Cref{lem:Qb nonempty}]
    \item We prove a lower bound on the size of $\mc{S}$ in terms of $\dd(Y)$. [\Cref{lem:size of S}]
    \item We show that the friends and strangers graph on $\mc{Q}_a$ is highly connected. [\Cref{lem:Qa swap} and \Cref{lem:Qa internal edge}, and \Cref{lem:swap cd in qa}]
    \item We examine the shortest path in $X$ between $\mc{Q}_a^*$ and $\mc{S}$. If there is a minimal length path which avoids $a$ and $b$, we may perform a sequence of preliminary swaps and obtain a setup that allows us to swap $a$ and $b$. [\Cref{lem:dX=1 or 2} and \Cref{lem:dX=3}]. If instead all minimal length paths include $a$ or $b$, we use an ad hoc argument to show that $a$ and $b$ may still be swapped. [\Cref{lem:ab must be in the path}] 
\end{enumerate}

\section{Proof of the Main Theorem: \Cref{thm:main}}

This section is devoted to proving the main theorem: \Cref{thm:main}. We will use the following setup for the remainder of the section. By \Cref{lem:XY=YX} we may assume $\dd(X)\le \dd(Y)$, so we rewrite our hypothesis as:
\[2\dd(X)+3\dd(Y)\ge 3n.\]
The techniques in the main proof, specifically \Cref{lem:structure}, require $\dd(X)\ge 4$. We will now prove the case $\dd(X)\le 3$.

\begin{lemma}\label{lem:d(X)le 3}
    If $X,Y$ are connected graphs on $n$ vertices with $\dd(X)\le 3$ and 
    \[2\dd(X)+3\dd(Y)\ge 3n,\]
    then $\FS$ is connected, and for any $a,b\in V(X)$ such that $\{a,b\}\in E(X)$ and any $\si\in V(\FS)$:
    \[d_{FS}(\si,\si\circ(ab))\le 9.\]
\end{lemma}

\begin{proof}
    If $\dd(X)\le 2$, then $2\dd(X)+3\dd(Y)\ge 3n$ forces $\dd(Y) = n-1$ and $Y = K_n$. For any $\{a,b\}\in E(X)$, we have $\{a',b'\}\in E(Y)$ because $Y$ is complete, so $d_{FS}(\si,\si\circ(ab))= 1$. We henceforth assume $\dd(Y) < n-1$.

    If $\dd(X) = 3$, then $2\dd(X)+3\dd(Y)\ge 3n$ forces $\dd(Y)\ge n-2$, so we assume $\dd(Y) = n-2$. Fix $\{a,b\}\in E(X)$ and assume $\{a',b'\}\notin E(Y)$, since if $\{a',b'\}\in E(Y)$ then we could immediately swap $a$ and $b$. Since $\dd(X) =3$, there exist $c,d\in V(X)$ such that $\{a,c\},\{a,d\}\in E(X)$. Because $\dd(Y) = n-2$ and $\{a',b'\}\notin E(Y)$, $a'$ is adjacent to all vertices except $b'$. If $\{c',d'\}\in E(Y)$, then we arrive at the setup in \Cref{fig:star and c4}. The swaps $(ac)(ab)(ad)(ac)(ab)(ac)(ad) = (ab)$ demonstrate $d_{FS}(\si,\si\circ(ab))\le 7$ here. 
    \begin{figure}[h!]
        \centering
        \begin{tikzpicture}[
    every node/.style={circle, draw, inner sep=1.5pt},
    scale=1.2
]

\node (a) at (-1,1) {$a$};
\node (b) at ( 1,1) {$b$};
\node (c) at (-1,-1) {$c$};
\node (d) at ( 1,-1) {$d$};

\draw[thick] (a) -- (c);
\draw[thick] (b) -- (c);
\draw[thick] (a) -- (d);
\draw[thick] (b) -- (d);
\draw[thick] (c) -- (d);

\draw[dashed, thick] (a) -- (b);

\draw[dashed, thick] (a) to[out=-105,in=105] (c);
\draw[dashed, thick] (a) to[out=-30,in=120] (d);

\end{tikzpicture}
\caption{Exchanging $a,b$ using $c,d$. Dashed lines are edges in $X$ and solid lines are edges in $Y$.}
        \label{fig:star and c4}
    \end{figure}
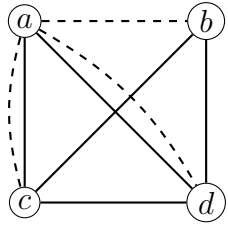
    Let us instead assume $\{c',d'\}\notin E(Y)$. Note that if $\{c,b\}\in E(X)$, then the swaps $(ac)(ab)(bc)$ show $d_{FS}(\si,\si\circ(ab))\le 3$, so we assume $\{c,b\}\notin E(X)$. Since $\dd(X) = 3$, there exists $e\in V(X)$ with $e\neq a,b,c,d$ and $\{c,e\}\in E(X)$. Recall that $a',b',c',d'$ are each nonadjacent to some vertex in $Y$, so they are all adjacent to $e'$; see \Cref{fig:dX=3 adding E} for the setup. The swaps
    \[(ce)(ac)(ab)(ad)(ac)(ab)(ac)(ad)(ce) = (ab)\]
    show that $d_{FS}(\si,\si\circ(ab))\le 9$ here.
    \begin{figure}[h!]
        \centering
        \begin{tikzpicture}[
    every node/.style={circle, draw, inner sep=1.5pt},
    scale=1.2
]

\node (a) at (-1,1) {$a$};
\node (b) at ( 1,1) {$b$};
\node (c) at (-1,-1) {$c$};
\node (d) at ( 1,-1) {$d$};
\node (e) at ( 0,-1.5) {$e$};

\draw[thick] (a) -- (c);
\draw[thick] (b) -- (c);
\draw[thick] (a) -- (d);
\draw[thick] (b) -- (d);
\draw[thick] (c) -- (e);
\draw[thick] (a) -- (e);
\draw[thick] (b) -- (e);
\draw[thick] (d) -- (e);

\draw[dashed, thick] (a) -- (b);

\draw[dashed, thick] (a) to[out=-105,in=105] (c);
\draw[dashed, thick] (a) to[out=-30,in=120] (d);
\draw[dashed, thick] (c) to[out=-45,in=170] (e);

\end{tikzpicture}
\caption{Swapping $a$ and $b$ using $c,d,$ and $e$.}
        \label{fig:dX=3 adding E}
    \end{figure}
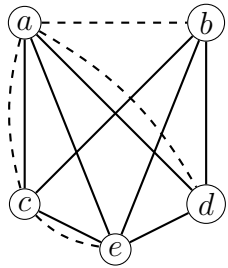
\end{proof}

\textbf{We thus assume $\dd(X)\ge 4$ for the remainder of this section.} Closed neighborhoods $\si(N_X[c])$ for any vertex $c\in V(X)$ will satisfy the hypotheses of \Cref{lem:structure}.

To show connectivity, it suffices to demonstrate that $\si,\si\circ (ab)$ are in the same component of $\FS$ for all $\si\in V(\FS)$ and all $\{a,b\}\in E(X)$. Additionally, we will find an upper bound on $d_{FS}(\si,\si\circ(ab))$, and by \Cref{lem:friend reduction} this will be sufficient to obtain a diameter bound on $\FS$. We fix an arbitrary $\si \in V(\FS)$ and friends $a,b\in V(X)$. We use the primed notation as in \Cref{sec:notation and terminology} (ex. $a'= \si(a)$).

For each $c\in V(X)$, we are interested in the pair $(c,\si(c)) = (c,c')$. In particular, we examine how many of the following hold:
\begin{equation}\label{eq:adjacencies of c,c'}
    c\in N_X(a),c\in N_X(b), c'\in N_Y(a'), c'\in N_Y(b').
\end{equation}
Note that if $\{a',b'\}\in E(Y)$, then we can perform the swap $(ab)$ and $d_{FS}(\si,\si\circ(ab))=1$. Thus, we will assume $\{a',b'\}\notin E(Y)$. Observe that
\[|N_X(a)|+|N_X(b)|+|N_Y(a')|+|N_Y(b')|\ge 2\dd(X)+2\dd(Y)\]
by the definition of minimum degree. Because $\{a,b\}\in E(X)$ and $\{a',b'\}\notin E(Y)$, $a'$ and $b'$ are in none of the  neighborhoods in \eqref{eq:adjacencies of c,c'}, and $a$ and $b$ are each in exactly one. Thus, the pairs $(c,c')$ for $c\neq a,b$ account for at least $2\dd(X)+2\dd(Y)-2$ elements in these neighborhoods. Suppose $(c,c')$ satisfies all $4$ conditions of \eqref{eq:adjacencies of c,c'}. The adjacencies between $(a,a'),(b,b'),$ and $(c,c')$ are depicted in \Cref{fig:4 conditions}.
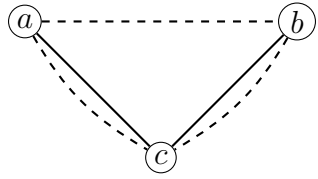
\begin{figure}[h!]
    \centering
    \begin{tikzpicture}[
    every node/.style={circle, draw, inner sep=1.5pt},
    scale=1.8
]
\node (a) at (-1,1) {$a$};
\node (b) at ( 1,1) {$b$};
\node (c) at (0,0) {$c$};

\draw[thick] (a) -- (c);
\draw[thick] (b) -- (c);

\draw[dashed, thick] (a) -- (b);

\draw[dashed, thick] (a) to[out=-60,in=150] (c);
\draw[dashed, thick] (b) to[out=-120,in=30] (c);
\end{tikzpicture}
    \caption{The case when $(c,c')$ satisfies all $4$ conditions of \eqref{eq:adjacencies of c,c'}.}
    \label{fig:4 conditions}
\end{figure}
The swaps $(ac)(ab)(bc)$ demonstrate that $d_{FS}(\si,\si\circ(ab))\le 3$, so we are done in this case.

Thus, we may assume that no pair $(c,c')$ lies in all $4$ neighborhoods of \eqref{eq:adjacencies of c,c'}. Define $\mc C_1,\mc C_2,\mc C_3$ to be the subsets of pairs $(c,c')$ for $c\neq a,b$ that lie in at least $1,2,$ or $3$ of the neighborhoods $N_X(a),N_X(b),N_Y(a'),N_Y(b')$, respectively. Note that $|\mc C_i|\le n-2$ for each $i\in \{1,2,3\}$, and 
\begin{align}\label{eq:C1+C2+C3}
    |\mc C_1|+|\mc C_2|+|\mc C_3|&\ge 2\dd(X)+2\dd(Y)-2;\\
    \nonumber|\mc C_3|&\ge 2\dd(X)+2\dd(Y)-2n+2.
\end{align}
Observe that there are four possible ways for a pair to be in $\mc C_3$, and they are as follows:
\begin{align}\label{eq:QaQbRaRb}
    \mc{Q}_a &= \{(c,c'):c\in N_X(a),c\in N_X(b),c'\in N_Y(a')\};\\
    \nonumber\mc{Q}_b &= \{(c,c'):c\in N_X(a),c\in N_X(b),c'\in N_Y(b')\};\\
    \nonumber\mc{R}_a &= \{(c,c'):c\in N_X(a),c'\in N_Y(a'),c'\in N_Y(b')\};\\
    \nonumber\mc{R}_b &= \{(c,c'):c\in N_X(b),c'\in N_Y(a'),c'\in N_Y(b')\}.
\end{align}
Clearly, $\mc C_3 = \mc{Q}_a\sqcup \mc{Q}_b\sqcup \mc{R}_a\sqcup \mc{R}_b$. 
Recall the convention from \Cref{rk:to the X notation}. It defines the set $\mc{Q}_a^X$ to be the projection of the tuples in $\mc{Q}_a$ onto $X$. We define $\mc{Q}_b^X,\mc{R}_a^X,\mc{R}_b^X$ and $\mc{Q}_a^Y$ and its counterparts similarly.

To use \Cref{lem:structure} with $U = \si(N_X[a])$, we would like two neighbors of $a$ to be adjacent, which would ensure $X|_{N_X[a]}\supset \starp$. We must first rule out the case where we do not have $\starp\subset X|_{N_{X}[a]}$. This can only happen when $N_X(a)$ and $N_X(b)$ are disjoint, which leads us to the following lemma.

\begin{lemma}\label{lem:no mutual friends}
    In the present setting, if $N_X(a)\cap N_X(b)$ then $d_{FS}(\si,\si\circ(ab))\le 13$. 
\end{lemma}

\begin{proof}
    If $N_X(a)\cap N_X(b) = \emptyset$, then $|(N_X(a)\cup N_X(b))\back \{a,b\} |\ge 2\dd(X)-2$. Moreover, $\mc{Q}_a=\mc{Q}_b = \emptyset$, and by \Cref{eq:QaQbRaRb} we have 
    \begin{equation}\label{eq:Ra+Rb}
        |\mc{R}_a|+|\mc{R}_b| = |\mc{C}_3|\ge 2\dd(X)+2\dd(Y)-2n+2.
    \end{equation}
    If there exist two pairs $(c,c'),(d,d')\in \mc{R}_a$ with $c'$ and $d'$ adjacent, then we would find the setup pictured in \Cref{fig:star and c4}. From this position, we can swap $a,b$ in 7 moves, by $(ac)(ab)(ad)(ac)(ab)(ac)(ad) = (ab)$. If instead $(c,c'),(d,d')\in \mc{R}_b$ with $c'$ and $d'$ adjacent, the argument is identical. Thus, we will assume $\mc{R}_a^Y\subset V(Y)$ and $\mc{R}_b^Y\subset V(Y)$ have no internal edges. Moreover, observe that if $|\mc{R}_a|\ge 2\dd(X)+2\dd(Y)-2n+1$, then we may fix any $(c,c')\in \mc{R}_a$ and find that
    \begin{align*}
        |\mc{R}_a^Y\cap N_Y(c')| &\ge |\mc{R}_a^Y|+|N_Y(c')|-n \ge 2\dd(X)+3\dd(Y)-3n+1 \ge 1.
    \end{align*}
    This implies there exists $(d,d')\in \mc{R}_a$ such that $c'$ and $d'$ are adjacent, so $\{c',d'\}$ is an internal edge of $\mc{R}_a^Y\subset V(Y)$. By an identical argument for $\mc{R}_b$, we find that
    \begin{equation}\label{eq:lem 5.3 bound on Ra and Rb}
        |\mc{R}_a|,|\mc{R}_b|\le 2\dd(X)+2\dd(Y)-2n.
    \end{equation}
    By \eqref{eq:Ra+Rb}, we find $|\mc{R}_a|,|\mc{R}_b|\ge 2$. Without loss of generality, we may assume $|\mc{R}_a|\ge 2$ for the remainder of the proof, since $|\mc{R}_b|\ge 2$ is symmetric. We introduce the following two sets of pairs $(c,c')$ that satisfy only one condition from \eqref{eq:adjacencies of c,c'}, and such that this condition is an adjacency in $X$. We define:
    \begin{align*}
        \mc{M}_a &= \{(c,c'):c\in N_X(a),c\notin N_X(b),c'\notin N_Y(a'),c'\notin N_Y(b'), c\neq a,b\};\\
        \mc{M}_b &= \{(c,c'):c\notin N_X(a),c\in N_X(b),c'\notin N_Y(a'),c'\notin N_Y(b'), c\neq a,b\}.
    \end{align*}
    Thus, any pair in $\mc{M}_a$ or $\mc{M}_b$ cannot be in $\mc{C}_2$. Observe that $\mc{M}_a$ and $\mc{M}_b$ are disjoint. By rearranging \eqref{eq:C1+C2+C3} and using the bound $|\mc{C}_1|\le n-2$, we find:
    \[-|\mc C_2|\le |\mc C_3|+n-2\dd(X)-2\dd(Y).\]
    Recall that $\mc C_3 = \mc{R}_a\sqcup \mc{R}_b$. Since pairs in $\mc{M}_a$ or $\mc{M}_b$ cannot be in $\mc{C}_2$, we have
    \begin{align*}
        |\mc{M}_a|+|\mc{M}_b|&\le n-2-|\mc C_2|\\
        &\le |\mc C_3|+2n-2\dd(X)-2\dd(Y)-2\\
        &= |\mc{R}_a|+|\mc{R}_b|+2n-2\dd(X)-2\dd(Y)-2.
    \end{align*}
    The following inequality must hold for $\mc{M}_a$ and $\mc{R}_a$ or for $\mc{M}_b$ and $\mc{R}_b$, so suppose without loss of generality that the former holds:
    \begin{equation}\label{eq:Ma le Ra}
        |\mc{M}_a|\le |\mc{R}_a|+n-\dd(X)-\dd(Y)-1.
    \end{equation}
    Recall that $|\mc{R}_a|\ge 2$, so we may pick distinct $(c,c'),(d,d')\in \mc{R}_a$. Since there are no internal edges in $\mc{R}_a^Y$, we have $N_Y(c')\cap \mc{R}_a^Y=\emptyset$ and $N_Y(d')\cap \mc{R}_a^Y=\emptyset$. By inclusion-exclusion,
    \[|N_Y(c')\cap N_Y(d')|\ge |N_Y(c')|+|N_Y(d')|-|N_Y(c')\cup N_Y(d')|\ge 2\dd(Y)-n+|\mc{R}_a|.\]
    Let us consider the intersection of $N_Y(c')\cap N_Y(d')$ with $\si(N_X[a])\sqcup \mc{R}_b^Y$. We use the disjoint union and the fact that $|\mc{R}_b|\ge 2$ to produce the lower bound:
    \[|\si(N_X[a])\sqcup \mc{R}_b^Y| = |\si(N_X[a])|+| \mc{R}_b^Y| \ge \dd(X)+3.\]
    We may now lower bound the intersection of $N_Y(c')\cap N_Y(d')$ with $\si(N_X[a])\sqcup \mc{R}_b^Y$. We use \eqref{eq:Ma le Ra} to lower bound $|\mc{R}_a|$ in terms of $|\mc{M}_a|$:
    \begin{align*}
        \big|\big(\si(N_X[a])\sqcup \mc{R}_b^Y\big)\cap N_Y(c')\cap N_Y(d')\big|&\ge  \dd(X)+2\dd(Y)+|\mc{R}_a|+3-2n\\
        &\ge 2\dd(X)+3\dd(Y)+|\mc{M}_a|+4-3n\\
        &\ge |\mc{M}_a|+4.
    \end{align*}
    Since this intersection is larger than $|\mc{M}_a|+2$, it must contain $e'\notin \mc{M}_a^Y$ and $e'\neq a',b'$. As a consequence of lying in this intersection, $e'$ is adjacent to $c'$ and $d'$, and $e$ is adjacent to one of $a$ and $b$.

    \textbf{Case 1: $e'\in \si(N_X(a))$.}\\
     Because $e$ is adjacent to $a$ but $e'\notin \mc{M}_a^Y$, we have that $(e,e')$ satisfies at least two conditions from \eqref{eq:adjacencies of c,c'}, so $(e,e')\in \mc C_2$. Since $N_X(a)\cap N_X(b) = \emptyset$, $e$ and $b$ are nonadjacent. Combined, $(e,e')\in\mc{C}_2$ and $\{b,e\}\notin E(X)$ imply that $e'$ is adjacent to either $a'$ or $b'$. In either case, we show that we can swap $a$ and $b$ in at most 13 moves. If $\{a',e'\}\in E(Y)$, then we find the setup on the left side of \Cref{fig:AE or BE}. We can perform the swaps
    \[(ac)(ae)(ad)(ac)(ae)(ad)(ae)(ac)(ab)(ad)(ac)(ab)(ad)=(ab).\]
    so $d_{FS}(\si,\si\circ(ab))\le 13$ in this case. 
    \begin{figure}[h!]
        \centering
        \begin{tikzpicture}[
    every node/.style={circle, draw, inner sep=1.5pt},
    scale=0.75
]

\node (a) at (-3,2) {$a$};
\node (b) at ( 3,2) {$b$};
\node (c) at (-3,-2) {$c$};
\node (d) at ( 3,-2) {$d$};
\node (e) at ( 0,-3) {$e$};

\draw[thick] (a) -- (c);
\draw[thick] (b) -- (d);
\draw[thick] (a) -- (e);
\draw[thick] (c) -- (e);
\draw[thick] (d) -- (e);
\draw[thick] (c) -- (b);
\draw[thick] (d) -- (a);

\draw[dashed, thick] (a) -- (b);

\draw[dashed, thick] (a) to[out=-105,in=105] (c);
\draw[dashed, thick] (a) to[out=-20,in=140] (d);
\draw[dashed, thick] (e) to[out=130,in=-70] (a);

\node (a') at (7,2) {$a$};
\node (b') at ( 13,2) {$b$};
\node (c') at (7,-2) {$c$};
\node (d') at ( 13,-2) {$d$};
\node (e') at ( 10,-3) {$e$};

\draw[thick] (a') -- (c');
\draw[thick] (b') -- (d');
\draw[thick] (b') -- (e');
\draw[thick] (c') -- (e');
\draw[thick] (d') -- (e');
\draw[thick] (c') -- (b');
\draw[thick] (d') -- (a');

\draw[dashed, thick] (a') -- (b');

\draw[dashed, thick] (a') to[out=-105,in=105] (c');
\draw[dashed, thick] (a') to[out=-20,in=140] (d');
\draw[dashed, thick] (e') -- (a');
\end{tikzpicture}
        \caption{Setup when $\{a',e'\}\in E(Y)$ on the left, and when $\{b',e'\}\in E(Y)$ on the right.}
        \label{fig:AE or BE}
    \end{figure}
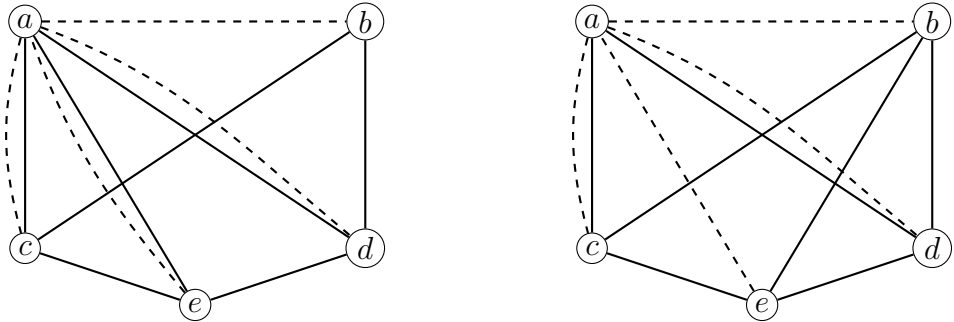

    Suppose instead $\{b',e'\}\in E(Y)$. Then we have the setup given by the right side of \Cref{fig:AE or BE}. We can perform the swaps:
    \[(ac)(ab)(ad)(ac)(ab)(ad)(ac)(ae)(ad)(ac)(ae)(ad)(ae)=(ab),\]
    so $d_{FS}(\si,\si\circ(ab))\le 13$ in this case as well.

    \textbf{Case 2: $e'\in \mc{R}_b^Y$.\\}
    Recall that in this case $e'$ is adjacent to $a',b',c',$ and $d'$, and $e$ is adjacent to $b$. The setup is shown in \Cref{fig:mutual in Rb}. In this case, we can perform the swaps:
    \[(ac)(ab)(ad)(ac)(ab)(be)(ab)(ac)(ad)(ab)(ac)(ab)(be) = (ab).\]
    This demonstrates that $d_{FS}(\si,\si\circ(ab))\le 13$ here.
    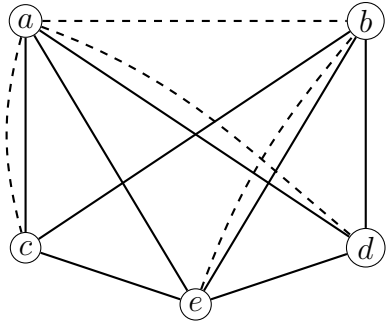
\begin{figure}[h!]
        \centering
        \begin{tikzpicture}[
    every node/.style={circle, draw, inner sep=1.5pt},
    scale=0.75
]

\node (a) at (-3,2) {$a$};
\node (b) at ( 3,2) {$b$};
\node (c) at (-3,-2) {$c$};
\node (d) at ( 3,-2) {$d$};
\node (e) at ( 0,-3) {$e$};

\draw[thick] (a) -- (c);
\draw[thick] (b) -- (d);
\draw[thick] (a) -- (e);
\draw[thick] (c) -- (e);
\draw[thick] (d) -- (e);
\draw[thick] (c) -- (b);
\draw[thick] (d) -- (a);
\draw[thick] (e) -- (b);

\draw[dashed, thick] (a) -- (b);

\draw[dashed, thick] (a) to[out=-105,in=105] (c);
\draw[dashed, thick] (a) to[out=-20,in=140] (d);
\draw[dashed, thick] (e) to[out=70,in=-130] (b);

\end{tikzpicture}
        \caption{Setup when $e'\in \mc{R}_b^Y$.}
        \label{fig:mutual in Rb}
    \end{figure}
    We have now shown that $d_{FS}(\si,\si\circ(ab))\le 13$ in all cases.
\end{proof}
\textbf{\Cref{lem:no mutual friends} allows us to reduce to the case where $N_X(a)\cap N_X(b)\neq \emptyset$, and we carry this assumption for the remainder of the section.} As a consequence, $X|_{N_X[a]}\supset \starp$, and we can use \Cref{lem:structure} more effectively. 

Recall that vertices in $\mc{R}_a^Y$ and $\mc{R}_b^Y$ are mutually adjacent to $a'$ and $b'$. Thus, if $\mc{R}_a$ or $\mc{R}_b$ is too large, it forces $a'$ and $b'$ to be in the same biconnected component of $Y|_{\si(N_X[a])}$ or $Y|_{\si(N_X[b])}$ respectively. By leveraging this idea, the following lemma reduces the problem to when $\mc{R}_a$ and $\mc{R}_b$ have at most $1$ element. The proof of this lemma will proceed by casework. Let $\al$ be the constant obtained in \Cref{lem:starp bound}.

\begin{lemma}\label{lem:Ra}
    In the present setting, if $|\mc{R}_a|\ge 2$ or $|\mc{R}_b|\ge 2$, then $d_{FS}(\si,\si\circ (ab))\le \max\{\al,21\}$.
\end{lemma}

\begin{proof}
    Note that $a,b$ are symmetric because we have only assumed that $\{a,b\}\in E(X)$. Thus, assume $|\mc{R}_a|\ge 2$ without loss of generality. In the present setting, we know $X|_{N_X[a]}\supset \starp$.
    
    Let us apply \Cref{lem:structure} to $U = \si(N_X[a])\subset V(Y)$. For each $(c,c')\in \mc{R}_a$, we have $c'\in V(Y|_U)$ and $\{a',c'\},\{b',c'\}\in E(Y|_U)$. Thus, $a',b'$ lie in the same component of $Y|_U$. We will now analyze the cases in \Cref{lem:bangachev structure} and the results from \Cref{lem:structure}.

    If $(C1)$ holds, then $d_{FS}(\si,\si\circ(ab))\le \al$.

    If $(C2)$ holds, then $a'$ and $b'$ must lie in the same component of $Y|_U$, so $d_{FS}(\si,\si\circ(ab))\le 15$.

    If $(C3)$ holds, then $Y|_U$ has a cut vertex $v'$. Label the components of $Y|_{U\back \{v'\}}$ as $F_1'$ and $F_2'$. Since $|\mc{R}_a|\ge 2$, there exists another pair $(d,d')\in \mc{R}_a$. Observe that $a'$ and $b'$ are mutually adjacent to both $c'$ and $d'$. Thus, $a',b'$ cannot be in separate components of $Y|_{U\back \{v'\}}$, however, $a'$ or $b'$ could be the cut vertex.

    If $a'$ and $b'$ are in the same component of $Y|_{U\back \{v'\}}$, then \Cref{lem:structure} implies $d_{FS}(\si,\si\circ(ab))\le 15$. 

    If $a'$ is the cut vertex, let $b'\in V(F_1')$ without loss of generality. Our setup is given by the left side of \Cref{fig:5.3 cut vertex}. Since $c'$ and $d'$ are adjacent to $b'$, we have that $c',d'\in V(F_1')$ as well. First perform the sequence:
    \[(ac)(ab)(ad)(ac)(ab).\]
    Then perform $(ac)$ in at most $15$ moves using \Cref{lem:structure} because $a,c$ are currently in $c',d'\in V(F_1')$. Note that we apply \Cref{lem:structure} to $\si\circ (ac)(ab)(ad)(ac)(ab)$, but case (C3) still holds, providing the desired bound. Then performing $(ad)$ completes the swap $(ab)$, so $d_{FS}(\si,\si\circ(ab))\le 21$.
    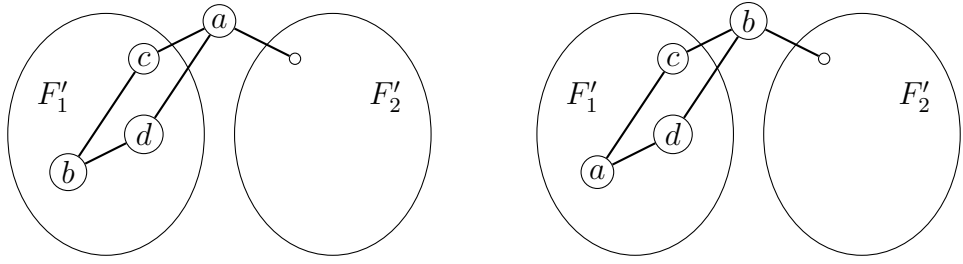
\begin{figure}
        \centering
        \begin{tikzpicture}[
    vertex/.style={circle, draw, inner sep=1.5pt},
    scale=1
]
\node[vertex] (a) at (0,2) {$a$};
\node[vertex] (c) at (-1,1.5) {$c$};
\node[vertex] (d) at (-1,0.5) {$d$};
\node[vertex] (b) at (-2,0) {$b$};
\node[vertex] (e) at (1,1.5) {};
\draw (-1.5,0.5) ellipse [x radius=1.3cm, y radius=1.6cm];
\draw (1.5,0.5) ellipse [x radius=1.3cm, y radius=1.6cm];

\node at (-2.2,1) {$F_1'$};
\node at (2.2,1) {$F_2'$};

\draw[thick] (a) -- (c);
\draw[thick] (a) -- (d);
\draw[thick] (a) -- (e);
\draw[thick] (b) -- (c);
\draw[thick] (b) -- (d);

\node[vertex] (a') at (7,2) {$b$};
\node[vertex] (c') at (6,1.5) {$c$};
\node[vertex] (d') at (6,0.5) {$d$};
\node[vertex] (b') at (5,0) {$a$};
\node[vertex] (e') at (8,1.5) {};
\draw (5.5,0.5) ellipse [x radius=1.3cm, y radius=1.6cm];
\draw (8.5,0.5) ellipse [x radius=1.3cm, y radius=1.6cm];

\node at (4.8,1) {$F_1'$};
\node at (9.2,1) {$F_2'$};

\draw[thick] (a') -- (c');
\draw[thick] (a') -- (d');
\draw[thick] (a') -- (e');
\draw[thick] (b') -- (c');
\draw[thick] (b') -- (d');

\end{tikzpicture}
        \caption{Swapping $(ab)$ in the case where $Y|_U$ has a cut vertex, and it is either $a'$ or $b'$. The cut vertex is $a'$ on the left, and $b'$ on the right.}
        \label{fig:5.3 cut vertex}
    \end{figure}
    
If $b'$ is the cut vertex of $Y|_U$, we find the setup on the right side of \Cref{fig:5.3 cut vertex}. Perform the swaps $(ac)(ab)(ad)(ac)(ab)$. As in the previous paragraph, \Cref{lem:structure} allows us to swap $(ac)$ in at most $15$ moves because $a$ and $c$ are in $c'$ and $d'$, respectively, which are both in $V(F_1')$ at the time of the swap. Then $(ad)$ completes the swap, and $d_{FS}(\si,\si\circ(ab))\le 21$ in this case as well. 

In each case, $d_{FS}(\si,\si\circ(ab))\le \max\{\al,21\}$. A symmetric argument establishes the same result when $|\mc{R}_b|\ge 2$.
\end{proof}

\textbf{We henceforth assume $|\mc{R}_a|,|\mc{R}_b|\le 1$.} Thus,
\begin{equation}\label{eq:Qa+Qb}
    |\mc{Q}_a|+|\mc{Q}_b|\ge |\mc C_3| -2\ge 2\dd(X)+2\dd(Y)-2n.
\end{equation}
Note that if equality holds in $2\dd(X)+3\dd(Y) = 3n$, then $\dd(X) + \dd(Y) - n = \frac13\dd(X)$. Thus, the reader should view a set of size $\dd(X) + \dd(Y) - n$ as being the size of one third of the minimum number of friends of any person. \textbf{For the remainder of the proof, we take $|\mc{Q}_a|\ge |\mc{Q}_b|$ without loss of generality.} We define
\begin{align*}
    \mc{Q}_a^* &= \mc{Q}_a \cup \{(a,a')\};\\
    \mc{Q}_b^* &= \mc{Q}_b \cup \{(b,b')\}.
\end{align*}
It follows that
\begin{align}\label{eq:bound on Qa, Qa*}
    |\mc{Q}_a|&\ge \frac12(|\mc{Q}_a|+|\mc{Q}_b|)\ge \dd(X)+\dd(Y)-n;\\
    \nonumber|\mc{Q}_a^*|&\ge \dd(X)+\dd(Y)-n+1.
\end{align}
Now define the set $\mc{S}$ to consist of tuples $(c,c')$ where $c'$ is adjacent to at least two vertices in $\mc{Q}_a^{*Y}$ and at least one vertex in $\mc{Q}_b^{*Y}$. Explicitly:
\begin{align}\label{eq:def of S}
    \mc{S} = \{(c,c'):\text{there exist distinct } d',e'\in \mc{Q}_a^{*Y} \text{ and } f'\in \mc{Q}_b^{*Y} \text{ adjacent to }c'\}.
\end{align}
The following structural lemma will allow us to reduce to the case where $\mc{Q}_a^{*Y}$ and $\mc{Q}_b^{*Y}$ are nonadjacent, and where a pair in $\mc{Q}_a$ and a pair in $\mc{S}$ cannot be adjacent in both $X$ and $Y$. 
\begin{lemma}\label{lem:ban to S}
    In the present setting, the following hold:
    \begin{enumerate}
        \item If there exist $c'\in \mc{Q}_a^Y,d'\in \mc{Q}_b^Y$ such that $\{c',d'\}\in E(Y)$, then
        \[d_{FS}(\si,\si\circ(ab))\le 5.\]
        \item If there exist $(c,c')\in \mc{Q}_a,(d,d')\in \mc{S}$ such that $\{c,d\}\in E(X)$ and $\{c',d'\}\in E(Y)$, then
        \[d_{FS}(\si,\si\circ(ab))\le 9.\]
    \end{enumerate}
\end{lemma}

\begin{proof}
    First suppose there exist $c'\in \mc{Q}_a^Y,d'\in \mc{Q}_b^Y$ such that $\{c',d'\}\in E(Y)$. Then we have the setup given by the left side of \Cref{fig:QaQbY}. The sequence $(ac)(ad)(ab)(bd)(bc)$ demonstrates that $d_{FS}(\si,\si\circ(ab))\le 5$, establishing (1).
    \begin{figure}[h!]
        \centering
        \begin{tikzpicture}[
    every node/.style={circle, draw, inner sep=1.5pt},
    scale=1
]

\node (a) at (-2,2) {$a$};
\node (b) at ( 2,2) {$b$};
\node (c) at (-2,-2) {$c$};
\node (d) at ( 2,-2) {$d$};

\draw[thick] (a) -- (c);
\draw[thick] (b) -- (d);
\draw[thick] (c) -- (d);

\draw[dashed, thick] (a) -- (b);

\draw[dashed, thick] (a) to[out=-105,in=105] (c);
\draw[dashed, thick] (b) to[out=-75,in=75] (d);
\draw[dashed, thick] (a) -- (d);
\draw[dashed, thick] (b) -- (c);

\node (a') at (5,2) {$a$};
\node (b') at ( 9,2) {$b$};
\node (c') at ( 5,0) {$c$};
\node (e') at (5.8,0) {$e$};
\node (g') at ( 9,0) {$f$};
\node (d') at ( 7,-2) {$d$};

\draw[thick] (a') -- (c');
\draw[thick] (a') -- (e');
\draw[thick] (b') -- (g');
\draw[thick] (c') -- (d');
\draw[thick] (d') -- (e');
\draw[thick] (d') -- (g');

\draw[dashed, thick] (a') -- (b');
\draw[dashed, thick] (a') -- (g');
\draw[dashed, thick] (a') to[out=-100,in=100] (c');
\draw[dashed, thick] (a')  to[out=-80,in=120]  (e');
\draw[dashed, thick] (b') to[out=-80,in=80] (g');
\draw[dashed, thick] (b') -- (c');
\draw[dashed, thick] (b') -- (e');
\draw[dashed, thick] (c') to[out=-55,in=145] (d');

\end{tikzpicture}
\caption{Left: an edge in $Y$ from $c'\in \mc{Q}_a^Y$ to $d'\in \mc{Q}_b^Y$. Right: $(c,c')\in \mc{Q}_a$ and $(d,d')\in \mc{S}$ are adjacent in $X$ and $Y$.}
        \label{fig:QaQbY}
    \end{figure}
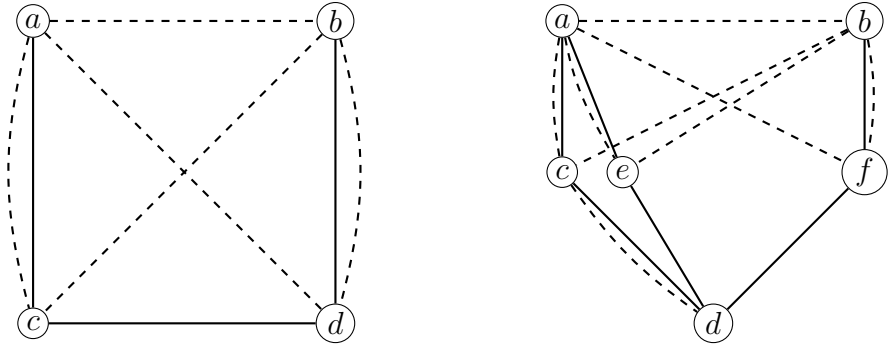

    We now show (2). Suppose there exist $(c,c')\in \mc{Q}_a,(d,d')\in \mc{S}$ such that $\{c,d\}\in E(X)$ and $\{c',d'\}\in E(Y)$. By construction of $\mc{S}$, we can pick $(e,e')\in \mc{Q}_a^*$ with $e\neq c$ and $(f,f')\in \mc{Q}_b^*$. This setup is depicted by the right side of \Cref{fig:QaQbY}.
    Note that this setup assumes $e\neq a,f\neq b$. If $e=a$ or $f=b$, the same sequence of swaps works, we can simply omit those that contain $e$ or $f$, respectively. The sequence:
\[(cd)(ae)(bf)(ac)(ab)(bc)(be)(af)(cd)=(ab)\]
implies that $d_{FS}(\si,\si\circ(ab))\le 9$. This establishes (2).
\end{proof}
Since $a'$ is nonadjacent to $\mc{Q}_b^{*Y}$ by construction, and similarly $b'$ is nonadjacent to $\mc{Q}_a^{*Y}$, part (1) of \Cref{lem:ban to S} allows us to \textbf{assume that $\mc{Q}_a^{*Y}$ and $\mc{Q}_b^{*Y}$ are nonadjacent.}

Next, we reduce to the case where $\mc{Q}_a,\mc{Q}_b$ each have at least two elements. Recall that we assume $|\mc{Q}_a|\ge|\mc{Q}_b|$, so it suffices to show that $|\mc{Q}_b|\ge 2$.
\begin{lemma}\label{lem:Qb nonempty}
    In the present setting, $|\mc{Q}_a|,|\mc{Q}_b|\ge 2$.
\end{lemma}

\begin{proof}
    The lower bound from \eqref{eq:Qa+Qb} implies that $|\mc{Q}_a^*|+|\mc{Q}_b^*|\ge  2\dd(X)+2\dd(Y)-2n+2$. Observe that $|N_Y(b')|\ge \dd(Y)$, so inclusion-exclusion implies
    \begin{align*}
        |N_Y(b')\cap (\mc{Q}_a^{*Y}\sqcup \mc{Q}_b^{*Y})| &\ge |N_Y(b')|+|\mc{Q}_a^{*Y}|+|\mc{Q}_b^{*Y}|-n\ge 2\dd(X)+3\dd(Y)-3n+2\ge 2.
    \end{align*}
    Recall that $\mc{Q}_a^{*Y}$ and $\mc{Q}_b^{*Y}$ are nonadjacent, and that $b'\in \mc{Q}_b^{*Y}$. Thus, the two elements of this intersection belong to $\mc{Q}_b^Y$ (since they cannot be $b'$), and thus $|\mc{Q}_b|\ge 2$.
\end{proof}

\textbf{We henceforth assume that $|\mc{Q}_a|\ge|\mc{Q}_b|\ge 2$.} We now show that $\mc{Q}_a$ is highly connected, and in particular, we would like to guarantee an internal edge in $Y$. 

\begin{lemma}\label{lem:Qa internal edge}
In the present setting, any vertex $c'\in \mc{Q}_a^Y$ is adjacent to another vertex in $\mc{Q}_a^Y$. Similarly, any vertex $d'\in \mc{Q}_b^Y$ is adjacent to another vertex in $\mc{Q}_b^Y$.
\end{lemma}

\begin{proof}
    Recall that the bound \eqref{eq:Qa+Qb} yields $|\mc{Q}_a\sqcup \mc{Q}_b|\ge 2\dd(X)+2\dd(Y)-2n$. We lower bound the size of the following sets by inclusion-exclusion:
    \begin{align*}
        |N_Y(c')\cap (\mc{Q}_a^Y\cup \mc{Q}_b^{*Y})|&\ge 2\dd(X)+3\dd(Y)-3n+1\ge1.\\
        |N_Y(d')\cap (\mc{Q}_a^{*Y}\cup \mc{Q}_b^{Y})|&\ge 2\dd(X)+3\dd(Y)-3n+1\ge1.
    \end{align*}
    By \Cref{lem:ban to S}, $\mc{Q}_a^{*Y}\ni c'$ is nonadjacent to $\mc{Q}_b^{*Y}\ni d'$. Thus, $c'$ is adjacent to a vertex in $\mc{Q}_a^Y$, and $d'$ is adjacent to another vertex in $\mc{Q}_b^Y$.
\end{proof}

Recall the definition of $\mc{S}$ in \eqref{eq:def of S}. Our next goal is to lower bound the size of $\mc{S}$, for which we need the following definition.

\begin{definition}
    We define $N(\mc{Q}_a^{*Y})_2$ to be the set of vertices $c'\in V(Y)$ such that there exist two distinct $d',e'\in \mc{Q}_a^{*Y}$ adjacent to $c'$. 
\end{definition}

Recall the definition of $\mc{S}$ from \eqref{eq:def of S}. Note that $(c,c')\in \mc{S}$ if and only if $c'\in N(\mc{Q}_a^{*Y})_2\cap N(\mc{Q}_b^{*Y})$. By inclusion-exclusion, we establish lower bounds on the sizes of $\mc{S}$ and $N_Y(\mc{Q}_a^{*Y})_2$.
\begin{lemma}\label{lem:size of S}
    $|\mc{S}|\ge 2\dd(Y)-n+1$ and $|N_Y(\mc{Q}_a^{*Y})_2|\ge \dd(Y)$.
\end{lemma}

\begin{proof}
    For each $c'\in \mc{Q}_a^{Y}$, $c'$ and $a'
    $ are adjacent by construction of $\mc{Q}_a^*$, and by \Cref{lem:Qa internal edge}, there exists some $d'\in\mc{Q}_a^Y$ adjacent to $c'$. Since $a'\notin \mc{Q}_a^Y$, we have $a'\neq d'$, and thus $c'\in N_Y(\mc{Q}_a^{*Y})_2$. Since $a'$ is adjacent to at least two vertices in $\mc{Q}_a^{Y}$, we find $\mc{Q}_a^{*Y}\subset N_Y(\mc{Q}_a^{*Y})_2$. 
    
    We now consider the neighbors of $\mc{Q}_a^{*Y}$. Each $c'\in \mc{Q}_a^{*Y}$ has at least $\dd(Y)$ neighbors, for a total of at least $|\mc{Q}_a^*|\dd(Y)$ neighbors of $\mc{Q}_a^{*Y}$, counted with multiplicity. Order the neighbors of $\mc{Q}_a^{*Y}$, so we may imagine placing them one by one. Recall that by \Cref{lem:ban to S}, we may assume $\mc{Q}_a^{*Y}$ and $\mc{Q}_b^{*Y}$ are nonadjacent, so there are at most $n-|\mc{Q}_b^*|$ neighbors that are the first occurrence of their vertex in the ordering. We have seen that $\mc{Q}_a^{*Y}\subset N_Y(\mc{Q}_a^{*Y})_2$. Since every vertex in $\mc{Q}_a^{*Y}$ is nonadjacent to itself and we are not counting the first occurrence, the set $\mc{Q}_a^{*Y}$ contains at most $(|\mc{Q}_a^*|-2)|\mc{Q}_a^*|$ of the neighbors that are not first occurrences.
    
    Removing first occurrences and vertices in $\mc{Q}_a^{*Y}$, the remaining neighbors of $\mc{Q}_a^{*Y}$ repeat any vertex with multiplicity at most $|\mc{Q}_a^*|-1$. By removing the first adjacencies and the adjacencies to $|\mc{Q}_a^{*Y}|$, we obtain the following bound:
    \begin{align}\label{eq:NQa*Y2}
        |N(\mc{Q}_a^{*Y})_2|&\ge |\mc{Q}_a^*|+\frac{|\mc{Q}_a^*|\dd(Y)-n+|\mc{Q}_b^*|-(|\mc{Q}_a^*|-2)|\mc{Q}_a^*|}{|\mc{Q}_a^*|-1}\\
        \nonumber&=  \frac{\dd(Y)|\mc{Q}_a^*|-n+|\mc{Q}_b^*|+|\mc{Q}_a^*|}{|\mc{Q}_a^*|-1}\\
        \nonumber&= \dd(Y)+\frac{\dd(Y)-n+|\mc{Q}_b^*|+|\mc{Q}_a^*|}{|\mc{Q}_a^*|-1}.
    \end{align}
    Applying the lower bound on $|\mc{Q}_a^*|+|\mc{Q}_b^*|$ from \eqref{eq:Qa+Qb}, we find that
    \[\dd(Y)+|\mc{Q}_a^*|+|\mc{Q}_b^*|\ge 2\dd(X)+3\dd(Y)-2n \ge n.\]
    Applying this bound to the numerator of \eqref{eq:NQa*Y2}, we obtain $|N(\mc{Q}_a^{*Y})_2| \ge  \dd(Y)$. Recall that $\mc{Q}_b$ is nonempty by standing assumption and that all elements of $\mc{Q}_b$ are adjacent to $b'$, so $\{b'\}\cup N_Y(b')\subset |N(\mc{Q}_b^{*Y})|$ implies 
    \begin{equation}\label{eq:NQb*Y}
        |N(\mc{Q}_b^{*Y})|\ge \dd(Y)+1.
    \end{equation}
    Using the principle of inclusion-exclusion, we calculate:
    \begin{align}\label{eq:lower bound S}
        |\mc{S}| &\ge |N(\mc{Q}_a^{*Y})_2|+|N(\mc{Q}_b^{*Y})|-n
    \ge 2\dd(Y)-n+1.
    \end{align}
\end{proof}

We will next establish a simple tool about swaps in $\mc{Q}_a$. Observe that $\mc{Q}_a$ is defined relative to vertices $a$ and $b$, and to the current configuration $\si$. To indicate $\mc{Q}_a$ defined with respect to a different configuration, we define the following notation.

\begin{definition}\label{def:sigma tau}
    Let $\tau\in S_n$. Then $\mc{Q}_a^{\tau}$ denotes the set $\mc{Q}_a$ defined from permutation $\si\circ \tau$. Explicitly:
    \[\mc{Q}_a^{\tau} = \{(c,\si\circ\tau(c)):c\in N_X(a),c\in N_X(b),\si\circ\tau(c)\in N_Y(\si\circ \tau(a))\}.\]
\end{definition}

If we specify $\tau$ on a subset of $S\subset [n]$, then we assume $\tau$ is the identity on $S^C$. The following lemma will allow us to swap friends within $\mc{Q}_a$, or people in $\mc{Q}_a$ whose houses share a road.

\begin{lemma}\label{lem:Qa swap}
    In the present setting, suppose $(c,c'),(d,d')\in \mc{Q}_a$ such that there exists a path from $c$ to $d$ in $X|_{\mc{Q}_a^X}$ or from $c'$ to $d'$ in $Y|_{\mc{Q}_a^Y}$ of length at most $\ell$. Then
    \[d_{FS}(\si,\si\circ(cd))\le 3(2\ell-1).\] 
    Moreover, $(\mc{Q}_a^{(cd)})^X = \mc{Q}_a^X$ and $(\mc{Q}_a^{(cd)})^Y = \mc{Q}_a^Y$. The analogous statement holds for pairs in $\mc{Q}_b$ connected by a path in $X|_{\mc{Q}_b^X}$ or $Y|_{\mc{Q}_b^Y}$.
\end{lemma}
\begin{proof}
    We show this first when $\ell = 1$, and then generalize to all $\ell$. If $\{c,d\}\in E(X)$, then we have the setup on the left side of \Cref{fig:cd swap Qa}, and $(ac)(cd)(ad)$ swaps $(cd)$ in 3 moves.
    \begin{figure}[h!]
        \centering
        \begin{tikzpicture}[
    every node/.style={circle, draw, inner sep=1.5pt},
    scale=0.9
]

\node (a) at (0,2) {$a$};
\node (ci) at ( -1.5,0) {$c$};
\node (ci+1) at ( 1.5,0) {$d$};

\draw[thick] (a) -- (ci);
\draw[thick] (a) -- (ci+1);

\draw[dashed, thick] (a) to[out=-145,in=60] (ci);
\draw[dashed, thick] (a) to[out=-35,in=120] (ci+1);
\draw[dashed, thick] (ci) -- (ci+1);

\node (a') at (5,2) {$a$};
\node (ci') at ( 3.5,0) {$c$};
\node (ci+1') at ( 6.5,0) {$d$};

\draw[thick] (a') -- (ci');
\draw[thick] (a') -- (ci+1');

\draw[dashed, thick] (a') to[out=-145,in=60] (ci');
\draw[dashed, thick] (a') to[out=-35,in=120] (ci+1');
\draw[thick] (ci') -- (ci+1');

\end{tikzpicture}
\caption{Transposition $(cd)$ in $\mc{Q}_a$ with $\{c,d\}\in E(X)$ shown on the left, and $\{c',d'\}\in E(Y)$ on the right.}
        \label{fig:cd swap Qa}
\end{figure}
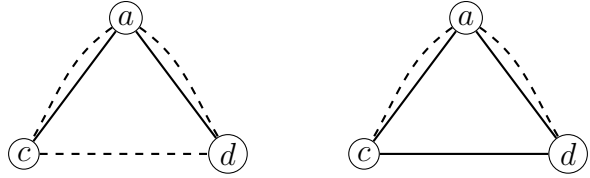
Similarly, if $\{c',d'\}\in E(Y)$, then we have the setup on the right side of \Cref{fig:cd swap Qa}, and $(ac)(ad)(ac)$ swaps $(cd)$ in 3 moves. The permutation $\si$ only alters the pairs $(c,c')$ and $(d,d')$ into $(c,d')$ and $(d,c')$. Since $c$ and $d$ are adjacent to both $a$ and $b$, and $c'$ and $d'$ are adjacent to $a'$, we find $(c,d'),(d,c')\in \mc{Q}_a^{(cd)}$. Thus, even though the pairs are not the same, the $X$ and $Y$ projections of $\mc{Q}_a^{(cd)}$ are the same as their $\mc{Q}_a$ counterparts.

Now let $\ell$ be some positive integer. If $c=p_0,p_1,\dots,p_\ell = d$ is a path in $X|_{\mc{Q}_a^X}$, then use the sequence of $2\ell-1$ transpositions
\[(cp_1)(p_1p_2)\cdots (p_{\ell-1}d) \cdot (p_{\ell-1}p_{\ell-2})(p_{\ell-2}p_{\ell-3})\cdots (p_1c) = (cd).\]
Since $\mc{Q}_a^X$ and $\mc{Q}_a^Y$ are preserved under transpositions in $\mc{Q}_a$, the sets $\mc{Q}_a^X$ and $\mc{Q}_a^Y$ are invariant under this sequence. Each transposition uses elements from $\mc{Q}_a^X$ that are adjacent in $X$, so they can each be performed in at most $3$ moves. Thus, $d_{FS}(\si,\si\circ(cd))\le 3(2\ell-1)$.

Similarly, if $c' = q_0',q_1',\dots,q_{\ell-1}',q_\ell' = d'$ is a path in $Y|_{\mc{Q}_a^Y}$, then use the sequence of $2\ell-1$ transpositions
\[(cq_1)(cq_2)\cdots (cd) \cdot (dq_{\ell-1})(dq_{\ell-2})\cdots (dq_1) = (cd).\]
All transpositions use elements from $\mc{Q}_a^X$ that are adjacent in $Y$ when they are swapped, so by similar logic, we have $d_{FS}(\si,\si\circ(cd))\le 3(2\ell-1)$ and $\mc{Q}_a^X$ and $\mc{Q}_a^Y$ are invariant. An identical argument follows for pairs in $\mc{Q}_b$ connected by a path in $X|_{\mc{Q}_b^X}$ or $Y|_{\mc{Q}_b^Y}$.
\end{proof}

We will reduce to the case where the friends and strangers graph of $\mc{Q}_a^*$ is highly connected. The proof is technical and will examine closely both the graphs $X$ and $Y$. See that if the case $d_{FS}(\si,\si\circ(ab))\le \max\{\al,21\}$ does not hold, then we may reduce to the present setting.

\begin{lemma}\label{lem:swap cd in qa}
    For all adjacent pairs $a$ and $b$ and all $\si\in V(\FS)$, at least one of the following holds:
    \begin{itemize}
        \item $d_{FS}(\si,\si\circ(ab))\le \max\{\al,21\}$.
        \item For all pairs $(c,c'),(d,d')\in \mc{Q}_a$, we have 
        \[d_{FS}(\si,\si\circ (cd))\le 33 \quad\text{ and }\quad d_{FS}(\si\circ(ab),\si\circ(ab)\circ (cd))\le 33.\]
    \end{itemize}
\end{lemma}

\begin{proof}
    We choose $(c,c'),(d,d')\in \mc{Q}_a$, and show that in the case where we cannot obtain a bound on $d_{FS}(\si,\si\circ(ab))$, we can instead bound $d_{FS}(\si,\si\circ(cd))$ and $d_{FS}(\si\circ(ab),\si\circ(ab)\circ(cd))$. The latter is the case where $a$ and $b$ have already been swapped. We are interested in the graphs $X|_{\mc{Q}_a^X},Y|_{\mc{Q}_a^Y}$. For the remainder of this proof, we will define the neighborhoods intersecting $\mc{Q}_a$ for $c,d,c'$, and $d'$; see \Cref{fig:XcYc}.
    \begin{align*}
        X_c &= N_X(c)\cap \mc{Q}_a^X,\quad  X_d = N_X(d)\cap \mc{Q}_a^X;\\
        Y_{c'} &= N_Y(c')\cap \mc{Q}_a^Y,\quad 
        Y_{d'} = N_Y(d')\cap \mc{Q}_a^Y.
    \end{align*}

\begin{figure}[ht]
\centering
\begin{tikzpicture}[
    scale=0.8,
    every node/.style={font=\small},
    v/.style={circle, fill=black, inner sep=2pt}
]


\fill[green!20]
    (-0.25,1.45) ellipse (3.05 and 1.45);

\fill[blue!20]
    (0,-1.15) ellipse (2.85 and 1.55);

\fill[red!20]
    (7.35,1.45) ellipse (3.1 and 1.45);

\fill[yellow!30]
    (7.15,-1.15) ellipse (2.9 and 1.55);


\node[v] (c) at (0,0) {};

\node[v] (cx1) at (-2.2,1.4) {};
\node[v] (cx2) at (-1.3,2.0) {};
\node[v] (cx3) at (-0.4,2.4) {};
\node[v] (cx4) at (0.8,2.4) {};
\node[v] (cx5) at (1.9,1.9) {};

\draw[dashed, thick] (c)--(cx1);
\draw[dashed, thick] (c)--(cx2);
\draw[dashed, thick] (c)--(cx3);
\draw[dashed, thick] (c)--(cx4);
\draw[dashed, thick] (c)--(cx5);

\node[v] (cy1) at (-1.7,-1.2) {};
\node[v] (cy2) at (-0.7,-1.8) {};
\node[v] (cy3) at (1.2,-1.9) {};
\node[v] (cy4) at (2.2,-1.25) {};

\draw[thick] (c)--(cy1);
\draw[thick] (c)--(cy2);
\draw[thick] (c)--(cy3);
\draw[thick] (c)--(cy4);

\node at (0.05,-0.6) {c};
\node at (-0.4,-2.1) {e};


\node[v] (d) at (7.3,0) {};

\node[v] (dx1) at (5.0,1.4) {};
\node[v] (dx2) at (6.1,2.0) {};
\node[v] (dx3) at (7.0,2.4) {};
\node[v] (dx4) at (8.1,2.4) {};
\node[v] (dx5) at (9.2,1.9) {};
\node[v] (dx6) at (10.1,1.4) {};

\draw[dashed, thick] (d)--(dx1);
\draw[dashed, thick] (d)--(dx2);
\draw[dashed, thick] (d)--(dx3);
\draw[dashed, thick] (d)--(dx4);
\draw[dashed, thick] (d)--(dx5);
\draw[dashed, thick] (d)--(dx6);

\node[v] (dy1) at (5.1,-1.3) {};
\node[v] (dy2) at (6.6,-1.8) {};
\node[v] (dy3) at (8.2,-1.8) {};
\node[v] (dy4) at (9.5,-1.25) {};

\draw[thick] (d)--(dy1);
\draw[thick] (d)--(dy2);
\draw[thick] (d)--(dy3);
\draw[thick] (d)--(dy4);

\node at (7.3,-0.6) {d};
\node at (8.5,-2.1) {f};


\draw[dashed, thick]
    (-0.25,1.45) ellipse (3.05 and 1.45);
\node at (-0.25,3.2) {$X_c$};

\draw[thick]
    (0,-1.15) ellipse (2.85 and 1.55);
\node at (0,-3.0) {$Y_{c'}$};

\draw[dashed, thick]
    (7.35,1.45) ellipse (3.1 and 1.45);
\node at (7.35,3.2) {$X_d$};

\draw[thick]
    (7.15,-1.15) ellipse (2.9 and 1.55);
\node at (7.15,-3.0) {$Y_{d'}$};


\draw[thick]
    (3.55,0.05) ellipse (7.7 and 4.35);

\node at (3.55,4.05) {$\mathcal{Q}_a$};

\end{tikzpicture}
\caption{The sets $X_c,X_d,Y_{c'}$, and $Y_{d'}$. Only edges internal to $\mc{Q}_a^X$ or $\mc{Q}_a^Y$ are pictured, and the graphs $X$ (dashed) and $Y$ (solid) are overlaid.}
\label{fig:XcYc}
\end{figure}
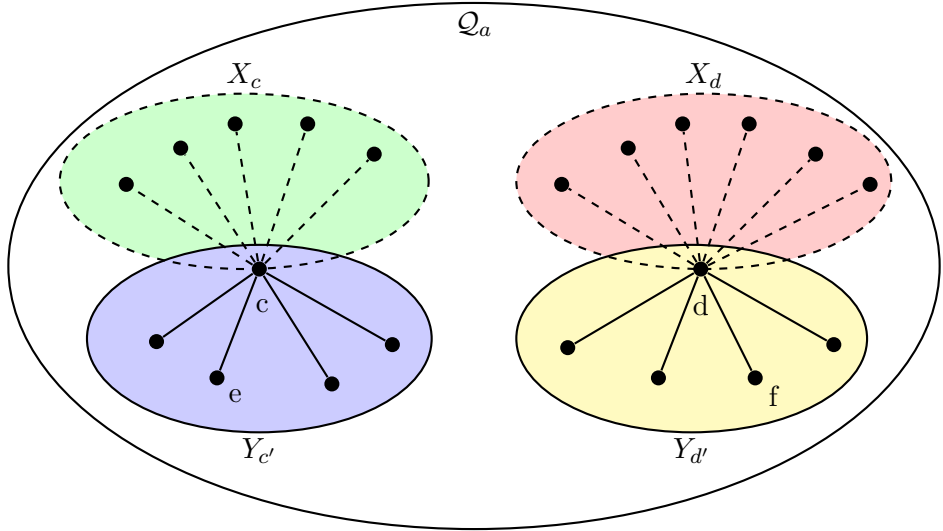
If either of the intersections $X_c \cap X_d$ or $Y_{c'} \cap Y_{d'}$ is nonempty, then there exists a path of length at most $2$ from $c$ to $d$ through $X|_{\mc{Q}_a^X}$ or from $c'$ to $d'$ in $Y|_{\mc{Q}_a^Y}$. By \Cref{lem:Qa swap}, $d_{FS}(\si,\si\circ(cd))\le 9$ here. \textbf{We may assume $X_c\cap X_d = \emptyset$ and $Y_{c'}\cap Y_{d'} = \emptyset$ for the remainder of the lemma.} By applying \Cref{lem:Qa internal edge} to $(c,c'),(d,d')\in \mc{Q}_a$, we find $e',f'\in \mc{Q}_a^Y$ such that $\{c',e'\},\{d',f'\}\in E(Y)$. Because $Y_{c'}$ and $Y_{d'}$ are disjoint, $e'\neq f'$. Recall from \Cref{lem:ban to S} that we may assume that no pair in $\mc{Q}_a$ is adjacent in both $X$ and $Y$ to a pair in $\mc{S}$. In particular, this holds for $(c,c'),(d,d'),(e,e'),(f,f')\in \mc{Q}_a$. 
\begin{claim}\label{claim:2 adjacencies of g}
    In the present setting, we may reduce to the case where any pair $(g,g')$ with $g'\in N(\mc{Q}_b^{*Y})$ is adjacent to at most two of $c,c',e,$ and $e'$.
\end{claim}
\begin{proof}
    Suppose that some pair $(g,g')$ with $g'\in N(\mc{Q}_b^{*Y})$ is adjacent in both $X$ and $Y$ to  $(c,c')$ (a symmetric argument holds for $(e,e')$ instead of $(c,c')$). Since $g'\in N(\mc{Q}_b^{*Y})$, there exists $(h,h')\in \mc{Q}_b^{*Y}$ such that $h'$ is adjacent to $g'$. Moreover, since $(g,g')\notin \mc{S}$ by \Cref{lem:ban to S}, $g'$ cannot have more than $1$ adjacency in $\mc{Q}_a^{*Y}$, so it must be that $N_Y(g')\cap \mc{Q}_a^{*Y} = \{c'\}$. In particular, $g'$ and $e'$ are nonadjacent. If $g$ and $e$ are adjacent, we have the setup in \Cref{fig:g in NXe}, and the swaps 
\[(cg)(eg)(ae)(ac)(bh)(ab)(ah)(bc)(be)(eg)(cg) =(ab)\]
demonstrate that $d_{FS}(\si,\si\circ(ab))\le 11$ in this case. If $b=h$, omit $(bh)$ and $(ah)$ to swap $a$ and $b$ in $9$ moves.
\begin{figure}[h!]
        \centering
        \begin{tikzpicture}[
    every node/.style={circle, draw, inner sep=1.5pt},
    scale=1
]

\node (a) at (-2,2) {$a$};
\node (b) at ( 2,2) {$b$};
\node (c) at ( -2,0) {$c$};
\node (e) at ( -1.2,0) {$e$};
\node (h) at ( 2,0) {$h$};
\node (g) at ( 0,-2) {$g$};

\draw[thick] (a) -- (c);
\draw[thick] (a) -- (e);
\draw[thick] (b) -- (h);
\draw[thick] (c) -- (g);
\draw[thick] (c) -- (e);
\draw[thick] (h) -- (g);

\draw[dashed, thick] (a) -- (b);
\draw[dashed, thick] (a) -- (h);
\draw[dashed, thick] (a) to[out=-100,in=100] (c);
\draw[dashed, thick] (a)  to[out=-80,in=120]  (e);
\draw[dashed, thick] (b) -- (c);
\draw[dashed, thick] (b) to[out=-75,in=75] (h);
\draw[dashed, thick] (b) -- (e);
\draw[dashed, thick] (c) to[out=-55,in=145] (g);
\draw[dashed, thick] (g) -- (e);

\end{tikzpicture}
\caption{The case $g\in N_X(e)$.}
        \label{fig:g in NXe}
\end{figure}
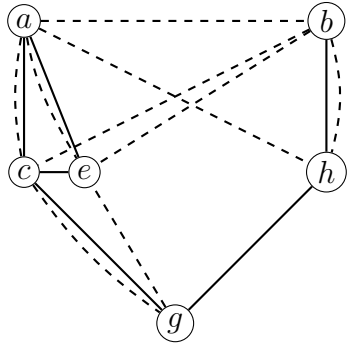
Thus, we reduce to the case where $e$ and $g$ are nonadjacent. We have shown that any pair $(g,g')$ with $g'\in N(\mc{Q}_b^{*Y})$ is adjacent in $X$ and $Y$ to $(c,c')$, then $(g,g')$ is nonadjacent in both $X$ and $Y$ to $(e,e')$. By a symmetric argument, if $(g,g')$ is adjacent to $e$ and $e'$, then we reduce to the case where $(g,g')$ is nonadjacent to both $c$ and $c'$. If $(g,g')$ is not adjacent to both $c$ and $c'$ nor adjacent to both $e$ and $e'$, then it is adjacent to at most $2$ of $c,c',e$, and $e'$. Thus, we may reduce to the case where all $(g,g')$ with $g'\in N(\mc{Q}_b^{*Y})$ are adjacent to at most two of $c,c',e$, and $e'$, proving the claim.
\end{proof}

We use this claim to limit the number of adjacencies (in $X$ or $Y$) between pairs in $\mc{Q}_a$ and pairs whose $Y$ coordinate lies in $N(\mc{Q}_b^{*Y})$. Note that  $\mc{Q}_a^Y$ and $N(\mc{Q}_b^{*Y})$ are disjoint by \Cref{lem:ban to S}. We define $\mc{T}$ to be the set of pairs whose $Y$ coordinate lies in the complement of $\mc{Q}_a^{*Y}\cup N(\mc{Q}_b^{*Y})$:
\begin{equation}\label{eq:def of T}
    \mc{T}:=\{(t,t'): t'\notin  \mc{Q}_a^{*Y}\cup N(\mc{Q}_b^{*Y})\}.
\end{equation}
Recall from \eqref{eq:bound on Qa, Qa*} that $|\mc{Q}_a^*|\ge \dd(X)+\dd(Y)-n+1$. Since $\mc{Q}_a^{*Y}$ and $N(\mc{Q}_b^{*Y})$ are disjoint, we bound the size of $\mc{T}$.  Thus,
\begin{align}\label{eq:upper bound T}
    |\mc{T}|&\le n-|\mc{Q}_a^*|-|N(\mc{Q}_b^{*Y})| \\
    \nonumber&\le  2n-\dd(X)-2\dd(Y)-2.\\
    \nonumber&\le |\mc{Q}_a^*|-3.
\end{align}
By construction, $N(\mc{Q}_b^{*Y})^C = \mc{Q}_a^{*Y}\sqcup \mc{T}^Y$. Recall that $|N(\mc{Q}_b^{*Y})|\ge \dd(Y)+1$ by \eqref{eq:NQb*Y}, so we may obtain
\begin{equation}\label{eq:Qa*+T}
    |\mc{Q}_a^*|+|\mc{T}| \le n-|N(\mc{Q}_b^{*Y})|\le n-\dd(Y)-1 \le 2\dd(X)+2\dd(Y)-2n-1.
\end{equation}
We let $(\# c)$ and $(\# c)|_{\mc{Q}_a}$ denote the number of adjacencies of $(c,c')$ to pairs in $\mc{Q}_a^*\sqcup \mc{T}$ and $\mc{Q}_a$, respectively:
\begin{align}\label{eq:num c}
    (\# c) &:= |\si(N_X(c))\cap (\mc{Q}_a^{*Y}\sqcup \mc{T}^{Y})|+|N_Y(c')\cap (\mc{Q}_a^{*Y}\sqcup \mc{T}^{Y})|.\\
    \nonumber(\# c)|_{\mc{Q}_a} &:= |\si(N_X(c))\cap \mc{Q}_a^Y|+|N_Y(c')\cap \mc{Q}_a^Y|=|X_c|+|Y_{c'}|.
\end{align}
We define $(\# d),(\# e),(\# f)$ and their counterparts restricted to $\mc{Q}_a$ similarly. 
Recall that by \Cref{claim:2 adjacencies of g} that pairs in $N(\mc{Q}_b^{*Y})$ have at most two adjacencies to $c,c',e$ and $e'$. Combined, $(c,c')$ and $(e,e')$ have at least $2\dd(X)+2\dd(Y)$ adjacencies. Subtracting at most two adjacencies for each element of $N(\mc{Q}_b^{*Y})$, we obtain:
\begin{align}
\label{eq:c+e}(\#c)+(\#e) &\ge 2\dd(X)+2\dd(Y)-2|N(\mc{Q}_b^{*Y})|\\
\nonumber&=2\dd(X)+2\dd(Y)-2n+2|\mc{Q}_a^*|+2|\mc{T}|.
\end{align}
Moreover, elements of $\mc{T}$ may (trivially) be adjacent to at most four vertices of $c,c',e,$ and $e'$. Since we are restricting to $\mc{Q}_a$ which does not include $(a,a')$, we subtract another at most four adjacencies from $(a,a')$.  This yields the following bound on $(\# c)|_{\mc{Q}_a}+(\#e)|_{\mc{Q}_a}$:
\begin{align}\label{eq:cQa+eQa}
      (\# c)|_{\mc{Q}_a}+(\#e)|_{\mc{Q}_a}&\ge  2\dd(X)+2\dd(Y)-2|N(\mc{Q}_b^{*Y})|-4|\mc{T}|-4\\
    \nonumber&\ge 2|\mc{Q}_a|+2\dd(X)+2\dd(Y)-2n-2|\mc{T}|-2\\
    \nonumber &\ge 2|\mc{Q}_a|+2,
\end{align}
where the second line follows because $|N(\mc{Q}_b^{*Y})|+|\mc{T}|+|\mc{Q}_a|+1 = n$, and the last line follows by the bound on $|\mc{T}|$ from the second line of \eqref{eq:upper bound T}. In particular, $(\# c)|_{\mc{Q}_a}+(\#e)|_{\mc{Q}_a}>2|\mc{Q}_a|$. Analogous inequalities to \eqref{eq:c+e} and \eqref{eq:cQa+eQa} hold for $(\#d)+(\#f)$ and $(\#d)|_{\mc{Q}_a}+(\#f)|_{\mc{Q}_a}$. Recall the definition of $(\#c)|_{\mc{Q}_a}$ from \eqref{eq:num c}. We use the trivial bound on subsets $|X_c\cup X_e|\ge \frac12|X_c|+\frac12|X_e|$ to produce the following inequalities:
\begin{align}\label{eq:Xc cup Xe + Yc cup Ye}
    |X_c\cup X_e|+|Y_{c'}\cup Y_{e'}|&\ge \frac12\Big[(\# c)|_{\mc{Q}_a}+(\#e)|_{\mc{Q}_a}\Big]>|\mc{Q}_a|;\\
    \nonumber|X_d\cup X_f|+|Y_{d'}\cup Y_{f'}|&\ge \frac12\Big[(\# d)|_{\mc{Q}_a}+(\#f)|_{\mc{Q}_a}\Big]>|\mc{Q}_a|.
\end{align}
This implies that either $(X_c\cup X_e)\cap (X_d\cup X_f)$ or $(Y_{c'}\cup Y_{e'})\cap (Y_{d'}\cup Y_{f'})$ is nonempty, since they are subsets of $\mc{Q}_a^X$ and $\mc{Q}_a^Y$, respectively. If $(Y_{c'}\cup Y_{e'})\cap (Y_{d'}\cup Y_{f'})$ is nonempty, then there exists vertex $v'$ in their intersection. Since $c'$ and $e'$ are adjacent and $d'$ and $f'$ are adjacent, there exists a path of length at most $4$ in $Y|_{\mc{Q}_a^Y}$ from $c'$ to $d'$ that has $v'$ as an intermediate vertex. By \Cref{lem:Qa swap}, $d_{FS}(\si,\si\circ(cd))\le 21$. \textbf{Thus, we may assume $(Y_{c'}\cup Y_{e'})\cap (Y_{d'}\cup Y_{f'}) = \emptyset$ and $(X_c\cup X_e)\cap (X_d\cup X_f)\neq \emptyset$ for the remainder of the lemma.}

We consider cases by which of the sets $X_c\cap X_e$ and $X_d\cap X_f$ are empty, and we first rule out the case where they are both empty. Since we now assume that $(Y_{c'}\cup Y_{e'})\cap (Y_{d'}\cup Y_{f'}) = \emptyset$, it must be that
\begin{equation}\label{eq:Yc+Yd+Ye+Yf}
    |Y_{c'}|+|Y_{d'}|+|Y_{e'}|+|Y_{f'}|\le 2|Y_{c'}\cup Y_{e'}|+2|Y_{d'}\cup Y_{f'}|\le 2|\mc{Q}_a|.
\end{equation}
However, adding the inequalities in \eqref{eq:cQa+eQa} for $(\# c)|_{\mc{Q}_a}+(\#e)|_{\mc{Q}_a}$ and $(\# d)|_{\mc{Q}_a}+(\#f)|_{\mc{Q}_a}$ and bounding the $Y$ contribution by \eqref{eq:Yc+Yd+Ye+Yf}, we find that:
\begin{align}\label{eq:Xc+Xd+Xe+Xf}
    &|X_c|+|X_d|+|X_e|+|X_f| \\
    \nonumber&= (\# c)|_{\mc{Q}_a}+(\#e)|_{\mc{Q}_a}+(\# d)|_{\mc{Q}_a}+(\#f)|_{\mc{Q}_a}-|Y_{c'}|-|Y_{d'}|-|Y_{e'}|-|Y_{f'}|\\
    \nonumber&\ge 2|\mc{Q}_a|+4.
\end{align}
This implies that either $|X_c|+|X_e|>|\mc{Q}_a|$ or $|X_d|+|X_f|>|\mc{Q}_a|$, so at least one of $X_c\cap X_e$ and $X_d\cap X_f$ is nonempty.

Moreover, if both $X_c\cap X_e$ and $X_d\cap X_f$ are nonempty, then there exist paths of length at most $2$ from $c$ to $e$ and from $d$ to $f$ in $X|_{\mc{Q}_a^X}$. Recall that we may assume $(X_c\cup X_e)\cap (X_d\cup X_f)\neq \emptyset$. Thus, there exists some vertex $v$ in their intersection, and there exists a path in $X|_{\mc{Q}_a^X}$ from $c$ to $d$ through $v$ of length at most $6$. By \Cref{lem:Qa swap}, $d_{FS}(\si,\si\circ(cd))\le 33$.

Thus, we reduce to the case where exactly one of $X_c\cap X_e,X_d\cap X_f$ is empty. \textbf{We may assume $X_c \cap X_e = \emptyset$ and $X_d\cap X_f\neq\emptyset$ without loss of generality.} By the lower bound \eqref{eq:Xc+Xd+Xe+Xf}, there exists $(g,g')\in \mc{Q}_a$ such that $g$ is in at least $3$ of $X_c,X_d,X_e$, and $X_f$. If $g\in X_c$, then because $g\in X_d$ or $g\in X_f$ as well (and since $X_d\cap X_f\neq\emptyset$), there exists a path of length at most $4$ in $X|_{\mc{Q}_a^X}$ from $c$ to $d$. By \Cref{lem:Qa swap}, $d_{FS}(\si,\si\circ(cd))\le 21$. Thus, we may assume $g\notin X_c$ and $X_d\cap X_e\cap X_f$ is nonempty.

Let us now suppose that there exists $(t,t')\in \mc{T}$ adjacent to $c,c',d,d',f,$ and $f'$. Recall that $d'$ and $f'$ are adjacent, so the setup is given by \Cref{fig:T to swap}. The sequence
\[(ad)(af)(tf)(af)(ac)(ad)(tc)=(cd)\]
shows that $d_{FS}(\si,\si\circ(cd))\le 7$.
\begin{figure}[h!]
        \centering
        \begin{tikzpicture}[
    every node/.style={circle, draw, inner sep=1.5pt},
    scale=1
]

\node (a) at (0,2) {$a$};
\node (c) at ( -2,0) {$c$};
\node (d) at ( 0,0) {$d$};
\node (f) at ( 2,0) {$f$};
\node (t) at ( 0,-2) {$t$};

\draw[thick] (a) -- (c);
\draw[thick] (a) -- (d);
\draw[thick] (a) -- (f);
\draw[thick] (t) -- (c);
\draw[thick] (t) -- (d);
\draw[thick] (t) -- (f);
\draw[thick] (f) -- (d);

\draw[dashed, thick] (a) to[out=-100,in=100] (d);
\draw[dashed, thick] (a) to[out=-30,in=120] (f);
\draw[dashed, thick] (a) to[out=-150,in=70] (c);
\draw[dashed, thick] (t) to[out=100,in=-100] (d);
\draw[dashed, thick] (t) to[out=30,in=-120] (f);
\draw[dashed, thick] (t) to[out=150,in=-70] (c);

\end{tikzpicture}
\caption{Swapping $c$ and $d$ using a common friend $(t,T)\in \mc{T}$.}
    \label{fig:T to swap}
\end{figure}
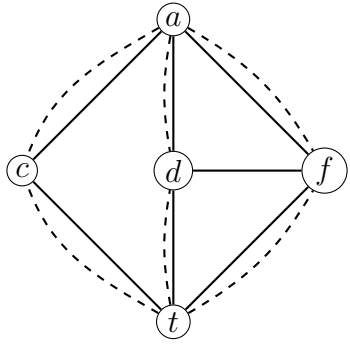
We may assume no such pair $(t,t')$ exists, so every pair in $\mc{T}$ is nonadjacent to at least one of $c,c',d,d',f$, or $f'$. Let 
\[c_{\mc{T}} = |\{(t,t')\in \mc{T}:\{c,t\}\notin E(X) \text{ or }\{c',t'\}\notin E(Y)\}|.\]
Let $d_{\mc{T}},f_{\mc{T}}$ be defined analogously for $d,d'$ and $f,f'$ respectively. Since we may assume each pair in $\mc{T}$ has a nonadjacency to at least one of $c,c',d,d',f$, and $f'$, every pair contributes to at least one of $c_{\mc{T}},d_{\mc{T}}$, or $f_{\mc{T}}$. Thus,
\begin{equation}\label{eq:c_T}
    c_{\mc{T}}+d_{\mc{T}}+f_{\mc{T}}\ge |\mc{T}|.
\end{equation}
Note that $(c,c')$ has at least $c_{\mc{T}}$ nonadjacencies (in $X$ and $Y$ combined) to $\mc{T}$, and analogous statements hold for $(d,d')$ and $(f,f')$. The bound in \eqref{eq:cQa+eQa} is obtained by subtracting $4|\mc{T}|$ from \eqref{eq:c+e}, which assumes that every pair $(t,t')\in \mc{T}$ is adjacent to $c,c',e$, and $e'$. However, we now assume that $c$ and $c'$ are nonadjacent to a total of $c_\mc{T}$ pairs in $\mc{T}$ and we have analogous statements for $(d,d')$ and $(f,f')$, so we obtain the following bounds by restricting \eqref{eq:c+e} to $\mc{Q}_a$:
\begin{align}\label{eq:cQa and eQa after T}
    (\# c)|_{\mc{Q}_a}+(\#e)|_{\mc{Q}_a}&\ge 2|\mc{Q}_a|+2\dd(X)+2\dd(Y)-2n-2|\mc{T}|+c_{\mc{T}}-2,\\
    \nonumber(\# d)|_{\mc{Q}_a}+(\#f)|_{\mc{Q}_a}&\ge 2|\mc{Q}_a|+2\dd(X)+2\dd(Y)-2n-2|\mc{T}|+d_{\mc{T}}+f_{\mc{T}}-2.
\end{align}
Since $X_c,X_e$ are disjoint,  $|X_c|+|X_e|\le |\mc{Q}_a|$. Trivially, $|X_d|+|X_f|\le 2|\mc{Q}_a|$ since they are subsets of $\mc{Q}_a^X$. Subtracting these upper bounds on $|X_c|+|X_e|$ and $|X_d|+|X_f|$ from \eqref{eq:cQa and eQa after T}, we obtain:
\begin{align}\label{eq:second to last Yc bounds}
    |Y_{c'}|+|Y_{e'}|&\ge |\mc{Q}_a|+2\dd(X)+2\dd(Y)-2n-2|\mc{T}|+c_{\mc{T}}-2,\\
    \nonumber|Y_{d'}|+|Y_{f'}|&\ge 2\dd(X)+2\dd(Y)-2n-2|\mc{T}|+d_{\mc{T}}+f_{\mc{T}}-2.
\end{align}
We appeal to the bounds of the form $|Y_{c'}\cup Y_{e'}|\ge \frac12|Y_{c'}|+\frac12|Y_{e'}|$ to make use of \eqref{eq:second to last Yc bounds}. We make use of upper bounds on $|\mc{T}|$ from \eqref{eq:upper bound T} and on $|\mc{Q}_a|+|\mc{T}|$ from \eqref{eq:Qa*+T} (with $\mc{Q}_a$ instead of $\mc{Q}_a^*$), and the lower bound on $c_{\mc{T}}+d_{\mc{T}}+f_{\mc{T}}$ in \eqref{eq:c_T} to show:
\begin{align}\label{eq:final Yc}
    |Y_{c'}\cup Y_{e'}|+ |Y_{d'}\cup Y_{f'}|&\ge \frac12\Big(|Y_{c'}|+|Y_{d'}|+|Y_{e'}|+|Y_{f'}|\Big)\\
    &\ge \frac12|\mc{Q}_a|+2\dd(X)+2\dd(Y)-2n-\nonumber 2|\mc{T}|+\frac12(c_{\mc{T}}+d_{\mc{T}}+f_{\mc{T}})-2\\
    &\ge \frac12|\mc{Q}_a|+|\mc{Q}_a|+|\mc{T}|+2-\nonumber 2|\mc{T}|+\frac12|\mc{T}|-2\\
    \nonumber&=|\mc{Q}_a| +\frac12|\mc{Q}_a|-\frac12|\mc{T}|\\
    \nonumber&>|\mc{Q}_a|.
\end{align}
It follows from \eqref{eq:final Yc} that $Y_{c'}\cup Y_{e'},Y_{d'}\cup Y_{f'}$ intersect, which contradicts our reduction that $(Y_{c'}\cup Y_{e'})\cap(Y_{d'}\cup Y_{f'}) = \emptyset$. Thus, after our reductions, there always exists a pair in $\mc{T}$ adjacent to $c,c',d,d',f$, and $f'$. 

We now justify the lemma for $d_{FS}(\si\circ(ab),\si\circ(ab)\circ (cd))$. Recall that $a$ and $b$ are both adjacent to $\mc{Q}_a^X\cup \mc{Q}_b^X$, and these are the only adjacencies of $a$ and $b$ that are required from the proof. In particular, the only elements of the proof that require structure outside $\mc{Q}_a$ are the bounds on $(\#c)+(\# e)$ and the use of $|\mc{T}|$. In both cases, there are no adjacencies required of $a$ or $b$ outside of $\mc{Q}_a^X\cup \mc{Q}_b^X$. Thus, $d_{FS}(\si\circ(ab),\si\circ(ab)\circ (cd))\le 33$, concluding the proof.
\end{proof}

We leverage this lemma to allow us to perform a sequence of swaps between elements in $\mc{Q}_a^X$. If any of these swaps cannot be performed in $33$ moves, then by \Cref{lem:swap cd in qa} we may swap $a$ and $b$ in $\max\{\al,21\}$ moves from this configuration. After swapping $a$ and $b$, we may undo each of the swaps in $\mc{Q}_a$ in $33$ moves each by case 2 of \Cref{lem:swap cd in qa}.

We now introduce the framing that will finish the argument. Recall the definition of $\mc{S}$ from \eqref{eq:def of S} as consisting of houses with at least two adjacencies in $\mc{Q}_a^{*Y}$ and at least one adjacency in $\mc{Q}_b^{*Y}$. We will examine the shortest path through $X$ from $\mc{Q}_a^X$ to $\mc{S}^X$, and we will denote the length of this path as $d_X(\mc{Q}_a^X,\mc{S}^X)$. Since $\mc{Q}_a^{*X}\cup \mc{S}^X$ is large, we prove that vertices outside this set have at least three adjacencies in $\mc{Q}_a^{*X}\cup \mc{S}^X$.

\begin{lemma}\label{lem:c at least 3 neighbors}
    In the present setting, let $c\in V(X)\sm (\mc{Q}_a^{*X}\sqcup \mc{S}^X)$. Then $c$ is adjacent to at least three vertices in $\mc{Q}_a^{*X}\sqcup \mc{S}^X$. Moreover, $d_X(\mc{Q}_a^X,\mc{S}^X)\le 3$.
\end{lemma}

\begin{proof}
    By \Cref{lem:ban to S} we may assume $\mc{Q}_a^{*Y}$ and$\mc{Q}_b^{*Y}$ are nonadjacent, so $\mc{Q}_a^*$ and $\mc{S}$ are disjoint. We lower bound the sizes of $|\mc{Q}_a^*|$ and $|\mc{S}|$ using \eqref{eq:bound on Qa, Qa*} and \eqref{eq:lower bound S}, respectively: 
    \begin{equation}\label{eq:QaX and SX}
        |\mc{Q}_a^{*X}\sqcup \mc{S}^X|\ge \dd(X)+3\dd(Y)-2n+2.
    \end{equation}
    Note that $|N_X[c]|\ge \dd(X)+1$, so by inclusion-exclusion using \eqref{eq:QaX and SX}:
    \begin{align}\label{eq:c neighbors in QaX or SX}
        |N_X[c]\cap (\mc{Q}_a^{*X}\sqcup \mc{S}^X)|&\ge  2\dd(X)+3\dd(Y)+3-3n\ge 3.
    \end{align}
    Since $c\notin \mc{Q}_a^{*X}\sqcup \mc{S}^X$, \eqref{eq:c neighbors in QaX or SX} implies that $c$ is adjacent to at least three elements of $\mc{Q}_a^{*X}\sqcup \mc{S}^X$.

    Since every vertex outside $\mc{Q}_a^{X}\sqcup \mc{S}^X$ is adjacent to $\mc{Q}_a^{X}\sqcup \mc{S}^X$, we have $N[\mc{Q}_a^X]\cup N[\mc{S}^X] = V(X)$. If $N[\mc{Q}_a^X]\cap N[\mc{S}^X]\neq\emptyset$, then $d_X(\mc{Q}_a^X,\mc{S}^X)\le 2$. If $N[\mc{Q}_a^X]\cap N[\mc{S}^X]=\emptyset$, then because $X$ is connected there exists an edge between $N[\mc{Q}_a^X]$ and $N[\mc{S}^X]$. This edge yields a path of length $3$, so $d_X(\mc{Q}_a^X,\mc{S}^X)\le 3$; see \Cref{fig:Layout of X}.
\begin{figure}[h!]
    \centering
    
\begin{tikzpicture}[scale=1, every node/.style={font=\large}]
  \draw (-5.5,0) circle (1.3);
  \node at (-5.5,0) {$\mc{Q}_a^X$};

  \draw (-2,0) circle (1.3);
  \node[font=\small] at (-2,0) {$N(\mc{Q}_a^{X})\back \mc{Q}_a^{X}$};

  \draw (0,-2) -- (0,2);

  \draw (2,0) circle (1.3);
  \node[font=\small] at (2,0) {$N(\mc{S}^X)\back \mc{S}^X$};

  \draw (5.5,0) circle (1.3);
  \node[align=center] at (5.5,0)
    {$\mc{S}^X$};

  \draw[dashed] (-4.2,0) -- (-3.3,0);
  \draw[dashed] (3.3,0) -- (4.2,0);
\end{tikzpicture}
\caption{Finding a path of length $\le 3$ between $\mc{Q}_a^{X}$ and $\mc{S}^X$.}
    \label{fig:Layout of X}
\end{figure}
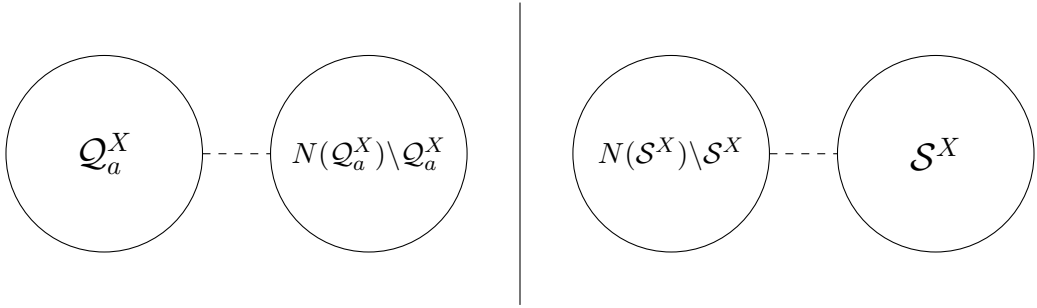
\end{proof}

A path of small length between $\mc{Q}_a^X$ and $\mc{S}^X$ is key to swapping $a$ and $b$. However, the nature of our solution will be different if either $a$ or $b$ is included in this path. We first consider the more general case where there exists a minimal length path from $\mc{Q}_a^X$ to $\mc{S}^X$ that does not include $a$ or $b$. The case $d_X(\mc{Q}_a^X,\mc{S}^X)=3$ is most difficult, so we first resolve the cases $d_X(\mc{Q}_a^X,\mc{S}^X)=1$ and $d_X(\mc{Q}_a^X,\mc{S}^X)=2$. Recall that $\al$ is the constant given by \Cref{lem:starp bound}.

\begin{lemma}\label{lem:dX=1 or 2}
    In the present setting, suppose there exists a path from $\mc{Q}_a^X$ to $\mc{S}^X$ of minimal length that does not include $a$ or $b$. If $d_X(\mc{Q}_a^X,\mc{S}^X)\le 2$, then 
    \[d_{FS}(\si,\si\circ(ab))\le 76+\max\{3\al,219\}.\]
\end{lemma}

\begin{proof}
We let $(c,c')\in \mc{Q}_a$ and $(d,d')\in \mc{S}$ be the endpoints of the minimum length path from $\mc{Q}_a^X$ to $\mc{S}^X$ throughout this lemma. Similarly, we let $e',f'\in \mc{Q}_a^Y$ and $g'\in \mc{Q}_b^Y$ denote the neighbors of $d'$ guaranteed by $\mc{S}$. We assume $e\neq a$ and $f\neq c$.\\
\textbf{Case 1: $d_X(\mc{Q}_a^X,\mc{S}^X)=1$.}\\
    Since $d_X(\mc{Q}_a^X,\mc{S}^X)=1$, $c$ and $d$ are adjacent. We could have $e=c$, $f=a$, or $g=b$. We will demonstrate that each of these cases follow from the general case. If all vertices are distinct, we obtain the setup in \Cref{fig:dX=1}. 
    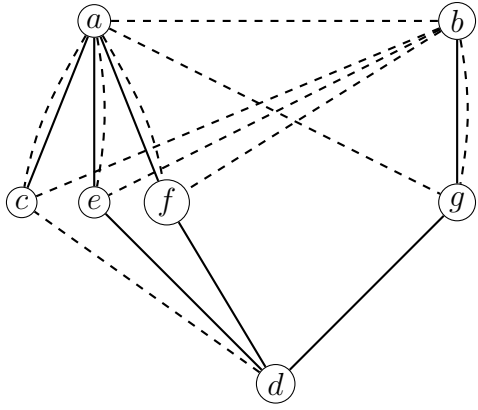
\begin{figure}[h!]
        \centering
        
        \begin{tikzpicture}[
    every node/.style={circle, draw, inner sep=1.5pt},
    scale=1.2
]

\node (a) at (-2,2) {$a$};
\node (b) at ( 2,2) {$b$};
\node (qa) at (-2.8,0) {$c$};
\node (qa') at ( -2,0) {$e$};
\node (qa'') at ( -1.2,0) {$f$};
\node (qb) at ( 2,0) {$g$};
\node (c) at ( 0,-2) {$d$};

\draw[thick] (a) -- (qa);
\draw[thick] (a) -- (qa');
\draw[thick] (a) -- (qa'');
\draw[thick] (b) -- (qb);
\draw[thick] (c) -- (qa'');
\draw[thick] (c) -- (qa');
\draw[thick] (c) -- (qb);

\draw[dashed, thick] (a) -- (b);
\draw[dashed, thick] (a) to[out=-120,in=80] (qa);
\draw[dashed, thick] (a)  to[out=-80,in=80]  (qa');
\draw[dashed, thick] (a) to[out=-60,in=100] (qa'');
\draw[dashed, thick] (b) -- (qa);
\draw[dashed, thick] (b) -- (qa');
\draw[dashed, thick] (b) -- (qa'');
\draw[dashed, thick] (a) -- (qb);
\draw[dashed, thick] (b) to[out=-80,in=80]  (qb);
\draw[dashed, thick] (c) -- (qa);

\end{tikzpicture}
\caption{The case $d_X = 1$. We have $c,e,f\in \mc{Q}_a^X$, $g\in \mc{Q}_b^X$, and $d\in \mc{S}^X$.}
        \label{fig:dX=1}
\end{figure}
By \Cref{lem:swap cd in qa} we can perform $(ce)^{33}$ in $33$ moves. After swapping $c$ and $e$, see that removing $(e,c')$ from \Cref{fig:dX=1} yields a setup identical to the right side of \Cref{fig:QaQbY}. Thus, we can swap $(ab)$ from $\si\circ(ce)$ in $9$ moves. By \Cref{lem:swap cd in qa}, we may swap $(ce)$ from configuration $\si\circ(ce)\circ(ab)$ in $33$ moves. Thus, $d_{FS}(\si,\si\circ(ab))\le 33+9+33 = 75$. If $c=e$ and $a,f,b,$ and $g$ are distinct, then we immediately have the setup in the right side of \Cref{fig:QaQbY}, and $d_{FS}(\si,\si\circ(ab))\le 9$. 

If either $f=a$ or $g=b$ or both, our sequence of swaps only simplifies. In all cases our goal is to move $c$ to $d'$ by using $e'$ as an intermediate step if necessary. Once $c$ is in $d'$, we leverage the fact that $c\in \mc{Q}_a^X$ implies that $c$ is adjacent to $a$ and $b$ to swap $a$ and $b$. In particular, we can imagine moving $a$ through $f$, then $c$, then $g$, and finally to $b$. Then we may move $b$ back through $g$, then $c$, and finally $f$. If either $f=a$ or $g=b$, then this path is shortened, but the strategy is unchanged. Thus, $d_{FS}(\si,\si\circ(ab))\le 33+9+33 = 75$ when $d_X = 1$.

\textbf{Case 2: $d_X(\mc{Q}_a^X,\mc{S}^X)=2$.}\\
Since $d_X(\mc{Q}_a^X,\mc{S}^X)=2$, there exists $(h,h')\notin \mc{Q}_a^*\cup \mc{S}$ such that $c,h,d$ is a path in $X$. Since there exists a minimal length path not containing $a$ or $b$, let $h\neq a,b$. Thus, $h$ is distinct from every vertex in our setup, except possibly $g$. If $h=g\in\mc{Q}_b^X$, the setup is depicted on the left side of \Cref{fig:dX=2}. By \Cref{lem:Qa internal edge} let $i'\in \mc{Q}_b^Y$ be adjacent to $h'$, and by \Cref{lem:Qa swap}, we may swap $(hi)^{3}$. For notational convenience, relabel $i=g$ to match the general case where $g$ and $h$ are distinct. We have obtained the setup depicted in the middle of \Cref{fig:dX=2}, and the rest of the proof is identical, except that we undo the swap $(hi)^3$ (or $(hg)^3$ after relabeling) at the end.
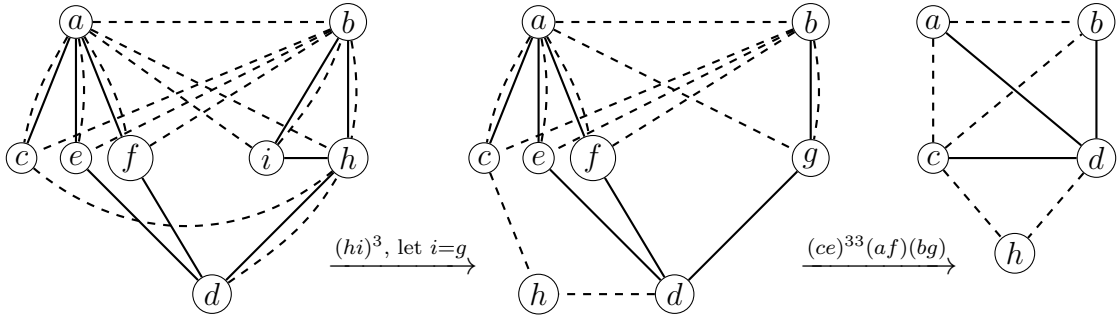
\begin{figure}[h!]
        \centering
        
        \begin{tikzpicture}[
    every node/.style={circle, draw, inner sep=1.5pt},
    scale=.9
]
\node (a1) at (-8.8,2) {$a$};
\node (b1) at ( -4.8,2) {$b$};
\node (qa1) at (-9.6,0) {$c$};
\node (qa'1) at ( -8.8,0) {$e$};
\node (qa''1) at ( -8,0) {$f$};
\node (qb1) at ( -4.8,0) {$h$};
\node (qb2) at ( -6,0) {$i$};
\node (c1) at ( -6.8,-2) {$d$};

\draw[thick] (a1) -- (qa1);
\draw[thick] (a1) -- (qa'1);
\draw[thick] (a1) -- (qa''1);
\draw[thick] (b1) -- (qb1);
\draw[thick] (c1) -- (qa''1);
\draw[thick] (c1) -- (qa'1);
\draw[thick] (c1) -- (qb1);
\draw[thick] (b1) -- (qb2);
\draw[thick] (qb1) -- (qb2);

\draw[dashed, thick] (a1) -- (b1);
\draw[dashed, thick] (a1) -- (qb2);
\draw[dashed, thick] (qb2) to[out=50,in=-110] (b1);
\draw[dashed, thick] (a1) to[out=-120,in=80] (qa1);
\draw[dashed, thick] (a1)  to[out=-80,in=80]  (qa'1);
\draw[dashed, thick] (a1) to[out=-60,in=100] (qa''1);
\draw[dashed, thick] (b1) -- (qa1);
\draw[dashed, thick] (b1) -- (qa'1);
\draw[dashed, thick] (b1) -- (qa''1);
\draw[dashed, thick] (a1) -- (qb1);
\draw[dashed, thick] (b1) to[out=-80,in=80]  (qb1);
\draw[dashed, thick] (qa1) to[out=-40,in=-140] (qb1);
\draw[dashed, thick] (c1) to[out=35,in=-120] (qb1);

\node (a) at (-2,2) {$a$};
\node (b) at ( 2,2) {$b$};
\node (qa) at (-2.8,0) {$c$};
\node (qa') at ( -2,0) {$e$};
\node (qa'') at ( -1.2,0) {$f$};
\node (qb) at ( 2,0) {$g$};
\node (c) at ( 0,-2) {$d$};
\node (e) at ( -2,-2) {$h$};

\draw[thick] (a) -- (qa);
\draw[thick] (a) -- (qa');
\draw[thick] (a) -- (qa'');
\draw[thick] (b) -- (qb);
\draw[thick] (c) -- (qa'');
\draw[thick] (c) -- (qa');
\draw[thick] (c) -- (qb);

\draw[dashed, thick] (a) -- (b);
\draw[dashed, thick] (a) to[out=-120,in=80] (qa);
\draw[dashed, thick] (a)  to[out=-80,in=80]  (qa');
\draw[dashed, thick] (a) to[out=-60,in=100] (qa'');
\draw[dashed, thick] (b) -- (qa);
\draw[dashed, thick] (b) -- (qa');
\draw[dashed, thick] (b) -- (qa'');
\draw[dashed, thick] (a) -- (qb);
\draw[dashed, thick] (b) to[out=-80,in=80]  (qb);
\draw[dashed, thick] (c) -- (e);
\draw[dashed, thick] (qa) -- (e);

\node (a') at (3.8,2) {$a$};
\node (b') at (6.2,2) {$b$};
\node (c') at (3.8,0) {$c$};
\node (d') at (6.2,0) {$d$};
\node (e') at (5,-1.4) {$h$};

\draw[thick] (a') -- (d');
\draw[thick] (b') -- (d');
\draw[thick] (c') -- (d');

\draw[dashed, thick] (a') -- (b');
\draw[dashed, thick] (a') -- (c');
\draw[dashed, thick] (c') -- (b');
\draw[dashed, thick] (d') -- (e');
\draw[dashed, thick] (c') -- (e');

\node[draw=none, circle=none] at (-4,-1.4)
{$\xrightarrow[]{(hi)^{3},\text{ let }i=g}$};
\node[draw=none, circle=none] at (3,-1.4)
{$\xrightarrow[]{(ce)^{33}(af)(bg)}$};

\end{tikzpicture}
\caption{The case $d_X = 2$. Left: The case $g=h$. We have $c,e,f\in \mc{Q}_a^X$, $h,i\in \mc{Q}_b^X$, and $d\in \mc{S}$. We swap $(hi)^{3}$ to obtain an identical setup to the middle. From the middle, after performing $(ce)^{33}(af)(bg)$ then restricting to $a,b,c,d,h$ we obtain the setup on the right.}
        \label{fig:dX=2}
\end{figure}

We now address the general case where $g\neq h$. By \Cref{lem:swap cd in qa}, we perform $(ce)^{33}$. Next, swap $(af)$ and $(bg)$, each in one move. Restricting our setup to the five vertices $(a,f'),(c,e'),(b,g'),(d,d'),$ and $(h,h')$, we find the setup on the right of \Cref{fig:dX=2}. This setup is identical to \Cref{fig:bangachev extension}, so by \Cref{lem:bangachev with P=1}, we may swap $(ab)$ in at most $\max\{3\al,219\}$ moves. By \Cref{lem:swap cd in qa}, we may then swap $(bf)(ag)(ce)^{33}$. This entire sequence would swap $a$ and $b$ in at most $70+\max\{3\al,219\}$ moves. If it was the case that $g=h$ to begin with, we now have a setup given by the left side of \Cref{fig:dX=2} with the swaps $(ab)(hi)$. We swap $(hi)^3$ using that $a$ is adjacent to $h$ and $i$, and thus $d_{FS}(\si,\si\circ(ab))\le 76+\max\{3\al,219\}$.

If any combination of $c=e$, $f=a$, or $g=b$ hold, then we simply remove the corresponding swaps from the beginning and the end of our sequence. For example, if $f=a$ and $g=b$, then we would omit $(af)$ and $(bg)$ from our preliminary swaps and only perform $(ce)^{33}$. We then have the setup on the right of \Cref{fig:dX=2}, so we swap $a$ and $b$ by \Cref{lem:bangachev with P=1}. Afterwards, we undo our sequence of preliminary swaps, again excluding $(af)$ and $(bg)$, so we only perform $(ce)^{33}$. It follows that any of the equalities $c=e$, $f=a$, or $g=b$ decreases the number of swaps required, so $d_{FS}(\si,\si\circ(ab))\le 76+\max\{3\al,219\}$.
\end{proof} 

Before addressing the case $d_X(\mc{Q}_a^X,\mc{S}^X) = 3$, we stress that the main result is the existence of an absolute constant $C$ such that friends $a$ and $b$ can always be swapped in $C$ moves. The constant
\begin{equation}\label{eq:C}
    C := 228+8D_3^2+\max\{3\al,219\}
\end{equation}
will suffice in the following lemma, where $D_3$ is the constant obtained from \Cref{lem:finite menger}, and $\al$ is the constant from \Cref{lem:starp bound}. 

\begin{lemma}\label{lem:dX=3}
    In the present setting, suppose there exists a path from $\mc{Q}_a^X$ to $\mc{S}^X$ of minimal length that does not include $a$ or $b$. If $d_X(\mc{Q}_a^X,\mc{S}^X)= 3$, then $d_{FS}(\si,\si\circ(ab))\le C$.
\end{lemma}

\begin{proof}
    Let $c,d,e,f$ be a path from $\mc{Q}_a^X$ to $\mc{S}^X$ with $d,e\neq a,b$. Let $g',h'\in \mc{Q}_a^{*Y}$ and $i'\in \mc{Q}_b^{*Y}$ be adjacent to $f'$, as guaranteed by the fact that $f'\in\mc{S}^Y$. Since $c,d,e,f$ is a path of minimal length, $d,e\notin \mc{Q}_a^X$, so $d,e\neq g,h$. Thus, the only duplicate we could have is $i=d$ or $i=e$. 
    
    We will perform a preliminary sequence of swaps that will allow us to assume $i\neq d,e$ going forward. The argument will be symmetric in $d$ and $e$, so let $i=d$ without loss of generality. By \Cref{lem:Qa internal edge}, $i'$ is adjacent to some vertex in $\mc{Q}_b^Y$. If there exists such $i_1'\in \mc{Q}_b^Y$ with $i_1'\neq e'$, then we may swap $(ii_1)^3$. This implies $i_1\neq d,e$, and $i_1$ is in $i'$ in the current setup, so we may relabel $i=i_1$ to match the general case given by the left side of \Cref{fig:length 3 to 2}. At the end, we simply undo the swap $(ii_1)^3$ since $a$ will be in $b'$ and $a$ is adjacent to $i$ and $i_1$. 

    If there does not exist such $i_1'$, then $N_Y(d')\cap \mc{Q}_b^Y = \{e'\}$. Recall from \eqref{eq:Qa+Qb} that  $|\mc{Q}_a^*|+|\mc{Q}_b^*|\ge 2\dd(X)+2\dd(Y)-2n+2$, and that $d'\in \mc{Q}_b^Y$ implies $d'$ is nonadjacent to $\mc{Q}_a^{*Y}$ by \Cref{lem:ban to S}. Thus, of the $\ge 2\dd(X)+2\dd(Y)-2n+1$ other vertices in $\mc{Q}_a^{*Y}\cup \mc{Q}_b^{*Y}$, we have that $d'$ is only adjacent to $b'$ and $e'$. Therefore $d'$ is nonadjacent to at least $2\dd(X)+2\dd(Y)-2n-1\ge n-\dd(Y)-1$ vertices in $\mc{Q}_a^{*Y}\cup \mc{Q}_b^{*Y}$, so $d'$ is adjacent to all vertices in $V(Y)\back(\mc{Q}_a^{*Y}\cup \mc{Q}_b^{*Y})$. Similarly, suppose $e'$ is adjacent to $i_2'\in \mc{Q}_b^Y$ with $i_2'\neq d'$. Leveraging that $e'$ is adjacent to $d'$ and $i_2'$ and that all three vertices lie in $\mc{Q}_b^Y$, we may perform the preliminary sequence $(i_2e)^3(i_2d)^3$. Afterwards, $i_2$ is in 
    $i'$, and the relabeling $i=i_2$ matches the general case. Again, we only need to undo these $6$ preliminary swaps after swapping $a$ and $b$. 

    Suppose there does not exist such $i_1'$ or $i_2'$. Then $d'$ and $e'$ are only adjacent to each other and $b'$ in $\mc{Q}_a^{*Y}\cup \mc{Q}_b^{*Y}$, so they are both adjacent to all vertices in $V(Y)\back(\mc{Q}_a^{*Y}\cup \mc{Q}_b^{*Y})$. In particular, $f'\in \mc{S}^Y$ is adjacent to $d'$ and $e'$. The sequence $(ef)(ag)(ad)(ae)(ab)(bd)(be)(bg)(ef) = (ab)$ shows that $d_{FS}(\si,\si\circ(ab))\le 9$ here. \textbf{Thus, we reduce to the case where $i\neq e,d$, though we add $12$ moves.}
    
    Because $d_X(\mc{Q}_a^X,\mc{S}^X)= 3$, we have $e\notin N_X(\mc{Q}_a^X)$. If we could move $e$ to $f'$ in a sufficiently nice way, then $c,d,e$ is a path of length $2$ to $f'\in \mc{S}^Y$. From here, we use the setup on the left of \Cref{fig:dX=2} to swap $(ab)$, and then return $e$ to $e'$. 

    \begin{claim}\label{claim:ef reduction}
        In the present setting and hypotheses of \Cref{lem:dX=3}, suppose there exists $\tau\in V(\FS)$ such that $\tau(e) = \si(f)$ and such that $\tau$ can be obtained from $\si$ using $L$ swaps in $V(X)\sm\{\mc{Q}_a^{*X}\cup\{b,i\}\}$. Then
        \begin{equation}\label{eq:2si to tau}
            d_{FS}(\si,\si\circ(ab))\le 2L+70+\max\{3\al,219\}.
        \end{equation}
    \end{claim}
     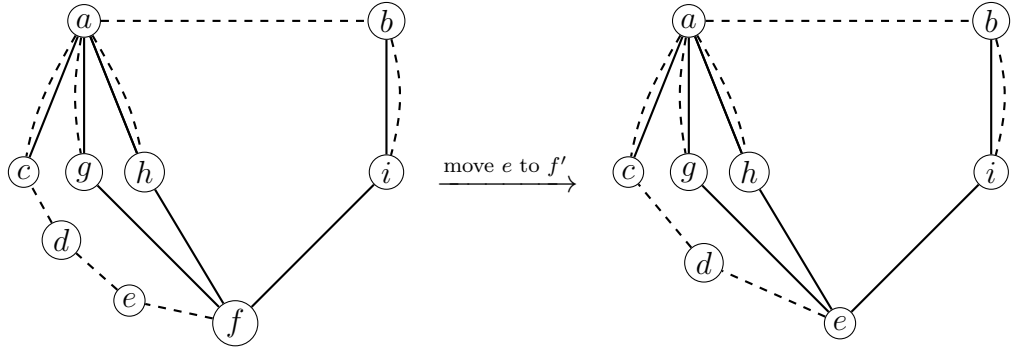
\begin{figure}[h!]
        \centering
        \begin{tikzpicture}[
    every node/.style={circle, draw, inner sep=1.5pt},
    scale=1
]

\node (a) at (-2,2) {$a$};
\node (b) at ( 2,2) {$b$};
\node (c) at (-2.8,0) {$c$};
\node (g) at ( -2,0) {$g$};
\node (h) at ( -1.2,0) {$h$};
\node (m) at ( 2,0) {$i$};
\node (d) at ( -2.3,-.9) {$d$};
\node (e) at ( -1.4,-1.7) {$e$};
\node (f) at ( 0,-2) {$f$};

\draw[thick] (a) -- (c);
\draw[thick] (a) -- (g);
\draw[thick] (a) -- (h);
\draw[thick] (a) -- (h);
\draw[thick] (b) -- (m);
\draw[thick] (f) -- (g);
\draw[thick] (f) -- (h);
\draw[thick] (f) -- (m);

\draw[dashed, thick] (a) -- (b);
\draw[dashed, thick] (a) to[out=-120,in=80] (c);
\draw[dashed, thick] (a) to[out=-100,in=100] (g);
\draw[dashed, thick] (a) to[out=-60,in=100] (h);
\draw[dashed, thick] (b) to[out=-75,in=75] (m);
\draw[dashed, thick] (d) -- (c);
\draw[dashed, thick] (f) -- (e);
\draw[dashed, thick] (e) -- (d);

\node (a') at (6,2) {$a$};
\node (b') at ( 10,2) {$b$};
\node (c') at (5.2,0) {$c$};
\node (g') at ( 6,0) {$g$};
\node (h') at ( 6.8,0) {$h$};
\node (m') at ( 10,0) {$i$};
\node (e') at ( 6.2,-1.2) {$d$};
\node (f') at ( 8,-2) {$e$};

\draw[thick] (a') -- (c');
\draw[thick] (a') -- (g');
\draw[thick] (a') -- (h');
\draw[thick] (a') -- (h');
\draw[thick] (b') -- (m');
\draw[thick] (f') -- (g');
\draw[thick] (f') -- (h');
\draw[thick] (f') -- (m');

\draw[dashed, thick] (a') -- (b');
\draw[dashed, thick] (a') to[out=-120,in=80] (c');
\draw[dashed, thick] (a') to[out=-100,in=100] (g');
\draw[dashed, thick] (a') to[out=-60,in=100] (h');
\draw[dashed, thick] (b') to[out=-75,in=75] (m');
\draw[dashed, thick] (e') -- (c');
\draw[dashed, thick] (f') -- (e');

\node[draw=none, circle=none] at (3.6,0)
{$\xrightarrow[]{\text{move }e\text{ to }f'}$};

\end{tikzpicture}
\caption{The path of length $3$ becomes length $2$ when we move $e$ to $f'$.}
        \label{fig:length 3 to 2}
\end{figure}
\begin{proof}
    The initial setup is pictured on the left of \Cref{fig:length 3 to 2}, and the setup from configuration $\tau$ is pictured on the right. Observe that the right of \Cref{fig:length 3 to 2} is identical to the left side of \Cref{fig:dX=2}, and that $a,b,c,g,h,$ and $i$ are all unchanged by the swaps which obtain $\tau$ from $\si$. Thus, we may apply case $2$ of \Cref{lem:dX=1 or 2} to $\tau$ to swap $(ab)$ from the configuration on the left of \Cref{fig:dX=2} in $\le 70+\max\{3\al,219\}$ moves. Since the sequence of $L$ moves to get from $\si$ to $\tau$ does not involve $a$ or $b$, we may undo this sequence to travel from $\tau\circ(ab)$ to $\si\circ (ab)$ in $\FS$. Using the triangle inequality, \eqref{eq:2si to tau} follows.
\end{proof}
To swap $a$ and $b$, we first reduce to the case where the following structural properties hold:
    \begin{itemize}
        \item[\emph{$(S1)$}] $d'$ is adjacent to a vertex in $ \mc{Q}_a^{*Y}$.
        \item[\emph{$(S2)$}] There exists some $h'\in \mc{Q}_a^{*Y}$ adjacent to $e'$ and $f'$.
        \item[\emph{$(S3)$}] $f'$ is adjacent to at least three vertices in $\mc{Q}_a^{*Y}$. 
    \end{itemize}
    
    \textbf{We reduce to the case where $(S1)$ and $(S2)$ hold.\\} We examine $Y|_W$ where $W = \si(N_X[e]\back\{a,b,i\})$, which has size $|W|\ge \dd(X)-2$. Moreover, $|W| = \dd(X)-2$ only if $e$ is adjacent to $a,b,$ and $i$. Observe that $W$ is not large enough for \Cref{lem:structure} to apply directly, but we use techniques analogous to those in \Cref{lem:structure} to bound $\dd(Y|_W)$ and the structure of the components of $Y|_W$. Since $d_X(\mc{Q}_a^X,e) =2$, we have that $Y|_W$ and $\mc{Q}_a$ are disjoint. Moreover, $a',b',i'\notin Y|_W$ by construction, so swaps internal to $\fs{X|_{\si^{-1}(W)}}{Y|_{W}}$ do not use $\mc{Q}_a^X\cup \{b,i\}$ and thus satisfy the hypotheses of \Cref{claim:ef reduction}. By inclusion-exclusion, $\dd(Y|_W)\ge |W|+\dd(Y)-n$. This implies
    \begin{align}\label{eq:Y|W}
        3\dd(Y|_W) &\ge 3|W|+3\dd(Y)-3n \ge |W|-4.
    \end{align}
    Moreover, by \Cref{lem:c at least 3 neighbors} $e$ is adjacent to at least two vertices in $\mc{S}^X$ since $e$ is nonadjacent to $\mc{Q}_a^X$ but may be adjacent to $a$. Combining this with the fact that $d\neq a,b,i$ and $d$ is adjacent to $e$ implies $|W| = |N_X[e]\back\{a,b,i\}|\ge 4$. If $|W|\ge 5$, then \eqref{eq:Y|W} implies $\dd(Y|_W)\ge 1$, so
    \begin{equation}\label{eq:6Y|W}
        8\dd(Y|_W)\ge 3\dd(Y|_W)+5\ge |W|+1.
    \end{equation}
    Thus, $\dd(Y|_W)>|W|/8$, so by \Cref{lem:diam from delta(X)>n/k}, the diameter of any connected component of $Y|_W$ cannot exceed $20$. If $|W| = 4$, then the diameter of any connected component of $Y|_W$ cannot exceed $4$ since $|V(Y|_W)| = 4$, so $20$ is a valid upper bound on the diameter of a component of $Y|_W$.
    
    \textbf{Suppose $e'$ and $f'$ are in the same connected component of $Y|_W$.} Since the diameter of any connected component is at most $20$, we may move $e$ to $f'$ in at most $20$ moves. By \Cref{claim:ef reduction}, $d_{FS}(\si,\si\circ(ab))\le 110+\max\{3\al,219\}$.

    \textbf{Suppose $e'$ and $f'$ are in different connected components of $Y|_W$.} Each component of $Y|_W$ has at least $\dd(Y|_W)+1 \ge |W|+\dd(Y)-n+1$ vertices. Thus, $d'$ is nonadjacent to all other connected components of $Y|_W$, of which there is at least one. As a result, $d'$ is nonadjacent to at least $|W|+\dd(Y)-n+1\ge \dd(X)+\dd(Y)-n-1$ vertices in some component $C_e$ of $V(Y|_W)$. \eqref{eq:bound on Qa, Qa*} implies $|\mc{Q}_a^*|\ge \dd(X)+\dd(Y)-n+1$, and note that $\mc{Q}_a^Y$ and $Y|_W$ are disjoint since $e$ is nonadjacent to $\mc{Q}_a^X$. We may rearrange $2\dd(X)+3\dd(Y)\ge 3n$ to obtain
    \begin{equation*}
        |\mc{Q}_a^{*Y}\sqcup V(C_e)|\ge 2\dd(X)+2\dd(Y)-2n \ge n-\dd(Y).
    \end{equation*}
    Since $d'$ is nonadjacent to a maximum of $n-1-\dd(Y)$ vertices but $d'$ is nonadjacent to $V(C_e)$, we have that $d'$ is adjacent to a vertex in $\mc{Q}_a^{*Y}$. This proves the reduction to $(S1)$. Since $e'$ and $f'$ are in different connected components, $N_Y[e']\cap N_Y[f']\cap W = \emptyset$. Recall that $\mc{Q}_a^Y$ and $W$ are disjoint and that $|\mc{Q}_a^*|$ is bounded by \eqref{eq:bound on Qa, Qa*}. We use inclusion-exclusion to lower bound $|N_Y[e']\cap N_Y[f']|$. We use inclusion-exclusion on the following expression:
    \begin{align}\label{eq:e' and f' common adjacency}
        &|(N_Y[e']]\cap N_Y[f'])\cap (W\cup \mc{Q}_a^{*Y})|\\
        &\ge \nonumber|N_Y[e']\cap N_Y[f']|+|W\cup \mc{Q}_a^{*Y}|-n\\
        \nonumber&\ge 2\dd(Y)+2-n+\dd(X)-2+\dd(X)+\dd(Y)-n+1-n\\
        \nonumber&\ge 1.
    \end{align}
    Since the above intersection is nonempty and no element of the intersection is in $W$, there exists $h'\in \mc{Q}_a^{*Y}$ adjacent to $e'$ and $f'$. This proves the reduction to $(S2)$. Since by \Cref{lem:ban to S} $\mc{Q}_a^{*Y}$ and $\mc{Q}_b^{*Y}$ are nonadjacent, conditions $(S1)$ and $(S2)$ imply $d',e'\notin \mc{Q}_b^{*Y}$, respectively. Moreover, we may assume that $f'$ is nonadjacent to the component of $Y|_W$ containing $e'$, which has at least $\dd(X)+\dd(Y)-n-1$ vertices.

    \textbf{We now show the reduction to $(S3)$.\\} 
    Since $f'$ is nonadjacent to $e'\notin \mc{Q}_a^{*Y}\cup \mc{Q}_b^{*Y}$, $f'$ is nonadjacent to at most $n-2-\dd(Y)$ vertices in $\mc{Q}_a^{*Y}\cup \mc{Q}_b^{*Y}$. \eqref{eq:Qa+Qb} implies $|\mc{Q}_a^*|+|\mc{Q}_b^*|\ge 2\dd(X)+2\dd(Y)-2n+2$, so $f'$ is adjacent to at least four vertices in $\mc{Q}_a^{*Y}\cup \mc{Q}_b^{*Y}$, and thus either $(S3)$ holds, or $f'$ is adjacent to at least two vertices in $\mc{Q}_b^{*Y}$. 
    
    Since $d_X(f,\mc{Q}_a^X) = 3$, we have $d_X(N_X[f],\mc{Q}_a^X) = 2$, and in particular, $\mc{Q}_a^{*X}$ and $b$ are disjoint from $N_X[f]$. To satisfy the hypotheses of \Cref{claim:ef reduction} with $U=\si(N_X[f])$, we require $i\notin N_X[f]$. If $i=b$, then $i\notin N_X[f]$, and if $i\neq b$, then we show the following claim holds.
    \begin{claim}\label{claim:fi friends}
        In the present setting of \Cref{lem:dX=3}, if $f$ is adjacent to $i\neq b$ and $f'$ is adjacent to some $i_3'\in \mc{Q}_b^{*Y}$ with $i_3'\neq i'$, then $d_{FS}(\si,\si\circ(ab))\le 9$.
    \end{claim}
    \begin{proof}
        We use the labeling introduced at the beginning of \Cref{lem:dX=3}, and let $i_3'\in\mc{Q}_b^{*Y}$ be adjacent to $f'$. The setup is pictured in \Cref{fig:fi friends}.
\begin{figure}[h!]
        \centering
        \begin{tikzpicture}[
    every node/.style={circle, draw, inner sep=1.5pt},
    scale=1
]

\node (a) at (-2,2) {$a$};
\node (b) at ( 2,2) {$b$};
\node (c) at (-2.8,0) {$c$};
\node (g) at ( -2,0) {$g$};
\node (h) at ( -1.2,0) {$h$};
\node (m) at ( 2,0) {$i$};
\node (d) at ( -2.3,-.9) {$d$};
\node (e) at ( -1.4,-1.7) {$e$};
\node (f) at ( 0,-2) {$f$};
\node (i3) at ( 1.2,0) {$i_3$};

\draw[thick] (a) -- (c);
\draw[thick] (a) -- (g);
\draw[thick] (a) -- (h);
\draw[thick] (a) -- (h);
\draw[thick] (b) -- (m);
\draw[thick] (f) -- (g);
\draw[thick] (f) -- (h);
\draw[thick] (f) -- (m);
\draw[thick] (b) -- (i3);
\draw[thick] (f) -- (i3);

\draw[dashed, thick] (a) -- (b);
\draw[dashed, thick] (a) -- (m);
\draw[dashed, thick] (a) -- (i3);
\draw[dashed, thick] (c) -- (b);
\draw[dashed, thick] (g) -- (b);
\draw[dashed, thick] (h) -- (b);
\draw[dashed, thick] (a) to[out=-120,in=80] (c);
\draw[dashed, thick] (a) to[out=-100,in=100] (g);
\draw[dashed, thick] (a) to[out=-60,in=100] (h);
\draw[dashed, thick] (b) to[out=-75,in=75] (m);
\draw[dashed, thick] (b) to[out=-130,in=85] (i3);
\draw[dashed, thick] (d) -- (c);
\draw[dashed, thick] (f) -- (e);
\draw[dashed, thick] (f) to[out=35,in=-115] (m);
\draw[dashed, thick] (e) -- (d);
\end{tikzpicture}
\caption{The case where $(f,f')$ is adjacent to two vertices in $i',i_3'\in \mc{Q}_b^{*Y}$.}
\label{fig:fi friends}
\end{figure}
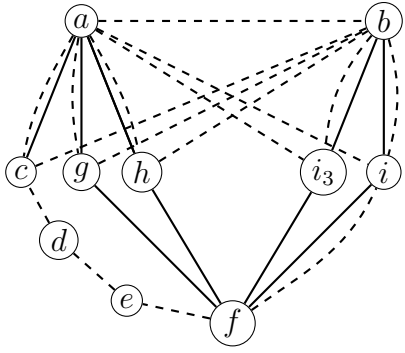
If $i_3\neq b$, we swap 
\[(if)(ah)(ai)(ai_3)(ab)(bi_3)(bi)(bh)(if) =(ab).\] 
If $i_3 = b$, we swap $(if)(ah)(ai)(ab)(bi)(bh)(if) =(ab)$. 
\end{proof}

Thus, if $i$ and $f$ are adjacent, \Cref{claim:fi friends} shows $d_{FS}(\si,\si\circ(ab))\le 9$, since we have assumed $f'$ is adjacent to two vertices in $\mc{Q}_b^{*Y}$.

\textbf{Assume $i\notin N_X[f]$ for the rest of the reduction to $(S3)$.} We use \Cref{lem:structure} with $U=\si(N_X[f])$. We may assume $U$ is disjoint from $\mc{Q}_a^{*Y}\cup \{b,i\}$, and thus the hypothesis of \Cref{claim:ef reduction} is satisfied. We split into the cases $(C1)$, $(C2)$, and $(C3)$.

\textbf{If $(C1)$ holds}, then by \Cref{lem:finite menger} we may find a cycle of length $\le 2D_3$ containing $e'$ and $f'$. By swapping $f$ around in the cycle, we may move $e$ to $f'$ in at most $4D_3^2$ moves. Explicitly, by moving $f$ around the cycle, we may cyclically shift all elements in the cycle 1 time. We repeat this until $e$ is at position $f'$. By \Cref{claim:ef reduction}, we can swap $(ab)$ in at most $70+8D_3^2+\max\{3\al,219\} \le C$ moves.

\textbf{If $(C2)$ holds}, let $F_1$ and $F_2$ be the two connected components of $Y|_U$. By \Cref{lem:structure}, if $e',f'$ were in the same component, we could swap $(ef)$ in at most $15$ moves, and \Cref{claim:ef reduction} implies $d_{FS}(\si,\si\circ(ab))\le 100+\max\{3\al,219\}$.

If $e',f'$ are in separate components, $f'$ is nonadjacent to the $\ge \dd(X)+\dd(Y)-n+2$ vertices in the component $C_e\ni e'$ of $Y|_U$. Using \eqref{eq:bound on Qa, Qa*}, inclusion-exclusion implies 
\begin{align}\label{eq:f' and S3}
    |N_Y[f']\cap (C_e\cup \mc{Q}_a^{*Y})| &\ge 2\dd(X)+3\dd(Y)-3n+4\ge 4.
\end{align}
Thus, $|N_Y[f']\cap \mc{Q}_a^{*Y}|\ge 4$, and $(S3)$ holds.

\textbf{If $(C3)$ holds}, then $Y|_U$ has a single component with a cut vertex $v'$. If $e',f'$ are in the same component, \Cref{lem:structure} implies we can swap $e$ and $f$ in at most $15$ moves, and \Cref{claim:ef reduction} implies $d_{FS}(\si,\si\circ(ab))\le 105+\max\{3\al,219\}.$

If $f'$ is not the cut vertex of $Y|_U$, then $f'$ is nonadjacent to the other component, which contains at least $\dd(X)+\dd(Y)-n+1$ vertices. This bound is smaller than \eqref{eq:f' and S3} by $1$, so $|N_Y[f']\cap \mc{Q}_a^{*Y}|\ge 3$, and $(S3)$ holds.

If $f'$ is the cut vertex of $Y|_U$, let $e'\in V(F_1')$ without loss of generality. We may assume $e'$ and $f'$ are nonadjacent, so $|V(F_1')|\ge |N_Y[e']\cap U|\ge \dd(X)+\dd(Y)-n+2$. If $f'$ is adjacent to only one vertex in $V(F_1')$, then it is nonadjacent to at least $\dd(X)+\dd(Y)-n+1$ vertices in $V(F_1')$. This bound is smaller than \eqref{eq:f' and S3} by $1$, so $|N_Y[f']\cap \mc{Q}_a^{*Y}|\ge 3$, and $(S3)$ holds.

Suppose instead that $f'$ is adjacent to at least two vertices $g_1',g_2'\in V(F_1')$; see \Cref{fig:E and F and cut vertex}. Note $g_1',g_2'\neq e'$ since $e'$ is nonadjacent to $f'$. By \Cref{lem:structure}, we may swap $(fg_1)(fe)^{15}(fg_2)^{15}(fg_1)(fe)$ to move $e$ to $f'$ in $33$ moves. By \Cref{claim:ef reduction}, $d_{FS}(\si,\si\circ(ab))\le 136+\max\{3\al,219\}\le C$. Thus, we reduce to the case where $(S3)$ holds.
 \begin{figure}[h!]
        \centering
        \begin{tikzpicture}[
    vertex/.style={circle, draw, inner sep=1.5pt},
    scale=1
]
\node[vertex] (a) at (0,2) {$f$};
\node[vertex] (c) at (-1,1.5) {$g_1$};
\node[vertex] (d) at (-1,0.5) {$g_2$};
\node[vertex] (l) at (1,1.5) {};
\node[vertex] (b) at (-2,0) {$e$};
\draw (-1.5,0.5) ellipse [x radius=1.3cm, y radius=1.6cm];
\draw (1.5,0.5) ellipse [x radius=1.3cm, y radius=1.6cm];

\node at (-2.2,1) {$F_1'$};
\node at (2.2,1) {$F_2'$};

\draw[thick] (a) -- (c);
\draw[thick] (a) -- (d);
\draw[thick] (a) -- (l);

\end{tikzpicture}
        \caption{The case where $f'$ is the cut vertex and is adjacent to $g_1',g_2'\in V(F_1')$.}
        \label{fig:E and F and cut vertex}
    \end{figure}
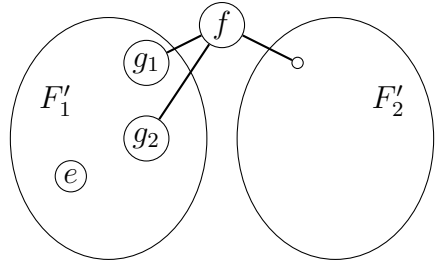

\textbf{We assume $(S1),(S2),$ and $(S3)$ hold for the remainder of the proof.} $(S1)$ implies $d'$ is adjacent to $j'\in \mc{Q}_a^{*Y}$, and $(S2)$ implies that $e'$ and $f'$ are both adjacent to $k'\in \mc{Q}_a^{*Y}$. Furthermore, $(S3)$ implies $f'$ is adjacent to two vertices in $\mc{Q}_a^{*Y}$ distinct from $k'$, which we label $g'$ and $h'$ for consistency with previous diagrams in this lemma. By \Cref{lem:c at least 3 neighbors}, $d$ is adjacent to some $\ell,m\in\mc{Q}_a^{*X}$ distinct from $c$. The setup is depicted on the left in \Cref{fig:dX=3}. We first solve the case where all vertices depicted in \Cref{fig:dX=3} are distinct, and then we will carefully handle the case where they are not all distinct.
 \begin{figure}[h!]
        \centering
        
        \begin{tikzpicture}[
    every node/.style={circle, draw, inner sep=1.5pt},
    scale=1.2
]

\node (a) at (-2,2) {$a$};
\node (b) at ( 3,2) {$b$};
\node (c) at (-2.8,0) {$c$};
\node (g) at ( -2,0) {$j$};
\node (h) at ( -1.2,0) {$k$};
\node (i) at ( -0.4,0) {$\ell$};
\node (j) at ( 0.4,0) {$m$};
\node (k) at ( 1.2,0) {$g$};
\node (l) at ( 2,0) {$h$};
\node (m) at ( 3,0) {$i$};
\node (d) at ( -2.5,-1.5) {$d$};
\node (e) at ( -1,-2) {$e$};
\node (f) at ( 0.5,-2) {$f$};

\draw[thick] (a) -- (c);
\draw[thick] (a) -- (g);
\draw[thick] (a) -- (h);
\draw[thick] (d) -- (g);
\draw[thick] (a) -- (h);
\draw[thick] (b) -- (m);
\draw[thick] (a) -- (i);
\draw[thick] (a) -- (j);
\draw[thick] (a) -- (k);
\draw[thick] (f) -- (m);
\draw[thick] (f) -- (k);
\draw[thick] (f) -- (h);
\draw[thick] (e) -- (h);
\draw[thick] (a) -- (l);
\draw[thick] (f) -- (l);

\draw[dashed, thick] (a) -- (b);
\draw[dashed, thick] (a) to[out=-120,in=80] (c);
\draw[dashed, thick] (a) to[out=-100,in=100] (g);
\draw[dashed, thick] (a) to[out=-60,in=100] (h);
\draw[dashed, thick] (a) to[out=-50,in=120] (i);
\draw[dashed, thick] (a) to[out=-35,in=128] (j);
\draw[dashed, thick] (b) to[out=-75,in=75] (m);
\draw[dashed, thick] (a) to[out=-27,in=138.5] (k);
\draw[dashed, thick] (a) to[out=-22,in=140] (l);
\draw[dashed, thick] (d) -- (c);
\draw[dashed, thick] (f) -- (e);
\draw[dashed, thick] (e) -- (d);
\draw[dashed, thick] (i) -- (d);
\draw[dashed, thick] (j) -- (d);

\node (a') at (5.8,2) {$a$};
\node (b') at ( 8.2,2) {$b$};
\node (c') at (5.8,0) {$\ell$};
\node (d') at ( 8.2,0) {$e$};
\node (e') at ( 7,-1.4) {$d$};

\draw[thick] (a') -- (d');
\draw[thick] (b') -- (d');
\draw[thick] (c') -- (d');

\draw[dashed, thick] (a') -- (b');
\draw[dashed, thick] (a') -- (c');
\draw[dashed, thick] (c') -- (b');
\draw[dashed, thick] (d') -- (e');
\draw[dashed, thick] (c') -- (e');

\node[draw=none, circle=none] at (4.4,0)
{$\xrightarrow[]{\text{after swaps}}$};

\end{tikzpicture}
\caption{The case $d_X = 3$. Note that $b$ is adjacent to $c,g,h,i,j,k,\ell$ and $a$ is adjacent to $m$, but we omit these edges for clarity. Left: initial setup. Right: the setup after the preliminary swaps, restricted to houses $g',h',i',f',$ and $e'$.}
        \label{fig:dX=3}
\end{figure}
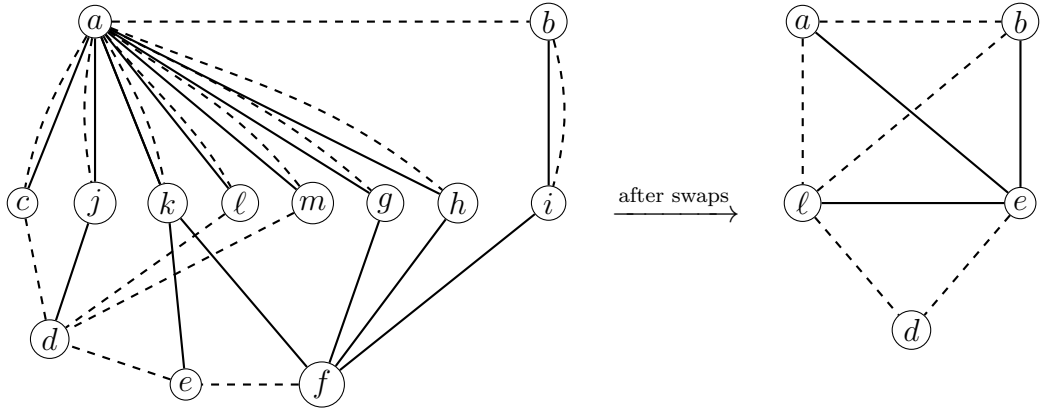
We may now perform a sequence of swaps. We use \Cref{fig:cd swap Qa} to perform swaps inside $\mc{Q}_a^*$ when $a$ is in $a'$. We swap
\begin{equation}\label{eq:big swap}
    (cj)^{33}(k\ell)^{33}(am)(cd)(dm)(d\ell)(de)(ef)(a\ell)(\ell g)^{33}(ah)(bi).
\end{equation}
After these $108$ moves, we arrive at the setup on the right side of \Cref{fig:dX=3}. By \Cref{lem:bangachev with P=1} we may swap $a$ and $b$ in $\max\{3\al,219\}$ moves. After swapping, we undo the beginning sequence of moves by leveraging \Cref{lem:swap cd in qa} to swap vertices in $\mc{Q}_a$ in $33$ moves:
\[(bh)(ai)(\ell g)^{33}(b\ell)(ef)(de)(d\ell)(dm)(cd)(bm)(k\ell)^{33}(cj)^{33}.\]
In the configuration where we swap $(\ell g)$, both $\ell$ and $g$ are in $\mc{Q}_a^X$ defined in this configuration. By \Cref{lem:swap cd in qa}, either $(\ell g)^{33}$ is valid from this configuration and from this configuration with $a$ and $b$ swapped, or $a$ and $b$ can be swapped in $\max\{\al,21\}$ moves. If the latter is true, we swap $a$ and $b$ and proceed to undo all previous moves, which can be done since $a$ and $b$ are both adjacent to all of $\mc{Q}_a^X\cup \mc{Q}_b^X$. This would require strictly fewer moves, so we have swapped $(ab)$ in at most $ 216+\max\{3\al,219\}$ moves when all vertices in \Cref{fig:dX=3} are distinct.

Finally, we justify this construction if not all vertices on the left of \Cref{fig:dX=3} are distinct. If $h=a$, we omit $(ah)$ and $(bh)$. We may choose $g\neq a$ since $g$ and $h$ arise from the definition of $\mc{S}$. If $k\neq a$, then $k,g,h\in \mc{Q}_a$ are distinct. If $j= k$, we omit the swaps $(dm)(d\ell)$ in the initial sequence and $(d\ell)(dm)$ in the ending sequence. If $j=a$, then we use the following sequence
\begin{equation}\label{eq:j=a seq}
    (k\ell)^{33}(ac)(cd)(d\ell)(a\ell)(de)(ef)(\ell g)^{33}(ah)(bi),
\end{equation}
call on \Cref{lem:bangachev with P=1}, and use the corresponding concluding sequence
\[(bh)(ai)(\ell g)^{33}(ef)(de)(b\ell)(d\ell)(cd)(bc)(k\ell)^{33}.\]
This is a total of $148+\max\{3\al,219\}$ moves. If $a=m$, there is no change, and since $c,\ell,$ and $m$ were given by \Cref{lem:c at least 3 neighbors}, we may assume $c,\ell\neq a$. The only vertices which could be coincident (in our case $j=a$) are some combination of $c,\ell,m$ being the same as $k,g,h$. These cases require no new techniques, and the sequence \eqref{eq:j=a seq} works. If $j,k\neq a$, then swaps \eqref{eq:big swap} work.

The outstanding case is $k=a$. If all other vertices are distinct, we swap
\begin{equation}\label{eq:k=a}
    (cj)^{33}(\ell g)^{33}(mh)^{33}(a\ell)(cd)(d\ell)(de)(ef)(bi).
\end{equation}
This is $108$ moves, so we again require a total of at most $216+\max\{3\al,219\}$ moves after applying the corresponding ending sequence. If $k=j=a$, then we swap
\begin{equation}\label{eq:j=k=a}
    (\ell h)^{33}(cg)^{33}(ac)(cd)(de)(ef)
\end{equation}
to arrive at the setup required for \Cref{lem:bangachev with P=1}. Including the corresponding ending sequence, this is a total of $140+\max\{3\al,219\}$ moves. Other coincident vertices are subsumed in previous cases, or follow from \eqref{eq:k=a} and \eqref{eq:j=k=a}. In general, the coincidences of $c,\ell,$ or $m$ are minor as they only contribute to edges of $X$, so problems may be swapped with \Cref{lem:swap cd in qa} or ignored on a case by case basis.


Since at the beginning of this lemma we reduced to the case where $i\neq e,d$ in $12$ moves, we may need to undo these $12$ moves at the end of our setup. In all cases, $d_{FS}(\si,\si\circ(ab))\le 228+8D_3^2+\max\{3\al,219\} = C$, as desired.
\end{proof}

To conclude the proof, we tackle the only remaining case: when all paths of minimal length from $\mc{Q}_a^X$ to $\mc{S}^X$ include $a$ or $b$. We will split this lemma into cases based on whether the shortest path has length $2$ or $3$.

\begin{lemma}\label{lem:ab must be in the path}
    In the present setting, suppose all paths of minimal length from $\mc{Q}_a^X$ to $\mc{S}^X$ include either $a$ or $b$. Then
    \[d_{FS}(\si,\si\circ(ab))\le 16+\max\{\al,21\}.\]
\end{lemma}

\begin{proof}
    \textbf{Case 1: The minimal path has length 2.}
    Either $c-a-d$ or $c-b-d$ is a path from $\mc{Q}_a^X$ to $\mc{S}^X$. We show that the following procedure works for both $a$ and $b$. Since $d'\in \mc{S}^Y$, there exist $e'\in\mc{Q}_a^{Y}$ and $f'\in\mc{Q}_b^{*Y}$ adjacent to $d'$. By \Cref{lem:ban to S}, $\mc{Q}_a^{*Y}\ni e'$ and $\mc{Q}_b^{*Y}\ni f'$ are nonadjacent, so $|N_Y(e')\cap N_Y(f')|\ge 2\dd(Y)-n+2$, and $N_Y(e')\cap N_Y(f')$ is disjoint from $\mc{Q}_a^{*Y}\cup \mc{Q}_b^{*Y}$. Recall that $d$ is nonadjacent to $\mc{Q}_a$ since $d_X(\mc{Q}_a^X,\mc{S}^X)=2$. By inclusion-exclusion and the lower bound \eqref{eq:bound on Qa, Qa*} on $|\mc{Q}_a|$, 
    \begin{align}\label{eq:last lemma, def}
        |\si(N_X(d))\cap N_Y(e')\cap N_Y(f')| &\ge \dd(X)+2\dd(Y)-2n+2+|\mc{Q}_a|\ge 2.
    \end{align}
    Let $d_1',d_2'$ be in this intersection, and let $c,a,d$ be the minimum length path. We arrive at the setup in the left side of \Cref{fig:a special case}. We perform the swaps:
\[(ae)(bf)(ad)(ab)(dd_1)(ad)(dd_2)(dd_1)(ad).\]
After performing these swaps, we find the setup on the right side of \Cref{fig:a special case}. 
     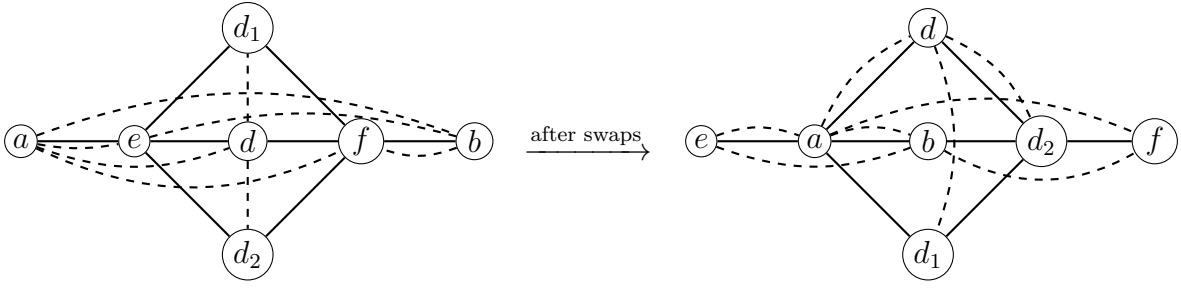
\begin{figure}[h!]
        \centering
        
        \begin{tikzpicture}[
    every node/.style={circle, draw, inner sep=1.5pt},
    scale=1
]

\node (a) at (-3,0) {$a$};
\node (b) at ( 3,0) {$b$};
\node (e) at (-1.5,0) {$e$};
\node (f) at ( 1.5,0) {$f$};
\node (d) at (0,0) {$d$};
\node (d1) at ( 0,1.5) {$d_1$};
\node (d2) at ( 0,-1.5) {$d_2$};

\draw[thick] (e) -- (d);
\draw[thick] (f) -- (d);
\draw[thick] (e) -- (d1);
\draw[thick] (f) -- (d1);
\draw[thick] (e) -- (d2);
\draw[thick] (f) -- (d2);
\draw[thick] (a) -- (e);
\draw[thick] (b) -- (f);

\draw[dashed, thick] (a) to[out=20, in=160] (b);
\draw[dashed, thick] (a) to[out=-10, in=-170] (e);
\draw[dashed, thick] (a) to[out=-20, in=-160] (d);
\draw[dashed, thick] (a) to[out=-25, in=-155] (f);
\draw[dashed, thick] (b) to[out=-160, in=-20] (f);
\draw[dashed, thick] (b) to[out=165, in=15] (e);
\draw[dashed, thick] (d) -- (d1);
\draw[dashed, thick] (d) -- (d2);

\node (a') at (6,0) {$e$};
\node (b') at (12,0) {$f$};
\node (e') at (7.5,0) {$a$};
\node (f') at ( 10.5,0) {$d_2$};
\node (d') at (9,0) {$b$};
\node (d1') at (9,1.5) {$d$};
\node (d2') at (9,-1.5) {$d_1$};

\draw[thick] (e') -- (d');
\draw[thick] (f') -- (d');
\draw[thick] (e') -- (d1');
\draw[thick] (f') -- (d1');
\draw[thick] (e') -- (d2');
\draw[thick] (f') -- (d2');
\draw[thick] (a') -- (e');
\draw[thick] (b') -- (f');

\draw[dashed, thick] (e') to[out=20, in=160] (d');
\draw[dashed, thick] (e') to[out=65, in=-155] (d1');
\draw[dashed, thick] (a') to[out=-20, in=-160] (d');
\draw[dashed, thick] (e') to[out=23, in=157] (b');
\draw[dashed, thick] (a') to[out=20, in=160] (e');
\draw[dashed, thick] (d') to[out=-30, in=-150] (b');
\draw[dashed, thick] (d1') to[out=-25, in=115] (f');
\draw[dashed, thick] (d1') to[out=-70, in=70] (d2');

\node[draw=none, circle=none] at (4.5,0)
{$\xrightarrow[]{\text{after swaps}}$};

\end{tikzpicture}
        \caption{The case where $\{a,d\}\in E(X)$ for $d\in \mc{S}^X$.}
        \label{fig:a special case}
\end{figure}
We wish to swap $d$ and $d_1$; they are currently in pairs $(d,d_1'),(d_1,d_2')$. Let $\si'$ denote the configuration on the right side of \Cref{fig:a special case}. Since $d$ and $d_1$ are adjacent, we call on our earlier reduction \Cref{lem:Ra} to swap them. As in \eqref{eq:QaQbRaRb}, we may define the analog of $\mc{R}_a$ for $(d,d_1')$ and $(d_1,d_2')$ in the setting $\si'$:
\begin{equation}\label{eq:Rd}
    \mc{R}_d = \{(r,\si'(r)):r\in N_X(d),r\notin N_X(d_1),\si'(r)\in N_Y(d_1'),\si'(r)\in N_Y(d_2')\}.
\end{equation}
We apply the same reductions from the beginning of this section in an attempt to swap $d$ and $d_1$ from $\si'$; in particular, we may assume no pair lies in all four neighborhoods (the analog of \eqref{eq:adjacencies of c,c'}), otherwise we could swap $d$ and $d_1$ in $3$ moves. After this reduction, $(a,e'),(d_2,f')\in\mc{R}_d$, so \Cref{lem:Ra} implies $d_{FS}(\si',\si'\circ(dd_1))\le \max\{\al,21\}$. Afterwards, we finish swapping $(ab)$ with the sequence:
\[(dd_2)(ab)(ad)(be)(af).\]
We demonstrate that in this case, $d_{FS}(\si,\si\circ(ab))\le 14+\max\{\al,21\}$. See that in the subcase where $f=b$, we may ignore $f$ and perform the swaps:
\[(ae)(ad)(ab)(dd_1)(ad)(dd_2)(dd_1)(ad)(dd_1)^{\max\{\al,21\}}(dd_2)(ab)(ad)(be) = (ab).\]
This swaps $a$ and $b$ in $12+\max\{\al,21\}$ moves. The case where $\{b,d\}\in E(X)$ is symmetric. \\
        




\textbf{Case 2: The minimal path has length 3.}
Note that $a,b\in N[\mc{Q}_a^X]$, so if the path begins at $\mc{Q}_a^X$ and ends at $\mc{S}^X$, then $a$ or $b$ may only be the second vertex in this path. Let the path be given by $c-a-d-e$ with $c\in \mc{Q}_a^X$ and $e\in \mc{S}^X$, as the case $c-b-d-e$ is symmetric. By \Cref{lem:size of S}, $|\mc{S}|\ge 2\dd(Y)-n+1$. See that $|\si(N_X(d))\cap N_Y(d')|\ge \dd(X)+\dd(Y)-n+1$ by inclusion-exclusion, and because neither set contains $d'$. Since all minimal length paths from $\mc{Q}_a^X$ to $\mc{S}^X$ use either $a$ or $b$, 
\begin{equation}\label{eq:NQaX cap NXf}
    N[\mc{Q}_a^X]\cap N_X(d)\subset \{a,b\}.
\end{equation}
Moreover, if $a',b'\in \si(N_X(d))\cap N_Y(d')$, then we have the setup in \Cref{fig:4 conditions}, and $d_{FS}(\si,\si\circ(ab))\le 3$. Thus, we may assume $|\si(N_X(d))\cap N_Y(d')\cap N[\mc{Q}_a^X]|\le 1$. Recall that $|\mc{S}^X\cap N[\mc{Q}_a^X]|=0$ since $d_X(\mc{Q}_a^X,S^X) = 3$, and that $|N[\mc{Q}_a^X]|\ge |N_X[c]|\ge\dd(X)+1$. We use these three inequalities and inclusion-exclusion to lower bound $|\si(N_X(d))\cap N_Y(d')\cap \mc{S}^Y|$:
\begin{align}\label{eq:f(dd') and S}
    |\si(N_X(d))\cap N_Y(d')\cap \mc{S}^Y| &\ge \dd(X)+3\dd(Y)-3n+1+|N[\mc{Q}_a^X]|\ge 2.
\end{align}
Since the above intersection is nonempty, let $f'\in \si(N_X(d))\cap N_Y(d')\cap \mc{S}^Y$. Note that we may swap $(df)$, and let $\si' = \si\circ(df)$.

We reduce to the case where $d,f\notin N_X(a)\cap N_X(b)$. Suppose that $f\in N_X(a)\cap N_X(b)$. Let $g'\in\mc{Q}_a^Y$ and $h'\in\mc{Q}_b^{*Y}$ be adjacent to $f'$, since $f'\in\mc{S}^Y$; the setup is pictured in \Cref{fig:f is a friend of both a and b}. The swaps $(ag)(af)(ah)(ab)(bh)(bf)(bg)$ demonstrate $d_{FS}(\si,\si\circ(ab))\le 7$ here. If $h=b$, then the swaps $(ag)(af)(ab)(bf)(bg)$ suffice, and show the same bound. 
    \begin{figure}[h!]
        \centering
        
        \begin{tikzpicture}[
    every node/.style={circle, draw, inner sep=1.5pt},
    scale=1.2
]

\node (a) at (-3,1) {$a$};
\node (b) at ( 3,1) {$b$};
\node (e) at (-1.5,0) {$g$};
\node (f) at ( 1.5,0) {$h$};
\node (d) at (0,-1) {$f$};

\draw[thick] (e) -- (d);
\draw[thick] (f) -- (d);
\draw[thick] (a) -- (e);
\draw[thick] (b) -- (f);

\draw[dashed, thick] (a) -- (b);
\draw[dashed, thick] (a) to[out=-55, in=170] (e);
\draw[dashed, thick] (a) to[out=-65, in=-180] (d);
\draw[dashed, thick] (b) to[out=-115, in=0] (d);
\draw[dashed, thick] (a) --(f);
\draw[dashed, thick] (b) to[out=-125, in=10] (f);
\draw[dashed, thick] (b) -- (e);

\end{tikzpicture}
        \caption{The case where $f\in N_X(a)\cap N_X(b)$.}
        \label{fig:f is a friend of both a and b}
\end{figure}
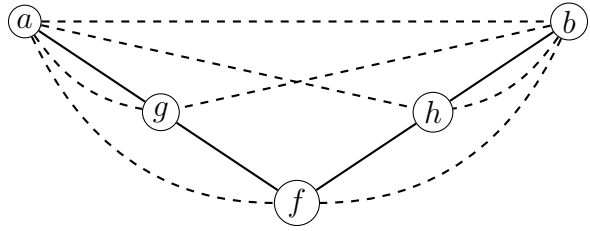
Similarly, if $d\in N_X(a)\cap N_X(b)$, then we first perform the swap $(df)$, apply the same procedure to swap $(ab)^7$, and then swap $(df)$ to conclude the sequence. Thus, $d_{FS}(\si,\si\circ(ab))\le 9$ here. 

\textbf{We may assume $d,f\notin N_X(a)\cap N_X(b)$.} Thus, $d,f\notin \mc{Q}_a^{*X}\cup\mc{Q}_b^{*X}$. Consequently, $\mc{Q}_a,\mc{Q}_b,$ and $\mc{S}^Y$ are invariant under $(df)$. Moreover, $(d,\si'(d)) = (d,f')\in \mc{S}_{\si'}$ (in words, $d$ is now in $\mc{S}_{\si'}^X$, whereas $f$ used to be in $\mc{S}^X$). Thus, we may apply Case 1 under this setup, since $c-a-d$ is a path from $\mc{Q}_{a,\si'}^X$ to $\mc{S}_{\si'}^X$. We may swap $(ab)$ in $\le 14+\max\{\al,21\}$ moves, and then perform $(df)$. Thus,
\[d_{FS}(\si,\si\circ(ab))\le 16+\max\{\al,21\}.\]
The case $c-b-d-e$ is symmetric and the lemma is proven.
\end{proof}

We may now prove that any two friends can be swapped in $C$ moves.

\begin{proof}[Proof of \Cref{thm:2x+3y in finite time}]
    We prove that $d_{FS}(\si,\si\circ(ab))\le C$. We use structural lemmas from \Cref{lem:d(X)le 3} to \Cref{lem:c at least 3 neighbors} reduce the problem. By \Cref{lem:c at least 3 neighbors}, $d_X(\mc{Q}_a^X,\mc{S}^X) \le 3$. If there exists a path from $\mc{Q}_a^X$ to $\mc{S}^X$ of minimal length that does not include $a$ or $b$, \Cref{lem:dX=1 or 2} and \Cref{lem:dX=3} resolve the cases $d_X(\mc{Q}_a^X,\mc{S}^X)\le2$ and $d_X(\mc{Q}_a^X,\mc{S}^X) = 3$, respectively. If instead all paths from $\mc{Q}_a^X$ to $\mc{S}^X$ of minimal length include $a$ or $b$, \Cref{lem:ab must be in the path} shows $d_{FS}(\si,\si\circ(ab))\le C$. Thus, $d_{FS}(\si,\si\circ(ab))\le C$ in all cases.
\end{proof}

The proof of \Cref{thm:main} follows immediately from \Cref{thm:2x+3y in finite time}:

\begin{proof}[Proof of \Cref{thm:main}]
    By \Cref{thm:2x+3y in finite time} and \Cref{lem:friend reduction}, $\FS$ is connected and 
    \[\diam(\FS)\le \frac{6n^2C}{\dd(X)}.\]
\end{proof}

This readily implies \Cref{cor:diam}.

\begin{proof}[Proof of \Cref{cor:diam}]
    For any connected friends and strangers graph, 
\[\diam(\FS)\ge \diam(\fs{K_n}{K_n}) = n-1\] 
because $\FS$ is a spanning subgraph of $\fs{K_n}{K_n}$. Thus, $\diam(\FS) = \Omega(n)$. When $\min\{\dd(X),\dd(Y)\} = \Omega(n)$, \Cref{thm:main} implies $\diam(\FS) = O(n)$, and \Cref{cor:diam} follows.
\end{proof}
\section{Diameter Bounds}

This section is devoted to proving \Cref{prop:diam lower bound} and \Cref{cor:diam of 2x+3y component is n2}. \Cref{prop:diam lower bound} is a direct consequence of the following proposition.

\begin{proposition}\label{prop:diam lower construction}
    Let $n$ and $k$ be positive integers with $n\ge 2k+1$. Let $X=(\mathrm{Cycle}_n)^k$ and $Y=K_n$. Then $\dd(X) = 2k,\dd(Y) = n-1$, and 
    \[\diam (\FS)\ge \frac{n^2-n}{2\min\{\dd(X),\dd(Y)\}}.\]
    If $n$ is even,
    \[\diam (\FS)\ge \frac{n^2}{2\min\{\dd(X),\dd(Y)\}}.\]
\end{proposition}

\begin{proof}
    Let $V(X)$ and $V(Y)$ be indexed by $[n]$. We examine the distance between the identity permutation $\si$ and the permutation $\tau(m) = m+\lfloor\frac{n}{2}\rfloor \mod n$. Define the total displacement between $\si,\tau$:
    \begin{align*}
        d_{\Sigma}(\si,\tau) &:= \sum_{v\in V(X)}d_X(v,\si^{-1}\circ\tau(v)).
    \end{align*}
    Each vertex in $V(X)$ moves a distance $\lfloor n/2\rfloor$ in $\mathrm{Cycle}_n$, it moves at least a distance $\lfloor n/2\rfloor/k$ in $X$. Total displacement is subadditive for a sequence of swaps, and each swap of two people in $X$ may contribute at most $2$ to the total displacement. Thus,
    \begin{align*}
        d_{FS}(\si,\tau) &\ge \frac{n}{2k}\cdot \left\lfloor \frac{n}{2}\right\rfloor \ge  \frac{n^2-n}{2\min\{\dd(X),\dd(Y)\}}.
    \end{align*}
    Moreover, $d_{FS}(\si,\tau)\ge n^2/4k = n^2/(2\min\{\dd(X),\dd(Y)\})$ for even $n$.
\end{proof}

The above proposition readily implies \Cref{prop:diam lower bound}. This yields a lower bound of $\Omega(n^2/\min\{\dd(X),\dd(Y)\})$ on the maximum diameter of a friends and strangers graph constructed from $X,Y$ on $n$ vertices with 
\[2\min\{\dd(X),\dd(Y)\}+3\max\{\dd(X),\dd(Y)\}\ge 3n.\]
Following the proof of \Cref{thm:main}, it is natural to ask if similar diameter bounds exist for graphs $X$ and $Y$ satisfying $2\dd(X)+3\dd(Y)\ge 3n$, but where $X$ and $Y$ are not necessarily connected. We prove \Cref{cor:diam of 2x+3y component is n2} by induction on $n$. Note that $X,Y$ need not be connected. We restate the theorem for convenience, and we let $C$ be the constant obtained in \Cref{thm:2x+3y in finite time}. Note that for all friends and strangers graphs in the \textcolor{dg}{green} region of \Cref{fig:connectivity diagram}, we have $\diam(\FS)\le Cn^2$ since \Cref{thm:main} holds with constant $C$.

\newtheorem*{cor:diam of 2x+3y component is n2}{Corollary \ref{cor:diam of 2x+3y component is n2}}
\begin{cor:diam of 2x+3y component is n2}
Let $X,Y$ be graphs on $n$ vertices such that: 
    \[2\min\{\dd(X),\dd(Y)\}+3\max\{\dd(X),\dd(Y)\}\ge 3n.\]
    Then the diameter of any connected component of $\FS$ is $O(n^2)$.
\end{cor:diam of 2x+3y component is n2}

\begin{proof}
    For $n=1$, the theorem holds trivially, as $|V(\FS)| = 1\le C$. Fix $n\in \N$, and suppose the theorem holds for all $m<n$. If $X,Y$ are connected, then we are done by \Cref{thm:main}. By \Cref{lem:XY=YX} let $\dd(X)\le \dd(Y)$, and thus $\dd(Y)\ge 3n/5$. Since $\dd(Y)>n/2$, we have that $Y$ is connected. Thus, we may suppose $X$ is not connected, and has components 
    \[X = X_1\sqcup \cdots \sqcup X_r.\]
Observe that $\si\in \FS$ must lie in the (possibly disconnected) subgraph of $\FS$ corresponding to the ordered set partition $(\si(X_1),\dots,\si(X_r))$. \Cref{prop:components} shows that this subgraph is isomorphic to $\fs{X_1}{Y|_{\si(X_1)}}\square \cdots \square \fs{X_r}{Y|_{\si(X_r)}}$. Thus, the diameter of any component of this subgraph is at most the sum of maximum diameters of connected components of $\fs{X_i}{Y|_{\si(X_i)}}$. We examine each subgraph $\fs{X_i}{Y|_{\si(X_i)}}$. Since each $X_i$ is a connected component of $X$, we have $\dd(X_i)\ge \dd(X)$ and $\dd(Y|_{\si(X_i)})\ge \dd(Y)-n+|V(X_i)|$ for all $i$. We find that:
    \begin{align*}
        2\dd(X_i)+ 3\dd(Y|_{\si(X_i)}) &\ge 2\dd(X)+3\dd(Y)-3n+3|V(X_i)|\ge 3|V(X_i)|.
    \end{align*}
    Regardless of which subgraph has the smaller minimum degree, we have the following:
    \begin{align*}
        &2\min\{\dd(X_i),\dd(Y|_{\si(X_i)})\}+3\max\{\dd(X_i),\dd(Y|_{\si(X_i)})\}\\
        &\ge 2\dd(X_i)+3\dd(Y|_{\si(X_i)})\\
        &\ge 3|V(X_i)|.
    \end{align*}
    It follows that the diameter of the component of $\FS$ containing $\si$ is bounded by the sum of the maximum diameter of connected components $\fs{X_i}{Y|_{\si(X_i)}}$. By induction, the diameter of any connected component of $\fs{X_i}{Y|_{\si(X_i)}}$ is at most $C |X_i|^2$, so the diameter of our connected component of $\FS$ is at most
    \begin{align*}
        C\sum_{i=1}^r|V(X_i)|^2\le C|V(X)|^2
    \end{align*}
    by superadditivity of $f(x) = x^2$. Thus, the theorem is proven by induction.
\end{proof}

\section{Insufficient Minimum Degrees for Connectivity}

This section will introduce weighted cliques, which we use to construct disconnected friends and strangers graphs outside of the \textcolor{red}{red} region of \Cref{fig:connectivity diagram}. We prove \Cref{prop:weighted clique pair}, which allows us to exhibit a family of disconnected $\FS$ with minimum degrees in the \textcolor{orange}{orange} region.
\begin{definition}\label{def:imply FS connected}
    We say a pair $(\al,\be)\in [0,1]^2$ \blue{implies $\FS$ is connected} if for any pair of connected graphs $X$ and $Y$ on $n$ vertices with $\dd(X)/n\ge \al$ and $\dd(Y)/n\ge \be$, $\FS$ is connected.
\end{definition}
One natural way to disconnect a friends and strangers graph is using  cliques, possibly of different sizes. Let $w:[n]\to [0,1]$ be a weight function with total weight $1$, so $\sum_{i\in [n]} w(i) = 1$. The weight of a subset $S\subset [n]$ is $w(S) = \sum_{i\in S} w(i)$. Let $G$ be a graph with vertex set $V(G) = [n]$ such that $G$ and its complement $G^C$ are connected. Define the \blue{closed weighted minimum degree} to be
\begin{equation*}
    \dd^+_w[G]:=\min_{v\in V(G)}w(N_G[v]).
\end{equation*}
We call $(\dd^+_w[G],\dd^+_w[G^C])$ a \blue{weighted clique pair} if it arises from a rational weight function $w:[k]\to (0,1]\cap \Q$ with total weight $1$ and a graph $G$ on $k$ vertices such that $G$ and $G^C$ are connected. We create cliques with sizes proportional to $w$, and let two cliques be adjacent in $X$ corresponding to edges in $G$, and be adjacent in $Y$ corresponding to edges in $G^C$. This yields a disconnected friends and strangers graph; see \Cref{fig:7 weights}.

\begin{proposition}\label{prop:construct disconnected}
    Let $k\ge 2$ be an integer, and let $w:[k]\to (0,1]\cap \Q$ be a rational weight function with total weight $1$. Let $n$ be a positive integer such that $nw(i)\in \Z$ for all $i\in[k]$, and let $G$ be a graph on $k$ vertices such that $G$ and $G^C$ are connected. Then for all positive integers $m$, there exist graphs $X$ and $Y$ on $mn$ vertices such that $\dd(X) = \dd^+_w[G]mn-1$ and $\dd(Y) = \dd^+_w[G^C]mn-1$, and such that $\FS$ is disconnected.
\end{proposition}

\begin{proof}
    Let $\mc P = \{P_1,\dots,P_k\}$ be a partition of $[mn]$ such that $|P_i| = mnw(i)$ for all $i\in[k]$. Let $X$ and $Y$ both have vertex set $[mn]$, where $a\in P_i$ and $b\in P_j$ are adjacent in $X$ if $i=j$ or if $i$ and $j$ are adjacent in $G$, and $a$ and $b$ are adjacent in $Y$ if $i=j$ or if $i$ and $j$ are nonadjacent in $G$. The degree in $X$ of a vertex $v\in P_i$ is
    \begin{align*}
        \deg_X(v) = \sum_{\substack{i=j\text{ or}\\ \{i, j\}\in E(G)}} |P_j| -1 = w(N_G[v])mn-1.
    \end{align*}
    Thus, $\dd(X) = \dd^+_w[G]mn-1$, and $\dd(Y) = \dd^+_w[G^C]mn-1$ by a symmetric argument. Finally, we show $\FS$ is disconnected by considering the identity vertex, $\si:[mn]\to [mn]$. Any $a\in P_i\subset V(X)$ is adjacent to all other vertices in $P_i$, as well as those $P_j$ such that $\{i,j\}\in E(G)$. However, $\si(a)\in P_i\subset V(Y)$ is adjacent to $P_i$ and all $P_j$ with $\{i,j\}\notin E(G)$. Thus, only swaps with both vertices in the same partition block can occur, since the structure of $X$ and $Y$ is unchanged under swaps inside any partition block. Since no vertex can leave its partition block, $\FS$ is disconnected. Moreover, since $G$ and $G^C$ are connected graphs and $X$ and $Y$ are obtained from $G$ and $G^C$ respectively by blowing up vertices into cliques, $X$ and $Y$ are connected.
\end{proof}

\begin{corollary}\label{cor:insuf under pairs}
    If $(\dd^+_w[G],\dd^+_w[G^C])$ is a weighted clique pair, then all pairs $(\al,\be)$ with $\al<\dd^+_w[G]$ and $\be<\dd^+_w[G^C]$ do not imply $\FS$ is connected.
\end{corollary}

\begin{proof}
    By \Cref{prop:construct disconnected} there exist disconnected $\FS$ with $\dd(X)/|X| = \dd_w^+[G]-1/|X|$ and $\dd(Y)/|Y| = \dd_w^+[G^C]-1/|Y|$ for arbitrarily large $|X|=|Y|=n$. Pick $1/n< \min\{\dd_w^+[G]-\al,\dd_w^+[G^C]-\be\}$ to show $(\al,\be)$ is insufficient to imply $\FS$ is connected.
\end{proof}

\begin{figure}[h!]
\centering
\begin{tikzpicture}[
    scale=1,
    every node/.style={font=\large},
    smallv/.style={circle, draw, fill=blue!15, minimum size=12mm, inner sep=0pt},
    medv/.style={circle, draw, fill=blue!15, minimum size=14mm, inner sep=0pt},
    bigv/.style={circle, draw, fill=blue!15, minimum size=20mm, inner sep=0pt},
    gedge/.style={solid, thick},
    cedge/.style={dashed, thick, color = blue}
]


\node[bigv]   (v1) at ( 3.5, 0.0) {$1$};
\node[medv]   (v2) at ( 2.2, 2.4) {$2$};
\node[medv] (v3) at (-0.3, 3.3) {$3$};
\node[medv] (v4) at (-3.0, 1.9) {$4$};
\node[smallv]   (v5) at (-3.0,-1.3) {$5$};
\node[smallv] (v6) at (-0.7,-3.1) {$6$};
\node[smallv]   (v7) at ( 1.7,-2.6) {$7$};


\draw[gedge] (v1) -- (v5);
\draw[gedge] (v1) -- (v6);
\draw[gedge] (v1) -- (v7);

\draw[gedge] (v2) -- (v3);
\draw[gedge] (v2) -- (v4);
\draw[gedge] (v3) -- (v4);

\draw[gedge] (v2) -- (v7);

\draw[gedge] (v3) -- (v6);

\draw[gedge] (v4) -- (v5);


\draw[cedge] (v1) -- (v2);
\draw[cedge] (v1) -- (v3);
\draw[cedge] (v1) -- (v4);

\draw[cedge] (v2) -- (v5);
\draw[cedge] (v2) -- (v6);

\draw[cedge] (v3) -- (v7);
\draw[cedge] (v3) -- (v5);

\draw[cedge] (v4) -- (v6);
\draw[cedge] (v4) -- (v7);

\draw[cedge] (v5) -- (v6);
\draw[cedge] (v6) -- (v7);
\draw[cedge] (v5) -- (v7);


\node[draw=none,font=\tiny] at (5.5,0)
    {$w(1)=\dfrac{34}{130}$};

\node[draw=none,font=\tiny] at (3.85,2.5)
    {$w(2)=\dfrac{17}{130}$};

\node[draw=none,font=\tiny] at (1.3,3.6) 
    {$w(3)=\dfrac{17}{130}$};

\node[draw=none,font=\tiny] at (-4.6,2.0) 
    {$w(4)=\dfrac{17}{130}$};

\node[draw=none,font=\tiny] at (-4.5,-1.3)
    {$w(5)=\dfrac{15}{130}$};

\node[draw=none,font=\tiny] at (3.15,-2.6)
    {$w(7)=\dfrac{15}{130}$};

\node[draw=none,font=\tiny] at (-2.2,-3.1)
    {$w(6)=\dfrac{15}{130}$};



\end{tikzpicture}
\caption{A weighted vertex graph $G$ (solid edges) and its complement $G^C$ (dashed edges).}
\label{fig:7 weights}
\end{figure}
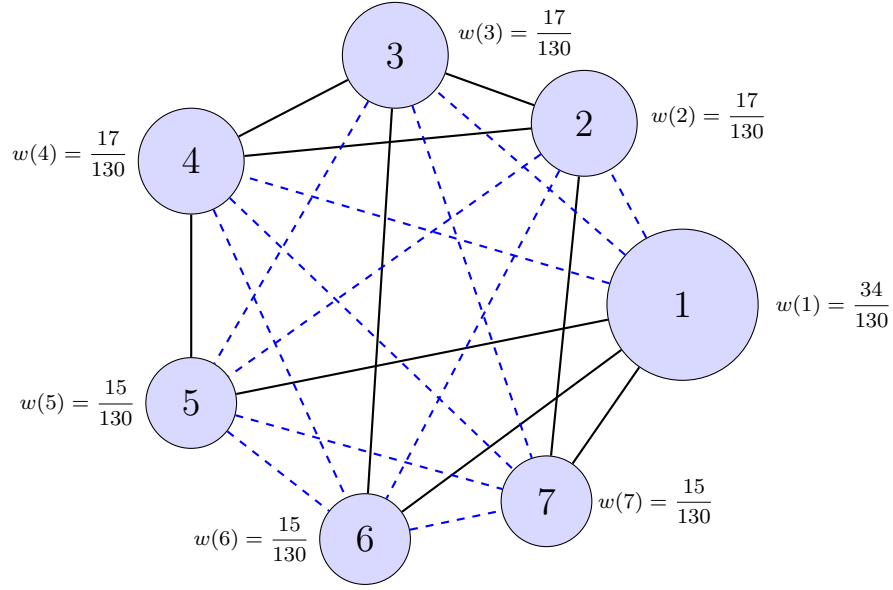

Surprisingly, we may use \Cref{prop:construct disconnected} to find pairs outside the \textcolor{red}{red} region that do not imply $\FS$ is connected, extending the results of \cite[Proposition 1.6]{Bangachev}. Computing large weighted clique pairs can be done with a linear program so long as $G$ is fixed. ChatGPT is responsible for the following example that serves as proof of concept, but many more examples can be found with a standard linear program.

\begin{example}\label{ex:GPT}
    Apply the construction in \Cref{prop:construct disconnected} with $k=7$ and $w(1) = 2A$, $w(2)=w(3)=w(4) = A$, and $w(5)=w(6)=w(7) = B$ for constants $A$ and $B$; see \Cref{fig:7 weights}. We find that 
    \begin{align*}
        \dd^+_w[G] = \min\{3A+B,2A+3B\},\quad \dd^+_w[G^C] = \min\{2A+3B,5A\}.
    \end{align*}
    The relevant regime is $B\le A\le 2B$, so $\dd^+_w[G] = 3A+B$ and $\dd^+_w[G^C] = 2A+3B$. This yields the lines between $(1/2,5/8)$ and $(7/13,7/13)$ and between $(5/8,1/2)$ and $(7/13,7/13)$, which define the \textcolor{orange}{orange} region of \Cref{fig:connectivity diagram}. These pairs only exit the \textcolor{red}{red} region when $B\le A\le 6B/5$.
    
    For an explicit example, take $n=130, A=17/130$, and $B=15/130$. Then $(\dd^+_w[G],\dd^+_w[G^C]) = (66/130,79/130)$. Since $66/130>1/2$ and $79/130>3/5$, by \Cref{cor:insuf under pairs}, this yields an insufficient pair above the \textcolor{red}{red} staircase, but in the \textcolor{orange}{orange} region. This example is depicted in \Cref{fig:7 weights}.
\end{example}

This family of examples implies that there exist pairs in the \textcolor{orange}{orange} region and outside the \textcolor{red}{red} region that do not imply $\FS$ is connected. Moreover, the previously known examples of Bangachev \cite[Proposition 1.6]{Bangachev} arise when $G = \mathrm{Cycle}_k$ and $w = 1/k$ for all vertices. See that the combination of \Cref{prop:construct disconnected}, \Cref{cor:insuf under pairs}, and \Cref{ex:GPT} readily imply \Cref{prop:weighted clique pair}.

\section{A Simple Proof in the Blue Regime}

Connectivity in the \textcolor{blue}{blue} region was first proven by Bangachev \cite[Theorem 1.5]{Bangachev} using Wilson's results on Star graphs \cite{Wilson}. The only notable previous diameter bound in the \textcolor{blue}{blue} region is $O(n^5)$ and is due to Akella and Li \cite[Theorem 1.3]{FSDiameter}. We give a simple proof that $\FS$ is connected and has $O(n^2/\min\{\dd(X),\dd(Y)\})$ diameter in the \textcolor{blue}{blue} regime of \Cref{fig:connectivity diagram}. Though the result of \Cref{thm:main} is much stronger and also holds in the \textcolor{dg}{green} region, we include this proof for its simplicity, and because it allows us to extract an explicit and small constant, $42$, as the coefficient of the $O(n^2/\min\{\dd(X),\dd(Y)\})$ bound. Notably, this proof does not require Wilson's results, and it shows that any two friends can be swapped in $7$ moves.  
\begin{proposition}\label{prop:x+2y diam}
    Let $X$ and $Y$ be connected graphs on $n$ vertices that satisfy:
    \[\min\{\dd(X),\dd(Y)\}+2\max\{\dd(X),\dd(Y)\}\ge 2n.\]
    Then
    \[\diam(\FS) \le  \frac{42n^2}{\min\{\dd(X),\dd(Y)\}}.\]
\end{proposition}

We will prove first the following lemma, and then the proposition will follow easily.
\begin{lemma}\label{lem:x+2y finite}
    Let $X,Y$ be connected graphs on $n$ vertices with $\dd(X)\le \dd(Y)$ and
    \[\min\{\dd(X),\dd(Y)\}+2\max\{\dd(X),\dd(Y)\}\ge 2n.\]
    Then for all adjacent $a,b\in V(X)$ and all $\si\in V(\FS)$,
    \[d_{FS}(\si,\si\circ(ab))\le 7.\]
\end{lemma}
\begin{proof}
    Let $a,b\in V(X)$ be adjacent, and let $\si\in V(\FS)$. If $a'$ and $b'$ are adjacent, the swap $(ab)$ is valid and $d_{FS}(\si,\si\circ(ab))=1$. Thus, we may assume $a'$ and $b'$ are nonadjacent.

    \textbf{Case 1: $a,b$ have a mutual friend.}\\
    Suppose there exists $c\in V(X)$ adjacent to $a$ and $b$. Since $a'$ and $b'$ are nonadjacent, $a'$ and $b'$ each have $\ge \dd(Y)-1$ neighbors in $V(Y)\back \{a',b',c'\}$. Moreover, $c$ has $\ge \dd(X)-2$ neighbors in $V(X)\back \{a,b,c\}$. By inclusion-exclusion,
    \begin{align*}
        \left|\big(\si(N_X(c))\cap N_Y(a')\cap N_Y(b')\big)\back\{a',b',c'\}\right| &\ge \dd(X)-2+2\dd(Y)-2-2(n-3)\ge 2.
    \end{align*}
    Since this intersection is nonempty, let $d'\in \si(N_X(c))\cap N_Y(a')\cap N_Y(b')$. If $c'$ and $d'$ are adjacent, we have the setup on the left of \Cref{fig:CD in EY}. The sequence $(cd)(ac)(ab)(bc)(cd)$ shows that $d_{FS}(\si,\si\circ(ab))\le 5$ in this case.
    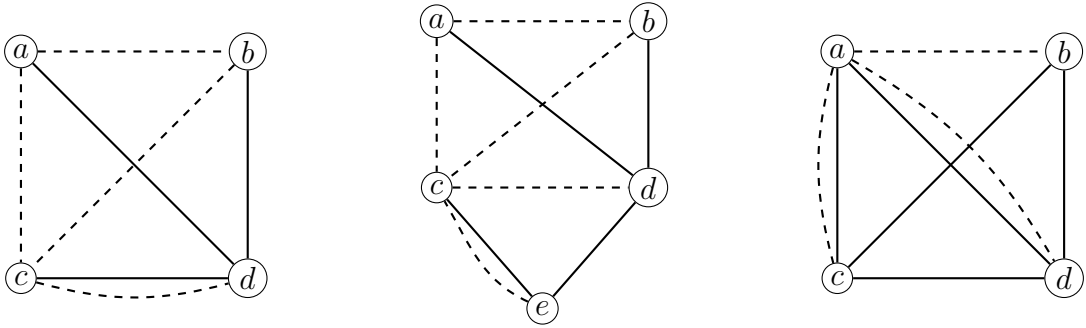
\begin{figure}[h!]
        \centering
        \begin{tikzpicture}[
    every node/.style={circle, draw, inner sep=1.5pt},
    scale=1
]

\node (a) at (-1.5,1.5) {$a$};
\node (b) at ( 1.5,1.5) {$b$};
\node (c) at (-1.5,-1.5) {$c$};
\node (d) at ( 1.5,-1.5) {$d$};

\draw[thick] (a) -- (d);
\draw[thick] (b) -- (d);
\draw[thick] (c) -- (d);

\draw[dashed, thick] (a) -- (b);

\draw[dashed, thick] (a) -- (c);
\draw[dashed, thick] (c) to[out=-15,in=195] (d);
\draw[dashed, thick] (b) -- (c);

\node (a') at (4,1.9) {$a$};
\node (b') at (6.8,1.9) {$b$};
\node (c') at (4,-0.3) {$c$};
\node (d') at (6.8,-0.3) {$d$};
\node (e') at (5.4,-1.9) {$e$};

\draw[thick] (a') -- (d');
\draw[thick] (b') -- (d');
\draw[thick] (c') -- (e');
\draw[thick] (d') -- (e');

\draw[dashed, thick] (a') -- (b');
\draw[dashed, thick] (a') -- (c');
\draw[dashed, thick] (c') -- (d');
\draw[dashed, thick] (b') -- (c');
\draw[dashed, thick] (e') to[out=155,in=-60] (c');

\node (a'') at (9.3,1.5) {$a$};
\node (b'') at (12.3,1.5) {$b$};
\node (c'') at (9.3,-1.5) {$c$};
\node (d'') at (12.3,-1.5) {$d$};

\draw[thick] (a'') -- (c'');
\draw[thick] (b'') -- (c'');
\draw[thick] (a'') -- (d'');
\draw[thick] (b'') -- (d'');
\draw[thick] (c'') -- (d'');

\draw[dashed, thick] (a'') -- (b'');

\draw[dashed, thick] (a'') to[out=-105,in=105] (c'');
\draw[dashed, thick] (a'') to[out=-30,in=120] (d'');

\end{tikzpicture}
        \caption{Swapping $a$ and $b$ in various setups. Left: the case $\{c',d'\}\in E(Y)$. Middle: the construction of $d'$ and $e'$. Right: $d'\in N_Y(c')\cap \mc{R}_a^Y$.}
        \label{fig:CD in EY}
    \end{figure}
    We may assume $c'$ and $d'$ are nonadjacent, so $c$ has $\ge \dd(X)-3$ neighbors in $V(X)\back \{a,b,c,d\}$, and $c'$ and $d'$ have $\ge\dd(Y)-2$ neighbors in $V(Y)\back \{a',b',c',d'\}$. By inclusion-exclusion:
    \begin{align*}
        \left|\big(\si(N_X(c))\cap N_Y(c')\cap N_Y(d')\big)\back\{a',b',c',d'\}\right| &\ge \dd(X)-3+2\dd(Y)-4-2(n-4)\ge 1.
    \end{align*}
    Since this intersection is nonempty, let $e'\in \si(N_X(c))\cap N_Y(c')\cap N_Y(d')$. The current setup is pictured in the middle of \Cref{fig:CD in EY}, and the sequence $(ce)(cd)(ac)(ab)(bc)(cd)(ce)$ ensures $d_{FS}(\si,\si\circ(ab))\le 7$. 

    \textbf{Case 2: $a,b$ have no mutual friend.}
    We suppose instead that $N_X(a),N_X(b)$ are disjoint (recall that $a\notin N_X(a),b\notin N_X(b)$). Then $\left|\big(N_X(a)\cup N_X(b)\big)\back \{a,b\}\right|\ge 2\dd(X)-2$. Since $a'$ and $b'$ are nonadjacent, $a'$ and $b'$ each have $\ge \dd(Y)$ neighbors in $V(Y)\back \{a',b'\}$, so
    \begin{align*}
        \left|\big(\si(N_X(a)\cup N_X(b))\cap N_Y(a') \cap N_Y(b')\big)\back\{a',b'\}\right|&\ge 2\dd(X)-2+2\dd(Y)-2(n-2)\\
        &\ge \dd(X)+2.
    \end{align*}
    Passing to our notation from \eqref{eq:QaQbRaRb}, we find that the set is exactly $\mc{R}_a^Y\sqcup \mc{R}_b^Y$. We have just shown $|\mc{R}_a|+|\mc{R}_b|\ge \dd(X)+2$, so let $|\mc{R}_a|\ge \frac12\dd(X)+1$ without loss of generality. Let $(c,c')\in \mc{R}_a$. By inclusion-exclusion,
    \begin{align*}
        |N_Y(c')\cap \mc{R}_a^Y|&\ge \dd(Y)+\frac12\dd(X)+1-n \ge 1.
    \end{align*}
    Thus, we may pick $d'\in N_Y(c')\cap \mc{R}_a^Y$. This yields the setup on the right in \Cref{fig:CD in EY}, and the sequence $(ac)(ab)(ad)(ac)(ab)(ac)(ad)$ guarantees $d_{FS}(\si,\si\circ(ab))\le 7$.
\end{proof}

\begin{proof}[Proof of \Cref{prop:x+2y diam}]
    By \Cref{lem:x+2y finite}, we may apply \Cref{lem:friend reduction} with $t=7$. \Cref{prop:x+2y diam} follows immediately since $\FS\cong \fs{Y}{X}$ by \Cref{lem:XY=YX}. 
\end{proof}

\section{Future Directions}

The main result of this paper is \Cref{thm:main}, which shows that friends and strangers graphs are connected in the \textcolor{dg}{green} region of \Cref{fig:connectivity diagram 2} (we use the \textcolor{dg}{green} region to refer to all points above the line $2\min+3\max = 3n$; this includes the \textcolor{dg!60!black}{dark green} and \textcolor{blue}{blue} regions of \Cref{fig:connectivity diagram 2}). Previously, friends and strangers graphs were only known to be connected in the \textcolor{dg!60!black}{dark green} and \textcolor{blue}{blue} regions \cite{Bangachev}, so the existing work in the \textcolor{dg}{green} required the hypothesis $\min\{\dd(X),\dd(Y)\}>n/2$. We have removed this condition. 

\Cref{thm:main} establishes that the diameter of $\FS$ in the \textcolor{dg}{green} region is $O(n^2/\min)$. There were no previous diameter bounds anywhere in the \textcolor{dg}{green} region, and the only known diameter bound was a bound of $O(n^5)$ in the \textcolor{blue}{blue} region, due to Akella and Li \cite{FSDiameter}.

Moreover, the \textcolor{dg}{green} region is the complement of the convex hull of the \textcolor{red}{red} staircase, so \Cref{thm:main} is the best linear hypothesis on $\dd(X),\dd(Y)$ that implies $\FS$ is connected. One compelling future direction is to characterize the points of \Cref{fig:connectivity diagram 2} that imply $\FS$ is connected, which would resolve the unknown white region between the \textcolor{red}{red} and \textcolor{dg}{green} regions. \Cref{prop:weighted clique pair} demonstrates that pairs in the \textcolor{orange}{orange} region do not imply $\FS$ is connected. We provide \Cref{fig:connectivity diagram 2} here for convenience.

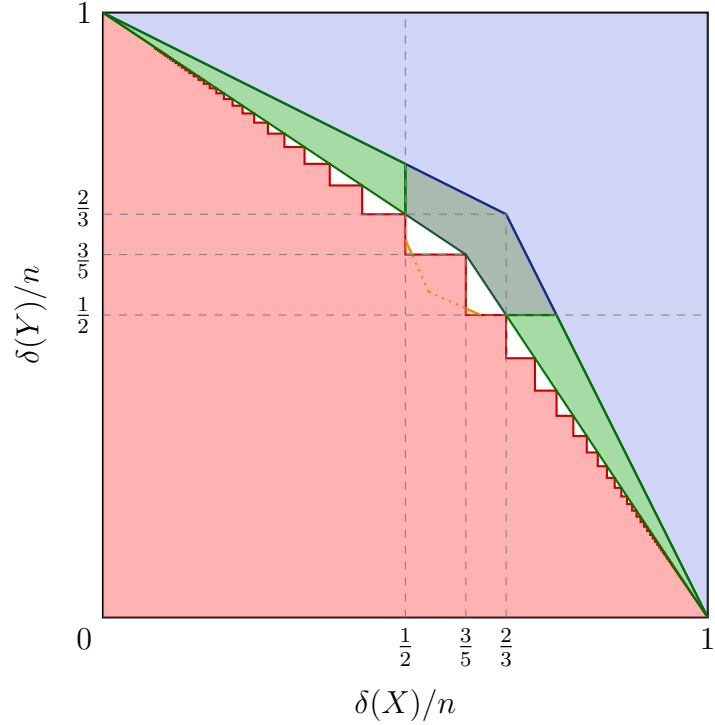
\begin{figure}
\centering
\begin{tikzpicture}[scale=8]

\draw[thick] (0,0) rectangle (1,1);


\def\RedPath{(1,0)}

\foreach \k in {34,33,...,5}{
  \pgfmathtruncatemacro{\kp}{\k+1}

  \xdef\RedPath{\RedPath
    -- ({(\k-2)/\k},{3/\kp})
    -- ({(\k-2)/\k},{3/\k})}
}

\foreach \k in {5,...,34}{
  \pgfmathtruncatemacro{\kp}{\k+1}

  \xdef\RedPath{\RedPath
    -- ({3/\kp},{(\k-2)/\k})
    -- ({3/\kp},{(\kp-2)/\kp})}
}

\fill[red!55,opacity=.55]
  (0,0) -- \RedPath -- (0,1) -- cycle;

\draw[ thick,red!80!black]
  \RedPath -- (0,1);

\fill[dg!30!black,opacity=.35]
  (1/2,2/3) -- (1/2,3/4) -- (2/3,2/3) -- (3/4,1/2) -- (2/3,1/2)
  -- (3/5,3/5)
  -- cycle;

\draw[ thick,dg!45!black]
  (1/2,2/3) -- (1/2,3/4) -- (2/3,2/3) -- (3/4,1/2) -- (2/3,1/2)
  -- (3/5,3/5)
  -- cycle;

\fill[db!60,opacity=.35]
  (0,1) -- (1,1) -- (1,0)
  -- (2/3,2/3)
  -- cycle;

\draw[ thick,db!70!black]
  (0,1) -- (2/3,2/3) -- (1,0);

\fill[green!60!black,opacity=.35]
  (0,1) -- (1/2,3/4) -- (1/2,2/3)
  -- cycle;

\fill[green!60!black,opacity=.35]
  (1,0) -- (3/4,1/2) -- (2/3,1/2)
  -- cycle;

\draw[ thick,green!45!black]
  (0,1) -- (1/2,3/4) -- (1/2,2/3)
  -- cycle;

\draw[ thick,green!45!black]
  (1,0) -- (3/4,1/2) -- (2/3,1/2)
  -- cycle;

\draw[thick,orange]
(1/2,5/8) -- (23/45,3/5);
\draw[thick,orange]
(5/8,1/2) -- (3/5,23/45);
\draw[dotted,thick,orange]
(7/13,7/13) -- (23/45,3/5);
\draw[dotted,thick,orange]
(7/13,7/13) -- (3/5,23/45);
\fill[orange, opacity = .35]
(1/2,5/8) -- (23/45,3/5) -- (1/2,3/5) -- cycle;
\fill[orange, opacity = .35]
(5/8,1/2) -- (3/5,23/45) -- (3/5,1/2) -- cycle;

\node[below] at (0.5,-0.1) {$\delta(X)/n$};
\node[rotate=90] at (-0.12,0.5) {$\delta(Y)/n$};

\node[below left] at (0,0) {$0$};
\node[below] at (1,0) {$1$};
\node[left] at (0,1) {$1$};

\node[below] at (3/5,0) {$\frac35$};
\node[left,font=\small] at (0,3/5-.01) {$\frac35$};

\node[below] at (2/3,0) {$\frac23$};
\node[left,font=\small] at (0,2/3+.01) {$\frac23$};

\node[below] at (1/2,0) {$\frac12$};
\node[left,font=\small] at (0,1/2) {$\frac12$};

\draw[dashed,gray] (3/5,0) -- (3/5,3/5) -- (0,3/5);
\draw[dashed,gray] (2/3,0) -- (2/3,2/3) -- (0,2/3);
\draw[dashed,gray] (1/2,0) -- (1/2,1);
\draw[dashed,gray] (0,1/2) -- (1,1/2);

\end{tikzpicture}
    \caption{Connectivity of $\FS$ for connected $X,Y$ in terms of the minimum degrees $\dd(X)$ and $\dd(Y)$. The graph $\FS$ is connected in the \textcolor{dg}{green}, \textcolor{dg!60!black}{dark green}, and \textcolor{blue}{blue} regions, however, $\FS$ may be disconnected in the \textcolor{red}{red} and \textcolor{orange}{orange} regions. It is unknown if $\FS$ in the white region are always connected.}
    \label{fig:connectivity diagram 2}
\end{figure}

\begin{question}\label{q:classify}
    Classify all pairs that imply $\FS$ is connected. 
\end{question}

Answering this question is equivalent to completely filling in \Cref{fig:connectivity diagram 2}. We conjecture that the diameter bound from \Cref{thm:main} holds for any pair $(\dd(X)/n,\dd(Y)/n)$ sufficient to imply $\FS$ is connected.

\begin{conjecture}\label{conj:conn implies O(n2/min)}
    If a pair implies $\FS$ is connected, then it also implies 
    \[\diam(\FS)=O(n^2/\min\{\dd(X),\dd(Y)\}).\]
\end{conjecture}

The diameter bound in \Cref{thm:main} gives a partial answer to the following question of Defant and Kravitz \cite{Defant}:

\begin{question}[{\cite[Subsection 7.3]{Defant}}]
    Does there exist an absolute constant $C$ such that if $\FS$ is connected, then $\diam(\FS) = O(n^C)$?
\end{question}

This is one of the biggest open problems regarding friends and strangers graphs. In a similar vein to \Cref{q:classify}, we ask if weighted cliques are the `best' way to disconnect friends and strangers graphs with large minimum degrees. Recall \Cref{def:imply FS connected}.

\begin{question}\label{q:weighted pairs exactly}
    Fix $(\al,\be)\in [0,1]^2$. If there does not exist a weighted clique pair such that $\al<\dd^+_w[G]$ and $\be<\dd^+_w[G^C]$, is it true that $(\al,\be)$ implies $\FS$ is connected?
\end{question}

If true, the answer to this question would imply that weighted cliques determine the true boundary for the connectivity of $\FS$ in terms of minimum degrees.

\subsection{Questions on the Green Regime.} It is natural to ask under what conditions the diameter bound of \Cref{thm:main} could be improved to $\Theta(n^2/\min\{\dd(X),\dd(Y)\})$. The lower bound is not $\Omega (n^2/\min\{\dd(X),\dd(Y)\})$ in general; for example, $\diam(\fs{\st_n}{K_n}) = O(n)$ despite the fact that $\dd(\st_n) = 1$. This can be seen because any transposition can be performed in $3$ moves. We ask the following question.

\begin{question}\label{q:regular}
    Let $X,Y$ be connected, regular graphs. What condition on $\dd(X),\dd(Y)$ will guarantee that $\FS$ is connected and
    \[\diam(\FS) = \Theta\left(\frac{n^2}{\min\{\dd(X),\dd(Y)\}}\right)?\]
\end{question}

We may examine the lack of a constant on the right hand side of the minimum degree condition in \Cref{thm:main}. Maximal examples arising from \Cref{prop:bangachev counterexamples} satisfy
\[2\min\{\dd(X),\dd(Y)\}+3\max\{\dd(X),\dd(Y)\}=3n-5.\]
Let $n\ge 4$ be even, and let $M$ be a perfect matching of $[n]$. If we let $X= \mathrm{Cycle}_n$ and $Y = K_n\sm M$, then $\dd(X) = 2$ and $\dd(Y) = n-2$, so $2\min\{\dd(X),\dd(Y)\}+3\max\{\dd(X),\dd(Y)\}=3n-2$. Moreover, by \cite[Corollary 4.12]{Defant}, since $Y^C$ is a forest of size $2$ trees, $\fs{\mathrm{Cycle}_n}{Y}$ has $2$ components, and is disconnected. We ask what constant is sharp in \Cref{thm:main}, which is equivalent to asking whether $c=0$ or $c=1$ in the following question.
\begin{question}\label{q:3n-4}
    What is the largest value of $c$ such that if $X$ and $Y$ are connected graphs on $n$ vertices and 
    \[2\min\{\dd(X),\dd(Y)\}+3\max\{\dd(X),\dd(Y)\}\ge 3n-c,\]
    then \Cref{thm:main} holds?
\end{question}
In light of \Cref{thm:2x+3y in finite time}, we ask for more specifics about the constant $C$, which is an upper bound on the number of moves needed to swap two friends (for $\FS$ in the \textcolor{dg}{green} region of \Cref{fig:connectivity diagram 2}).
\begin{question}
    Find an explicit constant $C$ such that \Cref{thm:2x+3y in finite time} holds. What value of $C$ is sharp?
\end{question}
Our last question regards the $k$-connectivity of friends and strangers graphs. This subject has been previously studied \cite{KConn}, but \Cref{thm:main} may yield improved results on $k$-connectivity. We ask the following question.
\begin{question}\label{q:kConn}
    Let $k$ be a positive integer. Suppose $X$ and $Y$ are connected graphs on $n$ vertices. For which functions $f:\Z\to \Z$ does the condition
    \[2\min\{\dd(X),\dd(Y)\}+3\max\{\dd(X),\dd(Y)\}\ge 3n+f(k)\]
    imply that $\FS$ is $k$-connected?
\end{question}
The author expects that there exists a polynomial function $f$ that satisfies this property, as this would match existing results on $k$-connectivity in smaller regimes \cite{KConn}. There are many other open questions on friends and strangers graphs, which can be found in \cite{Defant} or \cite{Jeong}.

\section*{Acknowledgements}
This work was conducted as part of Undergraduate Research at the University of Minnesota Duluth REU during the summer of 2026. This REU was supported by Jane Street Capital, Axiom Math AI, and a donation from Larry Penn. The author would like to thank the program for its generous support and for providing a good working environment. The author thanks Claire Kaneshiro, Carl Schildkraut, Glenn Bruda, Evan Chen, Noam Pasman, Colin Defant, Mark Saengrongkongka, Noah Kravitz, and Eliot Hodges for their helpful insights, comments, and conversations. 

\section*{AI Statement}
ChatGPT contributed very little; its only contributions were finding \Cref{ex:GPT}, which can also be found using a linear program, finding the example preceding \Cref{q:3n-4}, generating some of the diagrams, and suggesting minor edits. 

\printbibliography

\end{document}